\documentclass[12pt, letterpaper]{amsart}

\newcommand{\la}{\langle}
\newcommand{\ra}{\rangle}

\usepackage{amssymb}
\usepackage{amsthm}
\usepackage{amsxtra}
\usepackage{graphicx}
\usepackage{color}
\usepackage{xcolor}
\usepackage{mathrsfs}

\newtheorem{theorem}{Theorem}

\newtheorem{proposition}[theorem]{Proposition}
\newtheorem{lemma}[theorem]{Lemma}

\newtheorem{assumption}[theorem]{Assumption}

\theoremstyle{remark}
\newtheorem{remark}[theorem]{Remark}

\numberwithin{equation}{section}

\numberwithin{theorem}{section}

\numberwithin{table}{section}

\numberwithin{figure}{section}

\ifx\pdfoutput\undefined
  \DeclareGraphicsExtensions{.pstex, .eps}
\else
  \ifx\pdfoutput\relax
    \DeclareGraphicsExtensions{.pstex, .eps}
  \else
    \ifnum\pdfoutput>0
      \DeclareGraphicsExtensions{.pdf}
    \else
      \DeclareGraphicsExtensions{.pstex, .eps}
    \fi
  \fi
\fi

\title[Equilibria of the 3D Axisymmetric RVM System]{Continuous Family of Equilibria of the 3D Axisymmetric Relativistic Vlasov-Maxwell System}

\author{Katherine Zhiyuan Zhang}
\address{Courant Institute of Mathematical Sciences, New York University}

\begin{document}

%{\color{red}{Get rid of the even $K$}}

%{\color{red}{Structure of solutions set, general domain that contains the $z$-axis}}

%{\color{red}{$f^{0, \pm} = \mu^\pm ( \la v \ra, p^\pm, v_z ) $ or $f^{0, \pm} = \mu^\pm ( \la v \ra, v_z ) $, $\phi^0 \equiv 0$, quasineutral.}}

\begin{abstract}
We consider the relativistic Vlasov-Maxwell system (RVM) on a general axisymmetric spatial domain with perfect conducting boundary which reflects particles specularly, assuming axisymmetry in the problem. We construct continuous global parametric solution sets for the time-independent RVM. The solutions in these sets have arbitrarily large electromagnetic field and the particle density functions have the form $f^\pm = \mu^\pm (e^\pm (x, v), p^\pm (x, v))$, where $e^\pm$ and $p^\pm$ are the particle energy and angular momentum, respectively. In particular, for a certain class of examples, we show that the spectral stability changes as the parameter varies from $0$ to $\infty$. 
%We also construct global steady solution sets for an RVM system with a external magnetic field. 
% and a local steady solution set for RVM, in which the particle density functions do not necessarily have the form $f^\pm = \mu^\pm (e^\pm (x, v), p^\pm (x, v))$.
\end{abstract}

\maketitle

\section{Introduction}
\label{S:Intro}

We consider the 3D relativistic Vlasov-Maxwell (RVM) system on an axisymmetric bounded domain $\Omega \subset \mathbb{R}^3$. We study a plasma with two species (ion and electron) with non-negative distribution functions $f^\pm(t, x, v)$, where $t \geq 0$, $x \in \Omega \subset \mathbb{R}^3$, $v \in \mathbb{R}^3$.
%{\color{red}{$\Omega$ must be simply connected? }}
The system RVM is
\begin{equation} \label{Vlasovtwospecies}
\partial_t f^{\pm} + \hat{v} \cdot \nabla_x f^{\pm} \pm (\textbf{E} + \hat{v} \times \textbf{B}   ) \cdot \nabla_v f^{\pm} =0 ,
\end{equation}
\begin{equation}  \label{Maxwelltwospecies-1}
\nabla_x \cdot \textbf{E} = \rho = \int_{\mathbb{R}^3} (f^+ - f^-)dv, \  \nabla_x \cdot \textbf{B} =0 ,
\end{equation}
\begin{equation}  \label{Maxwelltwospecies-2}
\partial_t \textbf{E} -\nabla_x \times \textbf{B} = -\textbf{j} = -  \int_{\mathbb{R}^3} \hat{v} (f^+ - f^-)dv, \  \partial_t \textbf{B} + \nabla_x \times \textbf{E} =0 .
\end{equation}
In this system, $f^\pm (t, x, v) \geq 0$ is the density distribution of ions ($+$) and electrons ($-$). We consider the plasma in a region $ \Omega \subset \mathbb{R}^3$, so that $x \in \Omega$ is the particle position. $v \in \mathbb{R}^3$ is the particle momentum, $\la v\ra = \sqrt{1+ v^2}$ is the particle kinetic energy, and $\hat{v} = v/ \la v \ra$ is the particle velocity. Also, $\textbf{E}$ is the electric field, $\textbf{B}$ is the internal magnetic field, so $ \pm (\textbf{E} + \hat{v} \times \textbf{B} ) $ is the total (internal) electromagnetic force. Moreover, the charge density $\rho$ and the current density $\textbf{j}$ are defined as
\begin{equation}
\rho = \int_{\mathbb{R}^3} (f^+ - f^-) dv , \qquad \textbf{j} =  \int_{\mathbb{R}^3} \hat{v} (f^+ - f^-) dv  .
\end{equation}
A key direction of study is to construct the stationary solutions (equilibria) of the RVM system on a bounded domain with varied boundary conditions. 

In this paper, at the boundary we impose the specular condition (which means that $f^\pm$ is even with respect to  $v_n= v \cdot e_n $ on $\partial \Omega$, with $e_n$ being the outward normal vector of $\partial \Omega$ at $x$): 
\begin{equation}
f^{\pm} (t,x,v) = f^{\pm} (t,x,v-2(v \cdot e_n (x))e_n(x) ), e_n(x) \cdot v <0,  \  \forall x \in \partial \Omega  ,
\end{equation}
as well as the perfect conductor boundary condition
\begin{equation} 
\textbf{E}(t,x) \times n(x) =0,  \ \textbf{B} (t,x) \cdot n(x) =0, \ \forall x \in \partial\Omega  .
\end{equation}
We assume axisymmetry in the problem. Therefore it is natural to use cylindrical coordinates here. 

%Investigation for time-independent RVM on a domain with physical boundaries has been carried out over the years. Rein \cite{Rein1} studied stationary solutions of the RVM by variational methods. {\color{blue}{Dolbeault \cite{Dol1} ...}} Batt \& Fabian \cite{BattFabian}, Braasch \cite{Braasch1}, Poupaud \cite{Poupaud1}, Sinitsyn \& Dulov \cite{SD1} constructed stationary solutions for boundary value problems making use of sub-solutions and super-solutions. Schaeffer \cite{Schaeffer1} studies the stationary RVM in the 1.5D setting. Moreover, Weber \cite{Weber1} considered confined stationary solutions of the RVM in a long cylinder in a two and one-half dimensional setting, and Lin \& Strauss \cite{LS1} constructed confined stationary solutions of the RVM in 3D, see the appendix of \cite{LS1}. Sidorov \& Sinitsyn \cite{SS1} gives sufficient conditions for the existence of bifurcation points corresponding to distribution functions of a certain form, and the bifurcation equation is derived and studied for the solutions. 

% etc. contribute further to an expanding theory. 

We aim to construct equilibrium solutions for the system \eqref{Vlasovtwospecies} -- \eqref{Maxwelltwospecies-2} with the boundary conditions above. We consider equilibria with the form
$$f^\pm = \mu^\pm (e^\pm (x, v), p^\pm (x, v)) $$
where $e^\pm$ and $p^\pm$ are the particle energy and angular momentum, respectively, and use a potential formulation together with the axisymmetry assumption of \eqref{Vlasovtwospecies} -- \eqref{Maxwelltwospecies-2} so as to the equations satisfied by the equilibria can be reformulated as a coupled elliptic system (see \eqref{system01-0} -- \eqref{system03-0} below). The resulting elliptic system is highly nonlinear on its right hand side.

\subsection{Literature review}

Investigation for time-independent RVM on a domain with physical boundaries has been carried out over the years. Rein \cite{Rein1} studied weak stationary solutions of the RVM by variational methods. Batt \& Fabian \cite{BattFabian}, Braasch \cite{Braasch1}, Poupaud \cite{Poupaud1}, Sinitsyn \& Dulov \cite{SD1} constructed stationary solutions for boundary value problems making use of sub-solutions and super-solutions. In particular, Batt \& Fabian \cite{BattFabian} constructed a general family of large stationary solutions for boundary value problems when the domain is axially symmetric and does not touch the $z$-axis and Braasch \cite{Braasch1} constructed large stationary solutions for problems with a specific type of axial symmetric domain. Poupaud \cite{Poupaud1} solves for large equilibria for general domains (not necessarily axially symmetric) with compact and connected boundary as well as general Dirichlet boundary conditions. Sidorov \& Sinitsyn \cite{SS1} considered a specific parametrized family of (single species) particle density distribution, and gave sufficient conditions for the existence of bifurcation points corresponding to distribution functions of a certain form. Moreover, the lower dimensional models have also been considered. For example, Schaeffer \cite{Schaeffer1} studied the stationary Vlasov-Maxwell equation (VM) in the 1.5D setting, when the spatial domain being the real line $\mathbb{R}$, see also Guo \cite{G1}.  

The RVM with the presence of an external magnetic field is also of physical concern. This type of problem is considered in Ben-Artzi-Holding \cite{BH1}, in which confined stationary solutions of the 1.5D and 3D cylindrically symmetric RVM are constructed. Also, Weber \cite{Weber1} constructed confined stationary solutions of the RVM in a long cylinder in the 2.5D setting. Moreover, Lin \& Strauss \cite{LS1} constructed confined small stationary solutions of the RVM in 3D (see the appendix of \cite{LS1}). Such type of solutions are related to the mechanism of \emph{magnetic confinement}, see, for example, \cite{NS3}, \cite{Z4}. 

%Weber \cite{Weber1} considered confined stationary solutions of the RVM in a long cylinder in a two and one-half dimensional setting, and Lin \& Strauss \cite{LS1} constructed confined small stationary solutions of the RVM in 3D, see the appendix of \cite{LS1}. 

%{\color{blue}{Dolbeault \cite{Dol1} ... (cannot find this paper)}} 

On the other hand, in plasma theory, an important goal is to study the stability properties of plasmas. 
%The study of the stability properties of macroscopic systems like MHD and other fluid-like models has been carried out a lot (for example \cite{Friedberg1}, \cite{Nicholson1}). However, many plasma phenomena are microscopic so one must consider kinetic models, including the (relativistic and nonrelativistic) Vlasov-Maxwell system, Vlasov-Poisson system, Boltzmann equation, etc. (see \cite{Friedberg1}, \cite{Nicholson1}). 
%The stability of the RVM has been studied a lot in the physics literature. 
%The simplest case is a spatially homogeneous equilibrium with vanishing magnetic fields (for example \cite{G1}, \cite{GS3}). 
A pioneer result is the renowned Penrose's sharp criterion derived in \cite{Penrose1}, which is a sharp criterion on linear stability for a spatially homogeneous equilibrium of the Vlasov-Poisson system. In the subsequent papers \cite{G1}, \cite{GS3}, \cite{LS1}, \cite{LS3} and \cite{LS2}, analysis of the linear stability of a spatially inhomogeneous equilibrium was carried out in domains without spatial boundaries (i.e. whole space or periodic setting). The question of nonlinear stability is much more difficult, see, for example, \cite{LS3}. Moreover, a topic of physical importance is to understand the stability properties of a confined plasma. \cite{HK1} discussed the confinement of a tokamak plasma is discussed using some fluid models. For the microscopic model RVM in bounded domains, few rigorous studies are available. \cite{BH1} studied the instability for confined RVM with symmetry on unbounded domains. In \cite{NS1} and \cite{Z1}, the authors considered the case when the spatial domain is a torus or an axisymmetric bounded region (for example, a tokamak), and toroidal or axisymmetric symmetry is assumed. A sharp criterion of spectral stability was obtained, reducing the problem of determining the linear stability to the positivity of a simpler self-adjoint operator $\mathcal{L}^0$. 
%Moreover \cite{LS2} provided a sharp criterion for spectral stability, with some families of stable and unstable examples provided. 
%n \cite{HK1}, the confinement of a tokamak plasma is discussed using some fluid models and the role of different parts of the boundary are explored. 

\subsection{Description of main results}

%We are able to recover solutions in the historical literature \cite{Rein1}, {\color{blue}{\cite{Dol1},}} \cite{BattFabian}, \cite{Braasch1}, \cite{Poupaud1}, \cite{SD1}, etc. and extend them to more general settings. 

%In particular, Batt \& Fabian \cite{BattFabian} constructed a general family of stationary solutions for boundary value problems when the domain is axial symmetric and does not touch the $z$-axis. ...

\subsubsection{Construction of global sets of equilibria}

The main purpose of the paper is to construct equilibrium solutions for the system \eqref{Vlasovtwospecies} -- \eqref{Maxwelltwospecies-2}. Our strategy is to consider parametrized families of particle density distribution function for the equilibria, and to apply the global implicit function theorem (Theorem \ref{globalimplicitfunctiontheorem}) so as to obtain global solution sets that branch out from "trivial" equilibrium solutions. The result is stated in Theorem \ref{nonneutralglobalthm-general-1}. Similar methods have been applied in varied PDEs, see, for example, \cite{CSV1}, \cite{CWW1}, \cite{CWW2}, \cite{SS1}, \cite{SW1}.  

%For the Vlasov-Maxwell equation (VM), Sidorov \& Sinitsyn \cite{SS1} considered a specific parametrized family of (single species) particle density distribution, and gave sufficient conditions for the existence of bifurcation points corresponding to distribution functions of a certain form. 
 
%$\rightarrow$ Delete this sentence? 
%This allows the authors to obtain global continuums/loop sets of equilibria branching out from the bifurcation points for the VM equation on general spatial domains. 

%The bifurcation equation is derived and studied for the solutions. This allows the authors to obtain global continuums/loop sets of equilibria branching out from the bifurcation points for the VM equation on general spatial domains. 

%and analytic global implicit function theorem (Theorem \ref{analytic-globalimplicitfunctiontheorem})  

%Sidorov \& Sinitsyn \cite{SS1} considered a specific parametrized family of (single species) particle density distribution branching out from the trivial distribution $f (x, v)  = 0$

%Through this method, we are able to recover solutions in the historical literature \cite{Rein1}, {\color{blue}{\cite{Dol1},}} \cite{BattFabian}, \cite{Braasch1}, \cite{Poupaud1}, \cite{SD1}, etc. and extend them to more general settings. 

We consider a general parametrized family of particle density distribution $f^\pm = \mu^\pm (e^\pm (x, v), p^\pm (x, v))$, and construct a \emph{global continuum/loop} of equilibria branching out from the trivial or neutral steady states for the RVM on a general axisymmetric spatial domain with perfect conducting boundary which reflects particles specularly, assuming axisymmetry in the problem. Moreover, in the single species case, with some additional constraints, the global continuum/loop of equilibria we constructed contains a locally analytic curve that is \emph{unbounded}, see Theorem \ref{nonneutralglobalthm-general-analytic-1}. 

Our approach enables us to observe the structure of the solution set, and reveal more information about the change of the electromagnetic potential (and therefore the electromagnetic field) as the density distribution function changes in different ways. In the case when the solution set is unbounded, we make use of the maximum principles and elliptic estimates to investigate size properties of the equilibria constructed in Theorem \ref{nonneutralglobalthm-general-1}. In particular, we construct a continuum of equilibrium solutions with arbitrarily large electromagnetic field, see Proposition \ref{P:prop-sol-general}. It is an open question whether the solutions obtained in this paper are the same as the ones constructed in the historical literature \cite{Rein1}, \cite{BattFabian}, \cite{Braasch1}, \cite{Poupaud1}, \cite{SD1}, etc.. Our approach can also be applied in the case when an external magnetic field is present, see Remark \ref{R:C-Setting}.

\subsubsection{Spectral stability property of equilibria in the solution set}

We also investigate the change of spectral stability property along a family of solutions constructed in Theorem \ref{nonneutralglobalthm-general-1} using the result obtained in \cite{Z1}. Specifically, we track down the solutions along the solution set as the parameter $K$ changes from $0$ to $\infty$, and show that the corresponding equilibrium solutions gradually become spectrally unstable as $K$ grows large. The details are given in Section \ref{S:spectral-stability}, Proposition \ref{smallK-stable} and Theorem \ref{unstableexample2-general}.

\subsubsection{Example}

In addition, we provide a single species example on which the results above can be applied, where the ion density distribution function is parametrized as 
$$ \mu^{K, +} (e,  p) = a^+(K)  \mu^{\pm} (e,  K p) , $$
%These examples cover cases when the density distribution function is parametrized as 
%$$ \mu^{K, \pm} (e,  p) = \mu^{0} (e,  p) + K \mu^{\pm} (e,  p)  $$
%or 
%$$ \mu^{K, \pm} (e,  p) = \mu^{0} (e,  p) +  \mu^{\pm} (e,  K p) $$
where $K \in [0, \infty)$ is the parameter, $a^+(K)$ and $\mu^{+} (e, p)$ are fixed non-negative functions (of $K$ and $(e, p)$, respectively) satisfying certain conditions. Unbounded solution sets of single species equilibria can then be obtained, solutions in these sets have arbitrarily large electromagnetic field (as $K$ gets large). Moreover, the equilibria turn from spectrally stable to unstable as the parameter $K$ changes from $0$ to $\infty$. 

%In the example in Section \ref{S:global-cont-sp}, the equilibria turn from spectrally stable to unstable as the parameter $K$ changes from $0$ to $\infty$. 

%On the other hand, in many real world applications, the plasma is confined to a bounded region. A typical example is the tokamak, which is one of the main foci of research in fusion energy. Therefore, 

%However, there are other domains worthwhile studying. We want to investigate how the geometric structure of the domain influences the stability of the plasma. In this paper, we consider the RVM on a general axisymmetric spatial domain. We consider a certain class of equilibria, assuming axisymmetry in the problem. It is surprising that a sharp criterion of spectral stability can be proven, not just for a torus, but for any axisymmetric domain. In \cite{NS1} the domain $\Omega$ is exactly a torus and the authors used toroidal coordinates. Here in this paper we show that $\Omega$ can be any axisymmetric domain using cylindrical coordinates $(r, \varphi, z)$. An example here is the case when $\Omega$ is a solid ball. 

\subsection{Organization of the paper}

The contents in the paper are arranged as follows. In Section \ref{S:Setup}, we set up the problem in cylindrical coordinates, including the coordinates and the symmetry assumptions. The description of the particle trajectories and the family of equilibria we consider in this paper are given, and relevant function spaces are introduced. The difficulty caused by the singularity at $r=0$ is avoided by using the spaces $H^{k \dagger}$ and $\mathcal{X}$ (see \eqref{Hkdaggerdef} and \eqref{spaceXdef}). In Section \ref{S:global-cont-I} -- \ref{S:global-cont-sp}, we consider the time-independent RVM (with no external magnetic field). In Section \ref{S:global-cont-I}, we consider a family of parametrized particle density distribution functions and construct an unbounded solution set/loop for the time-independent problem \eqref{system01-0}--\eqref{system03-0} with $\textbf{A}^0 = A^0_\varphi e_\varphi $, which gives a steady solution set for RVM. Furthermore, in the single species case, with additional constraints, the global continuum/loop of equilibria we constructed contains a locally analytic curve that is unbounded, yielding a global steady solution set which contains an unbounded analytic steady solution curve for RVM. Section \ref{S:prop-sol-general} makes use of the maximum principles for elliptic equations and gives some properties of the equilibria constructed in Section \ref{S:global-cont-I}, Theorem \ref{nonneutralglobalthm-general-1}. It is then shown in Section \ref{S:spectral-stability} that by properly choosing ways of parametrization and apply results in \cite{Z1}, one can create solution sets along which the linear stability changes. After that, in Section \ref{S:global-cont-sp}, we give an example of parametrized family of particle density distribution functions, where Assumption \ref{Ass: mu-general-1}, Assumption \ref{prop-sol-assump} and Assumption \ref{Ass:stability-general} are all satisfied, and apply the theory in Sections \ref{S:global-cont-I}, \ref{S:prop-sol-general}, \ref{S:spectral-stability}. Lastly, for the readers' convenience, we provide the formulae for derivative operators in the cylindrical coordinates as well as the particle energy $e^\pm (x,v) = \la v \ra \pm \phi^0 (r, z)  $ and the particle angular momentum $ p^\pm (x,v) = r  ( v_\varphi  \pm A^0_\varphi (r, z))  $ (or $ p^\pm (x,v) = r  ( v_\varphi  \pm A^0_\varphi (r, z) \pm A_{\varphi, ext} (r, z) )  $ when an external magnetic potential $\textbf{A}_{ext} = A_{\varphi, ext} (r, z) e_\varphi$ is present) in Appendix \ref{AppendixA} and \ref{AppendixB}, respectively. Important theorems and lemmas that are applied as key tools in this paper are stated in Appendix \ref{AppendixC}. This includes the local and global implicit function theorems, which are utilized to provide the existence of the steady states.

\begin{flushleft}
\textbf{Acknowledgement.} The author would like to express her gratitude to her advisor, Professor Walter Strauss, for bringing this research topic to her attention, and also for the invaluable guidance, encouragement and patience, without which this work would be impossible. 
\end{flushleft}

\section{Set up}
\label{S:Setup}

Let us introduce the basic settings of the stationary RVM problem. We consider the 3D relativistic Vlasov-Maxwell (RVM) system \eqref{Vlasovtwospecies} -- \eqref{Maxwelltwospecies-2} on the axisymmetric bounded domain $\Omega \subset \mathbb{R}^3$ with axisymmetry in the problem, and use cylindrical coordinates $(r, \varphi , z )$. We consider a plasma constrained inside a closed bounded $C^1$ axisymmetric domain $\Omega$ in $\mathbb{R}^3$, i.e. $\Omega$ is rotational invariant around the $z$-axis. $\Omega$ can be viewed as a solid of revolution, determined as follows: Consider a counterclockwisely parametrized $C^1$ curve $\mathcal{C}$ in the plane $\{  \varphi = 0\}$, where $\beta$ is the arclength parameter:
\begin{equation}
r = \tilde{r}(\beta), \quad z = \tilde{z} (\beta ), \quad \beta \in [0, 1], \quad \tilde{r}  > 0  \ \text{for all} \ \beta \in (0, 1) , \quad \tilde{r}, \tilde{z} \in C^1 (0, 1)   ,
\end{equation}
with either
\begin{equation}
\tilde{r} (0) = \tilde{r} (1) > 0 ,  \quad \tilde{z} (0) = \tilde{z} (1) ,  \quad \tilde{z}' (0) = \tilde{z}' (1)   ,
\end{equation}
or, 
\begin{equation}
\tilde{r}(0) =\tilde{r} (1) =0  ,  \quad   \tilde{z}' / \tilde{r}' =0 \ \text{at} \ \beta = 0, 1   .
\end{equation}
Rotating $\mathcal{C}$ around the $z$-axis gives the boundary for the domain $\partial \Omega$. 
%We also consider the case when the counterclockwisely parametrized $C^1$ curve $\mathcal{C}$ in the plane $\{  \varphi = 0\}$ is given by
%\begin{equation}
%r = \tilde{r}(z), \quad z \in (-\infty, \infty), \quad \tilde{r} \geq 0  \ \text{for all} \ z \in (-\infty, \infty) , \quad \tilde{r} \in C^1    .
%\end{equation}
%Rotating $\mathcal{C}$ around the $z$-axis gives the boundary for the domain $\partial \Omega$. Notice that in this case $\Omega$ is a unbounded domain with the shape of a "infinite cylinder" type. Similarly one can consider two counterclockwisely parametrized $C^1$ curves $\mathcal{C}_1$, $\mathcal{C}_2$ in the plane $\{  \varphi = 0\}$ given by
%\begin{equation}
%\begin{split}
%& \mathcal{C}_1: \ r = \tilde{r}_1(z), \quad z \in (-\infty, \infty), \quad \tilde{r}_1 \geq 0  \ \text{for all} \ z \in (-\infty, \infty) , \quad \tilde{r}_1 \in C^1   ,  \\
%& \mathcal{C}_2: \ r = \tilde{r}_2 (z), \quad z \in (-\infty, \infty), \quad \tilde{r}_2 \geq \tilde{r}_1 \geq 0  \ \text{for all} \ z \in (-\infty, \infty) , \quad \tilde{r}_2 \in C^1   .  \\
%\end{split}
%\end{equation}
%Rotating $\mathcal{C}_1$, $\mathcal{C}_2$ around the $z$-axis gives the boundary for the domain $\partial \Omega$ by letting $\Omega := \{ x = (r, \varphi, z) \in \mathbb{R}^3 : \tilde{r}_1(z) < r < \tilde{r}_2(z)  \}$. Notice that in this case $\Omega$ is a unbounded domain with the shape of an "infinite annular cylinder" type. 
In all these cases, the boundary $\partial \Omega$ is $C^1$ smooth. We denote
$$d := \sup_{x \in \Omega} r(x) .  $$
Then $d < +\infty$.

%Let $\partial \Omega$ be the surface obtained by rotating $\mathcal{C}$ around the $z$-axis. Then the boundary $\partial \Omega$ is $C^1$ smooth.

%Let $\partial \Omega$ be the surface obtained by rotating $\mathcal{C}$ around the $z$-axis. 

Let $e_r$, $e_\varphi$ and $e_z$ be the unit vectors in the cylindrical coordinate system (see Appendix \ref{AppendixA}), and $e_n$, $e_{tg}$ to be the unit vectors in outward normal direction and tangential direction orthogonal to $e_\varphi$ on $\partial \Omega$, respectively. Then on $\partial \Omega$, we can write the outward unit normal vector as $e_n (x)  = ( \tilde{z}' e_r - \tilde{r}' e_z)/\sqrt{\tilde{z}'^2+\tilde{r}'^2} $, and $e_{tg} (x) =  ( - \tilde{r}'  e_r - \tilde{z}' e_z)/\sqrt{\tilde{z}'^2+\tilde{r}'^2} $.

On the boundary, we assume the specular condition on the density function $f^\pm$ in case the particle hits the boundary: 
% This means that $f^\pm$ is even with respect to $v_n= v \cdot e_n/|e_n|$ on $\partial \Omega$, i.e.
\begin{equation} \label{specularBC}
f^{\pm} (t,x,v) = f^{\pm} (t,x,v-2(v \cdot e_n (x)) e_n (x) ), \  e_n(x) \cdot v <0,  \forall x \in \partial \Omega, \forall v \in \mathbb{R}^3  .
\end{equation}
%We would like to remark that if $\textbf{B}_{ext}$ is a confinement field, i.e. it blows up at the boundary, then the particles are kept away from $\partial \Omega$ if they are initially away form it, Hence we actually do not need the specular boundary condition in this case. {\textcolor{red} {???}} 
For the electric and magnetic fields, we put down the perfect conductor boundary condition: 
\begin{equation} \label{perfectconductorboundary}
\textbf{E}(t,x) \times e_n(x) =0, \quad  \textbf{B} (t,x) \cdot e_n(x) =0, \forall x \in \partial\Omega  .
\end{equation}
We introduce the electric potential $\phi$ and the magnetic potential $\textbf{A}$:
\begin{equation} \label{potentialdefinition}
\textbf{E} = -\nabla \phi -\partial_t \textbf{A}, \quad  \textbf{B} = \nabla \times \textbf{A} 
\end{equation} 
and impose the Coulomb gauge
\begin{equation}
\nabla \cdot \textbf{A} =0  . 
\end{equation}

\begin{remark}
The choices of $\phi$ and $\textbf{A}$ are not unique. In fact, the choice of $\textbf{A}$ can differ by the gradient of a harmonic function.
\end{remark}

We use the cylindrical coordinates $(r, \varphi , z)$, and make the axisymmetry assumption:
\begin{equation}  \label{varphiindependenceassumption}
\phi, \ A_r, \  A_z,  \ A_\varphi \ \text{are independent  of} \ \varphi.
\end{equation}
Therefore, $f^{\pm}$ does not depend \textit{explicitly} on $\varphi$, although it might depend on it \textit{implicitly} through the components of $v$. The Maxwell system can be reformulated as
\begin{equation}  \label{RVM-potentialform-1}
-\Delta \phi = \rho   , 
\end{equation}
\begin{equation}   \label{RVM-potentialform-2}
(\partial_t^2 -\Delta + \frac{1}{r^2}) A_\varphi  = j_\varphi   , 
\end{equation}
\begin{equation}   \label{RVM-potentialform-3}
(\partial_t^2 -\Delta ) \tilde{\textbf{A}} + \partial_t \nabla \phi  = \tilde{\textbf{j}}  , 
\end{equation}
where $\tilde{\textbf{A}} := A_r e_r + A_z e_z$.

%We consider the equilibrium field, with the potential denoted by $(\phi^0, A_\varphi^0, \tilde{\textbf{A}}^0)$, and for which the Maxwell equations become
%\begin{equation} \label{system01-0}
%-\Delta \phi^0 = \rho^0   , 
%\end{equation}
%\begin{equation} \label{system02-0}
%( -\Delta + \frac{1}{r^2}) A_\varphi^0  = j_\varphi^0   , 
%\end{equation}
%\begin{equation} \label{system03-0}
%(  -\Delta ) \tilde{\textbf{A}}^0  = \tilde{\textbf{j}}^0   .
%\end{equation}
%We consider the system \eqref{system01-0} -- \eqref{system03-0} with 
%$$\tilde{\mathbf{A}}^0 =0 , $$
We consider an equilibrium field, with the potential denoted by $(\phi^0, A_\varphi^0, \tilde{\textbf{A}}^0)$:
\begin{equation} \label{system01-0}
-\Delta \phi^0 = \rho^0   , 
\end{equation}
\begin{equation} \label{system02-0}
( -\Delta + \frac{1}{r^2}) A_\varphi^0  = j_\varphi^0   , 
\end{equation}
\begin{equation} \label{system03-0}
(  -\Delta ) \tilde{\textbf{A}}^0  = \tilde{\textbf{j}}^0   .
\end{equation}
We assume that the magnetic potential $\textbf{A}^0$ satisfies
$$\textbf{A}^0 = A^0_\varphi e_\varphi ,  $$ 
which implies
$$B^0_\varphi =0 . $$
%Under this constraint, \eqref{system01-0} -- \eqref{system03-0} become a system with reduced dimension:
The time-independent Maxwell system \eqref{system01-0} -- \eqref{system03-0} is now reduced to 
\begin{equation} \label{system01}
-\Delta \phi^0 = \rho^0   , 
\end{equation}
\begin{equation} \label{system02}
( -\Delta + \frac{1}{r^2}) A_\varphi^0  = j_\varphi^0  . 
\end{equation}

According to \eqref{perfectconductorboundary}, we assume the following boundary conditions for the potentials, Section 4 in \cite{Z1} for details: 
%{\color{red}{(according to \eqref{perfectconductorboundary})}}
\begin{equation}  \label{fullboundarycondition}
\phi^0 =0, \qquad  A_\varphi^0 =0, \qquad   A_{tg}^0 =0, \qquad  \tilde{z}' \partial_r A_n^0 - \tilde{r}' \partial_z A_n^0 + \frac{\tilde{z}'}{\tilde{r}} A_n^0 =0, \qquad \forall x \in \partial \Omega  . 
\end{equation}
%We consider an equilibrium that satisfies
%$$B^0_\varphi =0 , $$
%which can be achieved by taking
%$$\textbf{A}^0 = A^0_\varphi e_\varphi . $$ 
%$$\textbf{A}^0 = A^0_\varphi e_\varphi $$ 
%which implies $B^0_\varphi =0$. 
%{\textcolor{red}{Necessary?}}
Then the equilibrium field is
\begin{equation}
\textbf{E}^0 = -\nabla \phi^0 = -\frac{\partial \phi^0}{\partial r} e_r    - \frac{\partial \phi^0}{ \partial z} e_z   , 
\end{equation}
\begin{equation}
\begin{split}
\textbf{B}^0 
& =  - \frac{\partial A^0_\varphi}{\partial z} e_r +  \frac{1}{r}  \frac{\partial ( r A^0_\varphi )}{\partial r} e_z    .  \\
\end{split}
\end{equation}
The boundary condition \eqref{fullboundarycondition} then becomes
\begin{equation}  \label{simplifiedboundarycondition}
\phi^0 =0, \qquad  A_\varphi^0 =0, \qquad \forall x \in \partial \Omega  . 
\end{equation}
We define the particle trajectories as
\begin{equation}  \label{particletrajectoryODE}
\dot{X}^\pm = \hat{V}^\pm,  \  \dot{V}^\pm =  \pm \textbf{E}^0 (X^\pm) \pm \hat{V}^\pm \times \textbf{B}^0 (X^\pm)
\end{equation}
with initial values $(X^\pm (0;x,v), V^\pm (0;x,v) ) = (x,v) $. Each particle trajectory exists and preserves the axisymmetry up to the first time it meets the boundary. Let $s_0$ be a time when the trajectory $X^\pm (s_0 -; x,v)$ hits the boundary $\partial \Omega$. Recall that $v_n  = v \cdot e_n = \frac{\tilde{z}' v_r - \tilde{r}' v_z}{\sqrt{ \tilde{z}'^2 + \tilde{r}'^2 }}$, $v_{tg} = v \cdot e_{tg}  = \frac{- \tilde{z}' v_z - \tilde{r}' v_r}{\sqrt{\tilde{z}'^2 + \tilde{r}'^2}}$. For any given $(x, v)$ and $(X^\pm, V^\pm)$ with $x$ and $X^\pm$ on $\partial \Omega$, we re-decompose $v$ and $V^\pm$ into their $n$-component, $tg$-component and $\varphi$- component: $v =  v_n e_n + v_{tg} e_{tg} +  v_\varphi e_\varphi   $, $V^\pm =  V^\pm_n e_n + V^\pm_{tg} e_{tg} +  V^\pm_\varphi e_\varphi   $, and define 
\begin{equation}  \label{reflectedvelocity}
v_* = - v_n e_n + v_{tg} e_{tg} +v_\varphi e_\varphi \ , \ V^\pm_* = - V^\pm_n e_n + V^\pm_{tg} e_{tg} +  V^\pm_\varphi e_\varphi  .
\end{equation}
Thus from the specular boundary condition, the trajectory can be continued by the rule
\begin{equation}   \label{particlereflection}
( X^\pm (s_0 + ; x,v), V^\pm (s_0 + ; x,v)  ) = ( X^\pm (s_0 - ; x,v), V_*^\pm (s_0 - ; x,v)  )   .
\end{equation}

%Furthermore, in the case when no external field is applied, 

Furthermore, we consider the following invariants along the particle trajectories given by \eqref{particletrajectoryODE}
\begin{equation}  \label{epdefinition}
\begin{split}
& e^\pm (x,v) = \la v \ra \pm \phi^0 (r, z)  ,   \  p^\pm (x,v) = r  ( v_\varphi  \pm A^0_\varphi (r, z))  . \\
\end{split}
\end{equation}
The proof of the invariance of $e^\pm$ and $p^\pm$ can be found in Appendix \ref{AppendixB}. We assume the equilibrium has a particle density of the form 
$$f^{0,\pm}  (x,v) = \mu^\pm (e^\pm (x,v), p^\pm (x,v)  ) ,  $$
% where 
%\begin{equation}  \label{epdefinition}
%\begin{split}
%& e^\pm (x,v) = \la v \ra \pm \phi^0 (r, z)  ,   \  p^\pm (x,v) = r  ( v_\varphi  \pm A^0_\varphi (r, z))  . \\
%\end{split}
%\end{equation}
%Here $e^\pm$ and $p^\pm$ are invariant along the particle trajectories under our assumptions (with specular boundary condition). (The proof of the invariance of $e^\pm$ and $p^\pm$ can be found in Appendix B.) 
where $\mu^{\pm} (e, p)$ are non-negative $C^1$ functions which satisfy
%\begin{equation} \label{decayassumption-2}
%\begin{split}
%%\mu^\pm_e (e, p) <  0, \  
%& |\mu^0  (e , p )|+ |\mu^0_p (e , p )| +|\mu^0_e (e, p)| \leq \frac{C_\mu }{1+ |e|^\delta + |p|^\delta}, \ \delta > 3 , \\ 
%& |\mu^\pm  (e , p )|+ |\mu^\pm_p (e , p )| +|\mu^\pm_e (e, p)| \leq \frac{C_\mu }{1+ |e|^\delta + |p|^\delta}, \ \delta > 3 . \\
%\end{split}
%\end{equation}
%{\color{blue}{
\begin{equation} \label{decayassumption-0}
\begin{split}
%\mu^\pm_e (e, p) <  0, \  
%& |\mu^0  (e , p )|+ |\mu^0_p (e , p )| +|\mu^0_e (e, p)| \leq \frac{C_\mu }{1+ |e|^\delta  }, \ \delta > 3 , \\ 
& |\mu^\pm  (e , p )|+ |\mu^\pm_p (e , p )| +|\mu^\pm_e (e, p)| \leq \frac{C_\mu }{1+ |e|^\delta }, \ \delta > 3 . \\
\end{split}
\end{equation}
%}}
Then $f^{0,\pm}  (x,v)$ are invariant along the particle trajectories given by \eqref{particletrajectoryODE}. 

\begin{remark} \label{R:C-Setting}
Our results can also be applied to the RVM with an external magnetic field being applied. Consider an external magnetic potential
$$\textbf{A}_{ext} = A_{\varphi, ext} e_\varphi , \  A_{\varphi, ext} \text{ is independent of } \varphi, \text{ only depends on } r, z .  $$ 
Hence the corresponding external magnetic field satisfies $B_{\varphi, ext} =0$. Again, we assume the equilibrium has a particle density of the form $f^{0,\pm}  (x,v) = \mu^\pm (e^\pm (x,v), p^\pm (x,v)  )  $, where 
\begin{equation}  \label{epdefinition-ext}
\begin{split}
& e^\pm (x,v) = \la v \ra \pm \phi^0 (r, z)  ,   \  p^\pm (x,v) = r  ( v_\varphi  \pm A^0_\varphi (r, z) \pm  A_{\varphi, ext} (r, z))  . \\
\end{split}
\end{equation}
$e^\pm$ and $p^\pm$ are invariant along the particle trajectories 
\begin{equation}  \label{particletrajectoryODE-ext}
\dot{X}^\pm = \hat{V}^\pm,  \  \dot{V}^\pm =  \pm \textbf{E}^0 (X^\pm) \pm \hat{V}^\pm \times \textbf{B}^0 (X^\pm) \pm \hat{V}^\pm \times \textbf{B}_{ext} (X^\pm) .
\end{equation}
with all the other assumptions as mentioned in the previous case without an external magnetic field, see Appendix B.
\end{remark}

We are going to solve the time-independent Maxwell system with the boundary conditions above as well as $\textbf{A}^0 = A^0_\varphi e_\varphi $, which simplifies to \eqref{system01} -- \eqref{system02}). 
%$$B^0_\varphi =0 , $$
%which can be achieved by taking
%$$\textbf{A}^0 = A^0_\varphi e_\varphi . $$ 
Denote $Y =L^2_{1/r^2 } (\Omega)$, where the weight $1/r^2$ gives some singularity at $r=0$ if $\Omega$ touches the $z$-axis. For this, we define  
%, i.e. the weighted-$L^2$ space with weight $1/r^2$. 
\begin{equation} \label{Hkdaggerdef}
H^{k \dagger}  (\Omega) := \{  g \in L^{2, \tau} (\Omega) |  e^{i\varphi} g \in H^k (\Omega)   \}  (k=1,2)   . 
\end{equation}
Here the letter $\tau$ denotes the axisymmetric constraint for functions. We then have $H^{2 \dagger} (\Omega) \subset Y$ because the identity 
\begin{equation} \label{Deltagexp}
-\Delta (g e^{i\varphi}) = (- \Delta + \frac{1}{r^2} ) g e^{i \varphi} 
\end{equation}  
holds for any $\varphi$-independent function $g$. 
%Indeed, 
%\begin{equation*}
%\begin{split}
%-\Delta (g e^{i\varphi})
%& = - \frac{1}{r} \frac{\partial}{\partial r} (r \frac{\partial g}{\partial r}) e^{i\varphi} - \frac{\partial^2 g}{\partial z^2} e^{i\varphi} - \frac{1}{r^2} \frac{\partial^2 (e^{i\varphi})}{\partial \varphi^2} g \\ 
%& = - \frac{1}{r} \frac{\partial}{\partial r} (r \frac{\partial g}{\partial r}) e^{i\varphi} - \frac{\partial^2 g}{\partial z^2} e^{i\varphi} + \frac{1}{r^2} g e^{i\varphi} \\ 
%& = (-\Delta + \frac{1}{r^2}) g e^{i\varphi}  . \\
%\end{split}
%\end{equation*}
We will use the following space regularly:
\begin{equation}   \label{spaceXdef}
\mathcal{X} := \{ g \in H^2 (\Omega) : g \ \text{independent of} \ \varphi, \ \exp (i \varphi) g \in H^2 (\Omega) , \ g|_{\partial \Omega} =0    \}  \subset H^{2 \dagger} (\Omega)  .  
\end{equation}
%The main difficulty arises from the nonlinearity on the right hand side of the equations \eqref{system01} -- \eqref{system03}. 

In the rest of the paper, $A\lesssim B$ ($A\gtrsim B$) means $A \leq CB$ ($A \geq CB$) for some constant $C>0$ independent of $K$, where $K$ is the parameter in the parametrized family of equilibria $\{(\mu^{K, \pm}, \phi^{K, 0}, A_\varphi^{K, 0})\}$, see Section \ref{S:global-cont-I}.

\section{Global continuation} 
%\section{Global continuation for the non-neutral case: Part I} 
\label{S:global-cont-I}

%with $\tilde{\mathbf{A}}^0 =0 $. 
In this section, we consider the system \eqref{system01} -- \eqref{system02}. Let us consider a parametrized family of functions $\{ \mu^{K, \pm} (e, p) \}_{K \in [0, +\infty)}$, which satisfy
%, in which all the functions $\mu^{K, \pm} (e, p)$ are non-negative $C^1$ functions that satisfy
%\begin{equation} \label{decayassumption}
%\begin{split}
%%\mu^\pm_e (e, p) <  0, \  
%%& |\mu^0  (e , p )|+ |\mu^0_p (e , p )| +|\mu^0_e (e, p)| \leq \frac{C_\mu }{1+ |e|^\delta  }, \ \delta > 3 , \\ 
%& |\mu^{K, \pm}  (e , p )|+ |\mu^{K, \pm}_p (e , p )| +|\mu^{K, \pm}_e (e, p)| \leq \frac{C_\mu (K) }{1+ |e|^\delta }, \ \delta > 3 . \\
%\end{split}
%\end{equation}
%Moreover, we assume the following:
\begin{assumption} \label{Ass: mu-general-1}
We assume the following:  \\
1) $\mu^{K, \pm} (e, p)$ are non-negative $C^1$ functions, and
\begin{equation} \label{decayassumption}
\begin{split}
%\mu^\pm_e (e, p) <  0, \  
%& |\mu^0  (e , p )|+ |\mu^0_p (e , p )| +|\mu^0_e (e, p)| \leq \frac{C_\mu }{1+ |e|^\delta  }, \ \delta > 3 , \\ 
& |\mu^{K, \pm}  (e , p )|+ |\mu^{K, \pm}_p (e , p )| +|\mu^{K, \pm}_e (e, p)| \leq \frac{C_\mu (K) }{1+ |e|^\delta }, \ \delta > 3 . \\
\end{split}
\end{equation}
2) $\mu^{K, \pm} (e, p)$ depend on $K$ in the $C^1$ sense, i.e. $\frac{\partial \mu^{K, \pm} (e, p)}{\partial K} \in C^0$. Moreover, $\frac{\partial \mu^{K, \pm} (e, p)}{\partial K} |_{K=0} = 0$.  \\
3) $\mu^{0, +} (e, p) = \mu^{0, -} (e, p) = M^{0} (e, p)$ for some non-negative function $M^{0} (e, p)$ satisfying the decay assumption \eqref{decayassumption}.  \\
%2) $\mu^{K, \pm} (e, p)$ are even in $K$.  \\
%2) $\mu^{0, +} (e, p) = \mu^{0, -} (e, p) := M^{0} (e, p)$, and $M^{0} (e, p)$ is even in $p$.  \\
4) The matrix operator 
%$(D_{(u, w)} G) (u_0, w_0, K_0) = (D_{(u, w)} G) (0, 0, 0)  : \mathcal{X} \times \mathcal{X} \rightarrow H^2(\Omega) \times H^2(\Omega)$ is bounded with a bounded inverse.
\begin{equation}
% \begin{array}{ccc}
%(D_{(u, w)} G) (u_0, w_0 , K_0)= 
% \end{array}  
\left( \begin{array}{ccc}
\tilde{J}_{11} \ \tilde{J}_{12} \\
\tilde{J}_{21} \ \tilde{J}_{22} \\
\end{array} \right)
: \mathcal{X} \times \mathcal{X}   \rightarrow H^2(\Omega) \times H^2(\Omega) 
\end{equation}
is bounded with a bounded inverse. Here
\begin{equation}
\begin{split}
& \tilde{J}_{11} (\delta u)   = \delta u -    ( -\Delta + \frac{1}{r^2})^{-1} \big\{ f_1 (r) \delta u  \big\} , \\
& \tilde{J}_{12} (\delta w)   = -   ( -\Delta + \frac{1}{r^2})^{-1} \big\{   f_2 (r)  \delta w   \big\} ,   \\
& \tilde{J}_{21} (\delta u)   =      \Delta^{-1} \big\{  f_3 (r) \delta u   \big\}  , \\
& \tilde{J}_{22} (\delta w)   = \delta w +   \Delta^{-1} \big\{  f_4 (r) \delta w   \big\} , \\
\end{split}  
\end{equation}
with 
\begin{equation*}
\begin{split}
& f_1 (r) := 2 \int_{\mathbb{R}^3} r \hat{v}_\varphi M^0_p (\la v \ra , r v_\varphi  )  dv , 
\
f_2 (r) := 2 \int_{\mathbb{R}^3} \hat{v}_\varphi M^0_e (\la v \ra , r v_\varphi  )  dv , \\
& f_3 (r) := 2 \int_{\mathbb{R}^3} r M^0_p (\la v \ra , r v_\varphi  )  dv , \
f_4 (r) := 2 \int_{\mathbb{R}^3} M^0_e (\la v \ra , r v_\varphi  )  dv . \\
\end{split}
\end{equation*}
\end{assumption}
%For any $K \in \mathbb{R}$, consider the equilibrium $\mu^{K, \pm}$ given by
%\begin{equation} \label{E:muK-def-general}
%\begin{split}
%& \mu^{K, +} (e^{K, +},  p^{K, +}) = \gamma \mu^{0} (e^{K, +},  p^{K, +}) + a^+(K) \mu^+ (e^{K, +},  p^{K, +})  ,   \\
%& \mu^{K, -} (e^{K, -},  p^{K, -}) =  \gamma \mu^{0} (e^{K, -},  p^{K, -}) + a^-(K) \mu^- (e^{K, -},   p^{K, -})  .   \\
%\end{split}
%\end{equation}
%Here $a^\pm (K)$ are nonnegative even functions of $K$ which satisfies $a^+ (0) =  a^- (0) =0$, and there exists no $K \neq 0$ such that $a^+ (K) =  a^- (K) =0$. Moreover, 
By \eqref{system01} -- \eqref{system02}, we aim to solve the elliptic system
\begin{equation} 
\label{system01-nonneutral-arbimass}
-\Delta \phi^{K, 0} = \int_{\mathbb{R}^3} (\mu^{K, +} (e^{K, +}, p^{K, +}) - \mu^{K, -} (e^{K, -}, p^{K, -})) dv   , 
\end{equation}
\begin{equation} 
\label{system02-nonneutral-arbimass}
( -\Delta + \frac{1}{r^2}) A_\varphi^{K, 0}  = \int_{\mathbb{R}^3} \hat{v}_\varphi (\mu^{K, +} (e^{K, +}, p^{K, +}) - \mu^{K, -} (e^{K, -}, p^{K, -})) dv  ,
\end{equation}
where
$$ e^{K, \pm} = \la v \ra \pm \phi^{K, 0}(x) , \  p^{K, \pm} =r ( v_\varphi \pm A_\varphi^{K , 0}(x) )  . $$
%where $(\phi^{K, 0}, A_\varphi^{K , 0}  )$ is the solution to the corresponding elliptic system
%Moreover, let $M$ be the total mass, then
%\begin{equation} 
%\label{system03-nonneutral-arbimass}
%\int_\Omega \int_{\mathbb{R}^3}  (\mu^{K, +} (e^{K, +}, p^{K, +}) + \mu^{K, -} (e^{K, -}, p^{K, -})) dv dx - M =0 \ .
%\end{equation}
%and $\tilde{\textbf{A}}^{\lambda, 0}  = 0 $. 
%\begin{equation} 
%%\label{system03}
%\tilde{\textbf{A}}^{\lambda, 0}  = 0 \ .
%\end{equation}
It is obvious that $(\phi^{K, 0},  A_\varphi^{K, 0} , K) = (0, 0, 0)$ is a solution to \eqref{system01-nonneutral-arbimass} -- \eqref{system02-nonneutral-arbimass}. Starting from this trivial solution, we will construct a global solution set 
$$\{ ( \mu^{K, \pm} (e^{K, \pm} (x,v), p^{K, \pm} (x,v)  ) ,  \phi^{K, 0} (x),  A^{K, 0}_\varphi (x) ) \}_{K \in [0, +\infty)} $$ 
for \eqref{system01-nonneutral-arbimass} - \eqref{system02-nonneutral-arbimass}, therefore obtain a global solution curve for the time-independent RVM. 

%\begin{remark}  \label{R:even-in-K}
%The assumption that $\mu^{K, \pm} (e, p)$ are even in $K$ is mainly for the convenience in argument. We will focus on $K \geq 0$ later. 
%\end{remark}

We use the Global Implicit Function Theorem \ref{globalimplicitfunctiontheorem} to obtain

\begin{theorem}  \label{nonneutralglobalthm-general-1}
%There exists a discrete subset $S_\gamma$ of $[0, +\infty)$, such that for all $\gamma \in [0, +\infty) \setminus S_\gamma$, the following holds: \\
Let $\Omega$ be a $C^1$ axisymmetric bounded domain as stated in Section \ref{S:Setup}. Let $\mu^{K, \pm} (e, p)$ satisfy Assumption \ref{Ass: mu-general-1}. Then there exists an unbounded continuous solution set or loop $\mathcal{C} := \{ ( A_\varphi^{K, 0} , \phi^{K, 0}, K) \} \subset H^2(\Omega) \times H^2(\Omega) \times \mathbb{R}$ to the system \eqref{system01-nonneutral-arbimass} -- \eqref{system02-nonneutral-arbimass} with $e^{K, \pm} = \la v \ra \pm \phi^{K, 0}(x) , \  p^{K, \pm} =r ( v_\varphi \pm A_\varphi^{K , 0}(x) ) $, in which the solution $(0, 0, 0)$ is included. %If $\mu^{0, \pm } =0$, then $\mathcal{C}$ is unbounded. \\
\end{theorem}

\begin{proof}
We extend the definition of $\mu^{K, \pm} (e, p)$ evenly with respect to $K$ by taking $\mu^{-K, \pm} (e, p) = \mu^{K, \pm} (e, p)$ for all $K \geq 0$. Recall
$$ \mathcal{X} := \{ g \in H^2 (\Omega) : g \ \text{independent of} \ \varphi, \ \exp (i \varphi) g \in H^2 (\Omega) , \ g|_{\partial \Omega} =0    \}  . $$
For any $h \in L^2 (\Omega)$, define $\Delta^{-1} h \in H^2 (\Omega) \cap H^1_0 (\Omega)$ to be the unique solution $g$ to the elliptic problem with Dirichlet boundary condition
\begin{equation}
\Delta g = h   , \ g|_{\partial \Omega} =0  .  
\end{equation}
Then $\Delta^{-1}$ is a compact operator from $L^2 (\Omega)$ to $L^2 (\Omega)$ since the inclusion of $H^1 (\Omega)$ into $L^2 (\Omega)$ is compact. Also, for any $h \in  L^2 (\Omega)$, we define $( -\Delta + \frac{1}{r^2})^{-1} h \in H^2 (\Omega) \cap H^1_0 (\Omega)$ to be the unique solution $g$ to the elliptic problem with Dirichlet boundary condition
\begin{equation}
( -\Delta + \frac{1}{r^2})  g = h   , \ g|_{\partial \Omega} =0  .  
\end{equation}
To do this, we first solve the equation
\begin{equation*}
  -\Delta \tilde{g} = h e^{i\varphi}  , \ \tilde{g}|_{\partial \Omega} =0  ,
\end{equation*}
then take 
\begin{equation*}
( -\Delta + \frac{1}{r^2})^{-1} h  =g :=e^{-i\varphi} \tilde{g} .
\end{equation*}
Now $( -\Delta + \frac{1}{r^2})^{-1}$ is a compact operator from $L^2 (\Omega)$ to $L^2 (\Omega)$ since the inclusion of $H^1 (\Omega)$ into $L^2 (\Omega)$ is compact.

%{\color{red}{Take $X = \{ (u, w) \in  \mathcal{X} \times \mathcal{X} : G(u, w) \in   \mathcal{X} \times \mathcal{X} \}$? }}
 
%We take $X =Z =  \mathcal{X} \times \mathcal{X} $ and $U = \mathcal{X} \times \mathcal{X} \times \mathbb{R}$. Let 
%We take $X =Z =  H^2_0 (\Omega) \times H^2_0 (\Omega)$ and $U = \mathcal{X} \times \mathcal{X} \times \mathbb{R}$. Let 
We take $X =  L^2_0 (\Omega) \times L^2_0 (\Omega)$, $Z =  H^2_0 (\Omega) \times H^2_0 (\Omega)$ and $U = \mathcal{X} \times \mathcal{X} \times \mathbb{R}$. Let 
\begin{equation}  \label{E:nonneutralglobalthm-general-1-eq1}
\begin{split}
& \begin{array}{ccc}
G(u, w, K) : = 
 \end{array}  
 \\
& \left( \begin{array}{ccc}
 u -  ( -\Delta + \frac{1}{r^2})^{-1} \big\{ \int_{\mathbb{R}^3} \hat{v}_\varphi (\mu^{K, +} (\la v \ra + w, r ( v_\varphi + u )) - \mu^{K, -} (\la v \ra - w, r ( v_\varphi - u ))) dv \big\} \\
 w +    \Delta^{-1} \big\{ \int_{\mathbb{R}^3}  (\mu^{K, +} (\la v \ra + w, r ( v_\varphi + u )) - \mu^{K, -} (\la v \ra - w, r ( v_\varphi - u ))) dv \big\}  
\end{array} \right)
\begin{array}{ccc}
  . \end{array} 
\end{split}
\end{equation}
Then $G$ is a mapping from $X \times \mathbb{R}$ to $Z$. Solving the system \eqref{system01-nonneutral-arbimass} -- \eqref{system02-nonneutral-arbimass} is equivalent to solving the equation
\begin{equation}
G(u, w, K) =0  .
\end{equation}
For each solution $(u,w, K)$ (if exists), $(u, w)$ is automatically in $(H^2 (\Omega) \cap H^1_0 (\Omega))^2$.

Consider the point $(u_0, w_0, K_0) = (0, 0, 0)$ in $X \times \mathbb{R}$. One can verify that $G(u_0, w_0, K_0) =0$. We have: $G$ is a $C^1$ mapping form $X \times \mathbb{R}$ to $Z$, with 
\begin{equation}
 \begin{array}{ccc}
(D_{(u, w)} G) (u, w , K)= 
 \end{array}  
\left( \begin{array}{ccc}
J_{11} \ J_{12} \\
J_{21} \ J_{22} \\
\end{array} \right)
\begin{array}{ccc}
  , \end{array} 
\end{equation}
where 
\begin{equation}
\begin{split}
J_{11} (\delta u) 
& = \delta u -   ( -\Delta + \frac{1}{r^2})^{-1} \big\{ \int_{\mathbb{R}^3} r \hat{v}_\varphi (\mu^{K, +}_p (\la v \ra + w, r ( v_\varphi + u )) + \mu^{K, -}_p (\la v \ra - w, r ( v_\varphi - u ))) \delta u dv \big\} , \\
\end{split}  
\end{equation}
\begin{equation}
\begin{split}
J_{12} (\delta w) 
& = -   ( -\Delta + \frac{1}{r^2})^{-1} \big\{ \int_{\mathbb{R}^3} \hat{v}_\varphi (\mu^{K, +}_e (\la v \ra + w, r ( v_\varphi + u )) + \mu^{K, -}_e (\la v \ra - w, r ( v_\varphi - u ))) \delta w dv \big\} , \\
\end{split}
\end{equation}
\begin{equation}
\begin{split}
J_{21} (\delta u) 
& =    \Delta^{-1} \big\{ \int_{\mathbb{R}^3} r (\mu^{K, +}_p (\la v \ra + w, r ( v_\varphi + u )) + \mu^{K, -}_p (\la v \ra - w, r ( v_\varphi - u ))) \delta u dv  \big\} , \\
\end{split}
\end{equation}
\begin{equation}
\begin{split}
J_{22} (\delta w) 
& = \delta w +  \Delta^{-1} \big\{ \int_{\mathbb{R}^3}  (\mu^{K, +}_e (\la v \ra + w, r ( v_\varphi + u )) + \mu^{K, -}_e (\la v \ra - w, r ( v_\varphi - u ))) \delta w dv \big\} . \\
\end{split}
\end{equation}
Hence we have
\begin{equation}
 \begin{array}{ccc}
(D_{(u, w)} G) (u_0, w_0 , K_0)= 
 \end{array}  
\left( \begin{array}{ccc}
\tilde{J}_{11} \ \tilde{J}_{12} \\
\tilde{J}_{21} \ \tilde{J}_{22} \\
\end{array} \right)
,
\end{equation}
where
\begin{equation}
\begin{split}
\tilde{J}_{11} (\delta u) 
& = \delta u - 2   ( -\Delta + \frac{1}{r^2})^{-1} \big\{ \int_{\mathbb{R}^3} r \hat{v}_\varphi M^0_p (\la v \ra , r v_\varphi  )  \delta u dv \big\}  \\
& = \delta u -    ( -\Delta + \frac{1}{r^2})^{-1} \big\{ f_1 (r) \delta u  \big\} , \\
%& -   \lambda^2 ( -\Delta + \frac{1}{r^2})^{-1} \big\{ \int_{\mathbb{R}^3} r \hat{v}_\varphi (\mu^{+}_p (\la v \ra + w, r ( v_\varphi + u )) + \mu^{-}_p (\la v \ra - w, r ( v_\varphi - u ))) \delta u dv \big\} \ , \\
\end{split}  
\end{equation}
\begin{equation}
\begin{split}
\tilde{J}_{12} (\delta w) 
& = - 2  ( -\Delta + \frac{1}{r^2})^{-1} \big\{ \int_{\mathbb{R}^3} \hat{v}_\varphi  M^0_e (\la v \ra  , r   v_\varphi )   \delta w dv \big\}   \\
& = -   ( -\Delta + \frac{1}{r^2})^{-1} \big\{   f_2 (r)  \delta w   \big\}   \\
%& -  \lambda^2 ( -\Delta + \frac{1}{r^2})^{-1} \big\{ \int_{\mathbb{R}^3} \hat{v}_\varphi (\mu^{+}_e (\la v \ra + w, r ( v_\varphi + u )) + \mu^{-}_e (\la v \ra - w, r ( v_\varphi - u ))) \delta w dv \big\} \ , \\
\end{split}
\end{equation}
\begin{equation}
\begin{split}
\tilde{J}_{21} (\delta u) 
& =  2   \Delta^{-1} \big\{ \int_{\mathbb{R}^3} r  M^0_p (\la v \ra , r  v_\varphi  )   \delta u dv  \big\}   \\
& =    \Delta^{-1} \big\{  f_3 (r) \delta u   \big\}  , \\
%& + \lambda^2 \Delta^{-1} \big\{ \int_{\mathbb{R}^3} r (\mu^{+}_p (\la v \ra + w, r ( v_\varphi + u )) + \mu^{-}_p (\la v \ra - w, r ( v_\varphi - u ))) \delta u dv \big\} \ , \\
\end{split}
\end{equation}
\begin{equation}
\begin{split}
\tilde{J}_{22} (\delta w) 
& = \delta w + 2   \Delta^{-1} \big\{ \int_{\mathbb{R}^3} M^0_e (\la v \ra , r v_\varphi ) \delta w dv \big\}  \\
& = \delta w +   \Delta^{-1} \big\{  f_4 (r) \delta w   \big\} . \\
%& + \lambda^2 \Delta^{-1} \big\{ \int_{\mathbb{R}^3}  (\mu^{+}_e (\la v \ra + w, r ( v_\varphi + u )) + \mu^{-}_e (\la v \ra - w, r ( v_\varphi - u ))) \delta w dv \big\} \ . \\
\end{split}
\end{equation}
Here
\begin{equation*}
\begin{split}
& f_1 (r) := 2 \int_{\mathbb{R}^3} r \hat{v}_\varphi M^0_p (\la v \ra , r v_\varphi  )  dv , 
\
f_2 (r) := 2 \int_{\mathbb{R}^3} \hat{v}_\varphi M^0_e (\la v \ra , r v_\varphi  )  dv , \\
& f_3 (r) := 2 \int_{\mathbb{R}^3} r M^0_p (\la v \ra , r v_\varphi  )  dv , \
f_4 (r) := 2 \int_{\mathbb{R}^3} M^0_e (\la v \ra , r v_\varphi  )  dv . \\
\end{split}
\end{equation*}

%We will use the real value version of the following result from \cite{RS1} Volume 1 to analyize the operator $D_{(u, w)} G (u_0, w_0, K_0) $:

%We apply Lemma \ref{lm-RS1} to the operator $D_{(u, w)} G (u_0, w_0, K_0) = D_{(u, w)} G (u_0, w_0, 0) : \mathcal{X} \times \mathcal{X} \times \mathbb{R} \rightarrow H^2(\Omega) \times H^2(\Omega) \times \mathbb{R}$, with the varying parameter $\gamma$ playing the role of $z$ in the lemma. We obtain that there exists a discrete set $T_\gamma \subset [0, +\infty)$ such that for all $\gamma \in [0, +\infty) \setminus T_\gamma$, 

By Assumption \ref{Ass: mu-general-1}, the operator $(D_{(u, w)} G) (u_0, w_0, K_0) = (D_{(u, w)} G) (0, 0, 0)  : \mathcal{X} \times \mathcal{X} \rightarrow H^2(\Omega) \times H^2(\Omega)$ is bounded with a bounded inverse. Moreover, we have 
\begin{equation}
\begin{split}
& \quad (u, w, K) \mapsto G(u, w, K) - (u, w)  = \\
& \left( \begin{array}{ccc}
  -   ( -\Delta + \frac{1}{r^2})^{-1} \big\{ \int_{\mathbb{R}^3} \hat{v}_\varphi (\mu^{K, +} (\la v \ra + w, r ( v_\varphi + u )) - \mu^{K, -} (\la v \ra - w, r ( v_\varphi - u ))) dv \big\} \\
     \Delta^{-1} \big\{ \int_{\mathbb{R}^3}  (\mu^{K, +} (\la v \ra + w, r ( v_\varphi + u )) - \mu^{K, -} (\la v \ra - w, r ( v_\varphi - u ))) dv \big\}  
\end{array} \right)
\begin{array}{ccc}
  . \end{array}
\end{split} 
\end{equation}
is a compact operator from $X \times \mathbb{R}$ to $Z = X$. We can now apply the Global Implicit Function Theorem \ref{globalimplicitfunctiontheorem} and obtain a solution continuum $\mathcal{C}$ for the equation $G(u, w, K) =0$, which is a solution continuum for the system \eqref{system01-nonneutral-arbimass} -- \eqref{system02-nonneutral-arbimass}. One of the three alternatives must hold. The third alternative does not hold since $X \times \mathbb{R} = U$. The proof is complete.

\end{proof}

In the single species case, under additional assumptions, we use the Analytic Global Implicit Function Theorem \ref{analytic-globalimplicitfunctiontheorem} to obtain:

\begin{theorem}  \label{nonneutralglobalthm-general-analytic-1}
Let $\Omega$ be a $C^1$ axisymmetric bounded domain as stated in Section \ref{S:Setup}. Under the assumptions of Theorem \ref{nonneutralglobalthm-general-1}, consider the (strict) single species case, that is, one of the following two cases hold  \\
1) $\mu^{0, +} = \mu^{K, -}  \equiv 0$ for all $K \geq 0$, and $\mu^{K, +}  > 0$ for all $K>0$;   \\
2) $\mu^{0, -} = \mu^{K, +}  \equiv 0$ for all $K \geq 0$, and $\mu^{K, -}  > 0$ for all $K>0$.   \\
%\begin{itemize}
%\item either $\mu^{0, +} = \mu^{K, -}  \equiv 0$, with $\mu^{K, +}  > 0$ for all $K>0$, or $\mu^{K, +}   \equiv 0$) 
%\end{itemize}
%Let $d = \sup_{x \in \Omega} r(x) < +\infty$, $\mu^{K, \pm}$ satisfies that $\mu^{K, \pm}$ is real-analytic with respect to $K$, and $\int_{\mathbb{R}^3}   \mu^{K, \pm} (\la v \ra + a, r(v_\varphi +b)) dv$, $\int_{\mathbb{R}^3} \hat{v}_\varphi   \mu^{K, \pm} (\la v \ra + a, r(v_\varphi +b)) dv$ are real-analytic with respect to $a$ and $b$, i.e. 
Let $\mu^{K, \pm}$ satisfies that $\mu^{K, \pm}$ is real-analytic with respect to $K$, and $\int_{\mathbb{R}^3}   \mu^{K, \pm} (\la v \ra + a, r(v_\varphi +b)) dv$, $\int_{\mathbb{R}^3} \hat{v}_\varphi   \mu^{K, \pm} (\la v \ra + a, r(v_\varphi +b)) dv$ are real-analytic with respect to $a$ and $b$, i.e. 
\begin{equation} \label{Ass:analytic}
\begin{split}
& \big| \int_{\mathbb{R}^3} \partial_1^k \partial_2^l  \mu^{K, \pm} (\la v \ra + a, r(v_\varphi +b)) dv  \big| \leq C^{k+l} \ k! \ l! , \\
& \big| \int_{\mathbb{R}^3} \hat{v}_\varphi \partial_1^k \partial_2^l  \mu^{K, \pm} (\la v \ra + a, r(v_\varphi +b)) dv \big| \leq C^{k+l} \  k! \ l! , \\
\end{split}
\end{equation}
for some $C=C(a, b, K)>0$ and all $k$, $l \in \mathbb{Z}_{\geq 0}$, $a$, $b \in \mathbb{R}$, then the solution set $\mathcal{C}$ contains a locally analytic curve $\widetilde{\mathcal{C}}$ parametrized by $s$, with the following holds: 
%\begin{equation}  \label{E:nonneutralglobalthm-general-1-eq0-1}
% \|( A_\varphi^{K(s), 0} , \phi^{K(s), 0})\|_{H^2 \times H^2} + K(s)  \rightarrow +\infty  \ \text{as} \ s \rightarrow + \infty . 
%\end{equation}
%{\color{red}{
There exists a sequence $s_j \rightarrow +\infty$, such that
\begin{equation}  \label{E:nonneutralglobalthm-general-1-eq0}
 \|( A_\varphi^{K(s_j), 0} , \phi^{K(s_j), 0})\|_{H^2 \times H^2} + K(s_j)  \rightarrow +\infty  \ \text{as} \ j \rightarrow + \infty . 
\end{equation}
%}}
%\begin{equation}  \label{E:nonneutralglobalthm-general-1-eq0}
% \|( A_\varphi^{K(s), 0} , \phi^{K(s), 0})\|_{H^2 \times H^2} + K(s) + \frac{1}{dist ((A_\varphi^{K(s), 0} , \phi^{K(s), 0}), \partial U)} \rightarrow +\infty  \ \text{as} \ s \rightarrow + \infty . 
%\end{equation}
Hence $\widetilde{\mathcal{C}}$ is an unbounded solution curve. 
%Moreover, either the analytic curve $\widetilde{\mathcal{C}}$ is unbounded, or it is a loop connecting the trivial solution $(0, 0)$ back to itself. 
%it connects the trivial solution $(0, 0)$ to a solution with zero electric field (purely magnetic equilibrium). 
\end{theorem}

\begin{remark}
Under the assumptions of Theorem \ref{nonneutralglobalthm-general-analytic-1}, the result above implies that the set $\mathcal{C}$ obtained in Theorem \ref{nonneutralglobalthm-general-1} is unbounded. 
\end{remark}

\begin{remark}
If $\mu^{K, \pm} $ satisfy
\begin{equation}  
% \label{Ass:analytic-1}
\begin{split}
& \partial_e^k \partial_p^l \mu^{K, \pm} (e, p) \leq  \frac{C}{1+|e|^\delta } \ k!  \ l!  \\
\end{split}
\end{equation}
%\begin{equation}  
%% \label{Ass:analytic-1}
%\begin{split}
%& \partial_e^k \partial_p^l \mu^{K, \pm} (e, p) \leq  \frac{C}{1+|e|^\delta+|p|^\delta}  \\
%\end{split}
%\end{equation}
for all $k$, $l \in \mathbb{Z}_{\geq 0}$ and some $C= C(K)>0$, then \eqref{Ass:analytic} holds. 
\end{remark}

%\begin{remark}
%In the case that the locally analytic curve $\widetilde{\mathcal{C}}$ is a loop connecting the trivial solution $(0, 0)$ back to itself, the exact "self-intersection point" is not $(0, 0)$ due to Local Implicit Function Theorem \ref{localimplicitfunctiontheorem}. This self-intersection of solution curve happens at a bifurcation point on $\widetilde{\mathcal{C}}$.
%\end{remark}

%{\color{blue}{Example: $\mu^{K, \pm} $ as well as their derivatives bounded by exponential functions like $\exp (-e^2-p^2)$ ... }}

\begin{proof}
As in the proof of Theorem \ref{nonneutralglobalthm-general-1}, we extend the definition of $\mu^{K, \pm} (e, p)$ evenly with respect to $K$ by taking $\mu^{-K, \pm} (e, p) = \mu^{K, \pm} (e, p)$ for all $K \geq 0$. Following the same argument in the proof for Theorem \ref{nonneutralglobalthm-general-1}, we can apply the Local Implicit Function Theorem \ref{localimplicitfunctiontheorem} and obtain a continuous local solution curve $\mathcal{C}_{loc}$. 

We consider the single species case with $\mu^{K, -} \equiv 0$ and let $I = \mathbb{R}$, $U = \{ (u, w) \in \mathcal{X} \times \mathcal{X} : w > 0 \text{ in } \Omega \}$, then $(u_0, w_0) = (0, 0) \in \partial U$ (here we use the single species assumption). If $\mu^{K, \pm}$ satisfies \eqref{Ass:analytic} for some $C= C(a, b, K)>0$ and all $k$, $l \in \mathbb{Z}_{\geq 0}$ and $a$, $b \in \mathbb{R}$, then 
\begin{equation}  
\begin{split}
& \big|  \partial_a^k \partial_b^l \int_{\mathbb{R}^3}   \mu^{K, \pm} (\la v \ra + a, r(v_\varphi +b)) dv \big| = \big|  r^l  \int_{\mathbb{R}^3} \partial_1^k \partial_2^l  \mu^{K, \pm} (\la v \ra + a, r(v_\varphi +b)) dv \big|  \leq C^{k+l} \ k! \ l! , \\
& \big|  \partial_a^k \partial_b^l  \int_{\mathbb{R}^3} \hat{v}_\varphi   \mu^{K, \pm} (\la v \ra + a, r(v_\varphi +b)) dv \big| = \big| \int_{\mathbb{R}^3} \hat{v}_\varphi \partial_1^k \partial_2^l  \mu^{K, \pm} (\la v \ra + a, r(v_\varphi +b)) dv \big|   \leq C^{k+l} \  k! \ l! . \\
\end{split}
\end{equation}
Hence $G$ is real-analytic in the Fr\'{e}chet sense. Also, the Fr\'{e}chet derivative $D_{(u, w)} G (u, w, K): X \rightarrow Z$ is Fredholm with index $0$ since the operators $( -\Delta + \frac{1}{r^2})^{-1} $ and $\Delta^{-1} $ are compact. Moreover, by Lemma \ref{lm-RS1} we have that $D_{(u, w)} G (u, w, K): X \rightarrow Z$ is invertible for the all the $(u, w)$'s in the local solution curve $\mathcal{C}_{loc}$. We apply the Analytic Global Implicit Function Theorem \ref{analytic-globalimplicitfunctiontheorem} to obtain that the solution set $\mathcal{C}$ contains a locally analytic curve $\widetilde{\mathcal{C}}$. 

%The argument for the other single species case $\mu^{K, +} \equiv 0$ is similar.

We rule out possibility ii) in Theorem \ref{analytic-globalimplicitfunctiontheorem}. Let $\{ s_n \}$ be a sequence such that $s_n \rightarrow +\infty$ such that $\sup_n N(s_n) < +\infty$, then 
\begin{equation}
\begin{split}
& \left( \begin{array}{ccc}
\big\{ \int_{\mathbb{R}^3} \hat{v}_\varphi (\mu^{K, +} (\la v \ra + w, r ( v_\varphi + u )) - \mu^{K, -} (\la v \ra - w, r ( v_\varphi - u ))) dv \big\} \\
 \big\{ \int_{\mathbb{R}^3}  (\mu^{K, +} (\la v \ra + w, r ( v_\varphi + u )) - \mu^{K, -} (\la v \ra - w, r ( v_\varphi - u ))) dv \big\}  
\end{array} \right)
\begin{array}{ccc}
\end{array} 
\end{split}
\end{equation}
must be bounded in $L^2$. Since the operators $( -\Delta + \frac{1}{r^2})^{-1} $ and $\Delta^{-1} $ are compact, and $(u(s_n), w(s_n))$ satisfies the equation $G(u, w, K) =0$ with $G$ described in \eqref{E:nonneutralglobalthm-general-1-eq1}, there must exist a subsequence of $(u(s_n), w(s_n))$ that converges in $X$. Hence possibility ii) in Theorem \ref{analytic-globalimplicitfunctiontheorem} cannot hold, and we must have
\begin{equation*}
N(s) = \|(u(s), w(s))\|_X + K(s) + \frac{1}{dist ((u(s), w(s)), \partial U)} + \frac{1}{dist (K(s), \partial I)} \rightarrow +\infty 
\end{equation*}
as $s \rightarrow +\infty$. It is not possible that $\frac{1}{dist (K(s), \partial I)} \rightarrow +\infty $ since $I = \mathbb{R}$. Hence there must hold
\begin{equation*}
 \|(u(s), w(s))\|_{H^2 \times H^2} + K(s) + \frac{1}{dist ((u(s), w(s)), \partial U)} \rightarrow +\infty 
\end{equation*}
as $s \rightarrow +\infty$. 

If $\|(u(s), w(s))\|_{H^2 \times H^2} + K(s) \rightarrow +\infty$ as $s \rightarrow +\infty$, then the analytic curve $\widetilde{\mathcal{C}}$ is unbounded. On the other hand,  if $\|(u(s), w(s))\|_{H^2 \times H^2} + K(s)$ stays bounded as $s \rightarrow +\infty$. Consider any sequence $\{ s_j \}_{j=1}^\infty$ with $s_j \rightarrow +\infty$ as $j \rightarrow \infty$. By compactness, there exists a subsequence still denoted by $\{ s_j \}_{j=1}^\infty$, such that
\begin{equation}
(u(s_j), w(s_j), K(s_j)) \rightarrow  (u_1, w_1,  K_1) \text{ in } H^2 \times H^2 \times \mathbb{R}  \text{ for some }  (u_1, w_1,  K_1) . 
\end{equation}
Moreover, we have $w_1 \in \partial U$, therefore there exists some $x_0 \in \Omega$, such that $w_1(x_0) = 0$, and $w_1 \geq 0$. By strong maximum principle, we have $w_1 \equiv 0$. Plugging this into \eqref{system01-nonneutral-arbimass} -- \eqref{system02-nonneutral-arbimass}, and noticing that $\mu^{0, +}  = \mu^{K, -} \equiv 0$, $\mu^{K, +} > 0$ for all $K > 0$. Hence we must have $K_1 =0$, and therefore $u_1 \equiv 0$. By basis analysis fact, we have
\begin{equation}
\lim_{s \rightarrow +\infty} (u(s), w(s), K(s)) = (0,  0 ,  0) .    
\end{equation}
By the uniqueness in the Local Implicit Function Theorem \ref{localimplicitfunctiontheorem}, the curve $\widetilde{\mathcal{C}}$ must rejoin itself at some $s \in (0, +\infty)$, which violates Theorem \ref{analytic-globalimplicitfunctiontheorem} 3). A contradiction. Hence $\|(u(s), w(s))\|_{H^2 \times H^2} + K(s)$ cannot stay bounded when $s \rightarrow +\infty$. Combining these arguments together, we conclude that there exists a sequence $s_j \rightarrow +\infty$, such that 
\begin{equation} 
% \label{E:nonneutralglobalthm-general-1-eq0}
 \|( A_\varphi^{K(s_j), 0} , \phi^{K(s_j), 0})\|_{H^2 \times H^2} + K(s_j)  \rightarrow +\infty  \ \text{as} \ j \rightarrow + \infty . 
\end{equation}
holds.

Similar argument as above can be carried out for the other single species case $\mu^{K, +} \equiv 0$. The proof is complete.
\end{proof}
%}}

\section{Properties of the solution set}
\label{S:prop-sol-general}

%In this section, we consider only $K \geq 0$ since the case $K \leq 0$ is essentially the same, see Remark \ref{R:even-in-K}. 

In this section, we consider the case when the set $\mathcal{C}$ obtained in Theorem \ref{nonneutralglobalthm-general-1} is unbounded, and obtain some properties of the solution set. Throughout the section, we assume: 
\begin{assumption} \label{prop-sol-assump}
For all $(e, p) \in \mathbb{R}^2$,
%\begin{assumption} \label{Ass:mu-1}
\begin{equation}  \label{mu-assump-general}
\begin{split}
&   \mu^{K, +} (e, p) +  \mu^{K, -} (e, p) \leq m (K) (1 + |e|^\beta  )^{-1} , \\
&  \int_{\mathbb{R}^3} [ \mu^{K, +} (\la v \ra +a, r(v_\varphi +b)) -  \mu^{K, -} (\la v \ra -a, r(v_\varphi -b))]  dv \geq N(a, b, K) .  \\
\end{split}
\end{equation}
Here $m (K)$ and $N(a, b, K)$ are positive functions, and the constant $\beta>3$. $m (K)$ satisfies $m (0) =0$, $m (K) \rightarrow +\infty$ as $K \rightarrow +\infty $. $N(a, b, K) $ is some function satisfies that when $K$ is large enough, $N(a, b, K) $ is positive and monotonically decreasing in $a$ and $b$. Moreover, we assume that for each $(a, b) \in \mathbb{R}^2$,
%where $ \ell (e, p, K)$ satisfies
\begin{equation} \label{N-assump}
\begin{split}
% \text{For each } (e, p) , \  
%& \ell (e, p, K) \rightarrow +\infty \text{ when } K \rightarrow +\infty , \\
& N (a, b, K) \rightarrow +\infty \text{ when } K \rightarrow +\infty .  \\
\end{split}
\end{equation}
\end{assumption}

We will show the following
% and when $\Omega$ is bounded, and show the following
\begin{proposition} \label{P:prop-sol-general}
Let $\Omega$ be a $C^1$ axisymmetric bounded domain as stated in Section \ref{S:Setup}. Assume that the set $\mathcal{C}$ obtained in Theorem \ref{nonneutralglobalthm-general-1} is unbounded. Assume that Assumption \ref{prop-sol-assump} holds. 
%, and for all $(e, p) \in \mathbb{R}^2$,
%%\begin{assumption} \label{Ass:mu-1}
%\begin{equation}  \label{mu-assump-general}
%\begin{split}
%&   \mu^{K, +} (e, p) +  \mu^{K, -} (e, p) \leq m (K) (1 + |e|^\beta  )^{-1} , \\
%&  \int_{\mathbb{R}^3} [ \mu^{K, +} (\la v \ra +a, r(v_\varphi +b)) -  \mu^{K, -} (\la v \ra -a, r(v_\varphi -b))]  dv \geq N(a, b, K) .  \\
%\end{split}
%\end{equation}
%Here $m (K)$ and $N(a, b, K)$ are positive functions, and the constant $\beta>3$. $m (K)$ satisfies $m (0) =0$, $m (K) \rightarrow +\infty$ as $K \rightarrow +\infty $. $N(a, b, K) $ is some function satisfies that when $K$ is large enough, $N(a, b, K) $ is positive and monotonically decreasing in $a$ and $b$. Moreover, we assume that for each $(a, b) \in \mathbb{R}^2$,
%%where $ \ell (e, p, K)$ satisfies
%\begin{equation} \label{N-assump}
%\begin{split}
%% \text{For each } (e, p) , \  
%%& \ell (e, p, K) \rightarrow +\infty \text{ when } K \rightarrow +\infty , \\
%& N (a, b, K) \rightarrow +\infty \text{ when } K \rightarrow +\infty .  \\
%\end{split}
%\end{equation}
Then we have  \\
1) 
\begin{equation}
\begin{split}
& \sup_{ \{ ( A_\varphi^{K, 0} , \phi^{K, 0}, K) \} \in \mathcal{C} }  \| (\phi^{K, 0}, A^{K, 0}_\varphi)\|_{H^2 (\Omega)} = +\infty .  \\
\end{split}
\end{equation}
%{\color{blue}{
%\begin{equation}
%\begin{split}
%& \sup_{ \{ ( A_\varphi^{K, 0} , \phi^{K, 0}, K) \} \in \mathcal{C} }  \| (\phi^{K, 0}, A^{K, 0}_\varphi)\|_{L^2 (\Omega)} = +\infty .  \\
%\end{split}
%\end{equation}
%Since we have $\| ( A_\varphi^{K, 0} , \phi^{K, 0}) \|_{L^2(\Omega)} \lesssim \| ( A_\varphi^{K, 0} , \phi^{K, 0}) \|_{L^\infty(\Omega)}$ by the boundedness of $\Omega$, we must have
%\begin{equation}
%\begin{split}
%& \sup_{ \{ ( A_\varphi^{K, 0} , \phi^{K, 0}, K) \} \in \mathcal{C} }  \| (\phi^{K, 0}, A^{K, 0}_\varphi)\|_{L^\infty (\Omega)} = +\infty .  \\
%\end{split}
%\end{equation}
%}}
%2) We have
%Assume
%$$ \int_{\mathbb{R}^3} \ell (\la v \ra + a, r(v_\varphi +b)) dv \geq  D (a, b) $$
%for some positive function $D (a, b)$ monotonic in $a$ and $b$. 
%{\color{blue}{
%\begin{equation}
%\begin{split}
%& \sup_{ \{ ( A_\varphi^{K, 0} , \phi^{K, 0}, K) \} \in \mathcal{K} }  \| (\phi^{K, 0}, A^{K, 0}_\varphi)\|_{L^\infty (\Omega)} = +\infty .  \\
%\end{split}
%\end{equation}
%}}
%In particular, in the case that \eqref{E:lambdaunbdd-general} below holds, we have:
2) In particular, in the case that $\sup_{ \{ ( A_\varphi^{K, 0} , \phi^{K, 0}, K) \} \in \mathcal{C} } K = +\infty $ below holds, we have:
%For any $0< p \leq \infty$, as $|K| \rightarrow \infty$,
\begin{equation}
\text{As $K \rightarrow \infty$, } \|(\phi^{K, 0}, A^{K, 0}_\varphi) \|_{L^\infty (\Omega)} \rightarrow +\infty .
\end{equation}
%{\color{red}{Do we need item 2) ? Can $\|(\phi^{K, 0}, A^{K, 0}_\varphi) \|_{L^\infty (\Omega)} \rightarrow +\infty$ happen at other places? Compactness of $\mathcal{C}$? }}  \\
3) If $\mu^{K, -} \equiv 0$, then 
\begin{equation} \label{E:lambdaunbdd-general}
\sup_{ \{ ( A_\varphi^{K, 0} , \phi^{K, 0}, K) \} \in \mathcal{C} } K = +\infty .
\end{equation}
%By continuity, fixing any $\mu^0=0$ or $\gamma =0$, and some $\mu^-$ not necessarily zero, as well as $\mu^+$ and $a^+ (K)$, there exists a positive continuous function $\widetilde{a}^- (K)$, such that if $0 \leq a^- (K)\leq \widetilde{a}^- (K)$, then \eqref{E:lambdaunbdd-general} holds. 
\end{proposition}

%{\color{red}{
%\begin{remark}
%We cannot switch the role of ions and electrons ($+$ and $-$).
%%s if $a^- (K) \rightarrow +\infty \text{ as  } |K| \rightarrow \infty$ and $\frac{a^+(K)}{a^-(K)} = C_K \leq \widetilde{C}_K^0 < 1 \text{ for all } K \text{ large enough}$. 
%\end{remark}
%??
%}}

\begin{remark}
We can switch the role of ions and electrons ($+$ and $-$), and the same results hold. 
%s if $a^- (K) \rightarrow +\infty \text{ as  } |K| \rightarrow \infty$ and $\frac{a^+(K)}{a^-(K)} = C_K \leq \widetilde{C}_K^0 < 1 \text{ for all } K \text{ large enough}$. 
\end{remark}
%Notice that we can switch the role of ions and electrons ($+$ and $-$), and the same results hold. In the case of electrons ($\mu^{K, -} \geq 0$, $\mu^{K, +} = 0$), $\phi^{K, 0}(x) \leq 0$, $e^- = \la v \ra - \phi^{K, 0}(x) \geq \la v \ra$. 

\begin{remark}
In general, without the single species assumption, it is not certain if \eqref{E:lambdaunbdd-general} holds, that is, it is not clear if the parameter $K$ stays bounded in the set $\mathcal{C}$. 
\end{remark}

The proof of Proposition \ref{P:prop-sol-general} is processed by showing the following lemmas below: Lemma \ref{L:field-H2-unbounded-general}, Lemma \ref{L:phiAvarphiunbdd-general}, Lemma \ref{L:field-K-unbounded-general}.

%\begin{lemma} \label{lm:ellipticupperbound-1}
%Let $\Omega$ be a bounded domain in $\mathbb{R}^N$, $ \partial \Omega \in C^2$. Assume $u \in H^2 (\Omega)$ satisfies
%\begin{equation}
%\begin{split}
%& -\Delta u = f(u)  , \ \text{for} \ x \in \Omega, \\ 
%& u = 0 , \ \text{for} \ x \in \partial \Omega ,  \\ 
%\end{split}
%\end{equation}
%and $| f(u) | \lesssim k$ for some constant $k \geq 0$. Then 
%\begin{equation}
%\| u \|_{H^2 (\Omega)} \lesssim \| u \|_{L^2 (\Omega)} + k . 
%\end{equation}
%\end{lemma}

%\begin{proof}
%The proof is similar as the one for Theorem 8.12 in \cite{GT1} so we omit it. Notice that $\Omega$ is bounded so $\| f (u) \|_{L^2} \lesssim k$. 
%\end{proof}

%Denote
%\begin{equation*}
%C_{int, 1} := \int_{\mathbb{R}^3} \frac{1}{ 1+ \la v\ra^\delta } dv , \ C_{int, 2} := \int_{\mathbb{R}^3} \frac{1}{ 1+ \la v\ra^\delta + r^\delta |v_\varphi|^\delta } dv .
%\end{equation*}

%The proof of Proposition \ref{P:prop-sol-general} is given by the following three lemmas:

\begin{lemma}  \label{L:field-H2-unbounded-general}
%Assume that assumption \eqref{mu-assump-general} of Proposition \ref{P:prop-sol-general} holds. Then we have
Assume that Assumption \ref{prop-sol-assump} holds. Then we have
%\begin{equation}  \label{E:field-H2-unbounded-general-eq0-2}
%\begin{split}
%& \sup_{ \{ ( A_\varphi^{K, 0} , \phi^{K, 0}, K) \} \in \mathcal{C} }  \| (\phi^{K, 0}, A^{K, 0}_\varphi)\|_{L^2 (\Omega)} = +\infty . \\
%\end{split}
%\end{equation}
%In particular, 
\begin{equation}  \label{E:field-H2-unbounded-general-eq0-1}
\begin{split}
& \sup_{ \{ ( A_\varphi^{K, 0} , \phi^{K, 0}, K) \} \in \mathcal{C} }  \| (\phi^{K, 0}, A^{K, 0}_\varphi)\|_{H^2 (\Omega)} = +\infty .  \\
\end{split}
\end{equation}
\end{lemma}

\begin{proof}
%Let us show \eqref{E:field-H2-unbounded-general-eq0-1}. 
%Firstly we  
Assume $\sup_{ \{ ( A_\varphi^{K, 0} , \phi^{K, 0}, K) \} \in \mathcal{C} } K = +\infty$. We are going to show 
$$\sup_{ \{ ( A_\varphi^{K, 0} , \phi^{K, 0}, K) \} \in \mathcal{C} }  \| (\phi^{K, 0}, A^{K, 0}_\varphi)\|_{H^2 (\Omega)} = +\infty $$ 
by arguing by contradiction. Assume the contrary, then there must hold
$$\sup_{ \{ ( A_\varphi^{K, 0} , \phi^{K, 0}, K ) \} \in \mathcal{C} }  \| (\phi^{K, 0}, A^{K, 0}_\varphi)\|_{L^\infty (\Omega)} \leq c(K) \leq C_0 $$ 
for some constant $C_0 >0$ independent of $K$. By the assumption \eqref{mu-assump-general} of Proposition \ref{P:prop-sol-general}, we then have
\begin{equation}
\begin{split}
& | \int_{\mathbb{R}^3} [  \mu^{K, +} (\la v \ra +  \phi^{K, 0}(x), r(v_\varphi + A^{K, 0}_\varphi (x))) - \mu^{K, -} (\la v \ra -  \phi^{K, 0}(x), r(v_\varphi - A^{K, 0}_\varphi (x))) ] dv |  \\
%& \leq N(c(K)) \leq N (C_0)  \\
& \geq N (C_0, C_0, K)  \\
\end{split}
\end{equation}
for some positive continuous function $N (a, b, K) $, and $ N (C_0, C_0, K)  $ goes to $+\infty$ as $K \rightarrow +\infty$ for any fixed $C_0$.

%As in \eqref{integralbound-0-eq4}, 

%By Lemma \ref{L:integralbound-0} as well as \eqref{mu-assump-1}, the right hand side of the elliptic equation \eqref{system01-nonneutral-arbimass} for $\phi^{K, 0}$ is bounded from below as
%\begin{equation}  
%\begin{split}
%& \int_{\mathbb{R}^3} [ \mu^{K, +} (\la v \ra +  \phi^{K, 0}(x), r(v_\varphi + A^{K, 0}_\varphi (x))) - \mu^{K, -} (\la v \ra -  \phi^{K, 0}(x), r(v_\varphi - A^{K, 0}_\varphi (x)))] dv \\
%& \geq   2^{-\delta} \gamma C_{int, 2} C_{\mu^+} N(c(K) ) + 2^{-\delta} a^+(K) C_{int, 2} C_{\mu^+} N(c(K))  \\
%& \quad  - 2^{\delta} \gamma C_{int, 1} C_{\mu^-} (1+c(K)^\beta)  - 2^{\delta} a^-(K) C_{int, 1} C_{\mu^-} (1+c(K)^\beta)   \\
%& \geq   2^{-\delta} \gamma C_{int, 2} C_{\mu^+} N (C_0) + 2^{-\delta} a^+(K) C_{int, 2} C_{\mu^+} N (C_0 )   \\
%& \quad  - 2^{\delta} \gamma C_{int, 1} C_{\mu^-} (1+C_0^\beta)  - 2^{\delta} a^-(K) C_{int, 1} C_{\mu^-} (1+C_0^\beta) .  \\
%\end{split}
%\end{equation}
%Notice that by our assumptions on $a^\pm (K)$, the last line in the equation above remains positive for $K$ large enough. 

For large $K$, we apply Lemma \ref{lm:ellipticlowerbound} in Appendix \ref{AppendixC} to \eqref{system01-nonneutral-arbimass} with Dirichlet boundary condition, taking $\alpha (x) \equiv 1 $, $\zeta  \equiv  N (C_0, C_0)  $. 
%$\zeta  \equiv  2^{-\delta} \gamma C_{int, 2} C_{\mu^+} N (C_0 )  + 2^{-\delta} a^+(K) C_{int, 2} C_{\mu^+} N (C_0 )  - 2^{\delta} \gamma C_{int, 1} C_{\mu^-} (1+C_0^\beta)  - 2^{\delta} a^-(K) C_{int, 1} C_{\mu^-} (1+C_0^\beta) >0$. 
We obtain, for all $x \in \Omega$,
\begin{equation}
F(\phi^{K, 0} (x) ) \geq \frac{d_\Omega (x)^2}{6}  ,
\end{equation} 
with
\begin{equation*}
\begin{split}
F(\phi^{K, 0}(x))  
& =  \int^{\phi^{K, 0}(x)}_0  \frac{1}{N (C_0, C_0 ,  K )} ds .   \\
%& =  \int^{\phi^{K, 0}(x)}_0 [ 2^{-\delta} \gamma C_{int, 2} C_{\mu^+} N (C_0 )  + 2^{-\delta} a^+(K) C_{int, 2} C_{\mu^+} N (C_0)   \\
%&  \quad  - 2^{\delta} \gamma C_{int, 1} C_{\mu^-} (1+C_0^\beta)  - 2^{\delta} a^-(K) C_{int, 1} C_{\mu^-} (1+C_0^\beta) ]^{-1} ds  \\
%& =  \phi^{K, 0}(x) [ 2^{-\delta} \gamma C_{int, 2} C_{\mu^+} N (C_0 ) + 2^{-\delta} a^+(K) C_{int, 2} C_{\mu^+} N (C_0 )   \\
%&  \quad  - 2^{\delta} \gamma C_{int, 1} C_{\mu^-} (1+C_0^\beta)  - 2^{\delta} a^-(K) C_{int, 1} C_{\mu^-} (1+C_0^\beta) ]^{-1} .  \\
\end{split}
\end{equation*}
%= \int^{\phi^{\lambda, 0}(x)}_0 \frac{ds}{a^+ (\lambda) \inf_x \int_{\mathbb{R}^3} \mu^+ (\la v \ra +  \phi^{\lambda, 0}(x) ) dv}
%Hence
%\begin{equation}
% \phi^{\lambda, 0}(x)  \frac{1}{ \inf_x \int_{\mathbb{R}^3} \mu^+ (\la v \ra +  \phi^{\lambda, 0}(x) , r(v_\varphi + A^{\lambda, 0}_\varphi (x)) ) dv}\geq  \frac{d_\Omega (x)^2}{6} ,
%\end{equation}
%which is equivalent to
Hence
\begin{equation}  \label{E:field-H2-unbounded-general-eq1}
\begin{split}
%&  \phi^{K, 0}(x) [ 2^{-\delta} \gamma C_{int, 2} C_{\mu^+} N (C_0 )  + 2^{-\delta} a^+(K) C_{int, 2} C_{\mu^+} N (C_0 )   \\
%&    - 2^{\delta} \gamma C_{int, 1} C_{\mu^-} (1+C_0^\beta)  - 2^{\delta} a^-(K) C_{int, 1} C_{\mu^-} (1+C_0^\beta) ]^{-1} \\
& \frac{ \phi^{K, 0}(x) }{ N (C_0, C_0 ,  K ) }    \geq  \frac{  d_\Omega (x)^2}{6} .  \\
\end{split} 
\end{equation}
%Notice that by our assumptions on $a^\pm (K)$, there holds
%\begin{equation}
%\begin{split}
%& 2^{-\delta} \gamma C_{int, 2} C_{\mu^+} N (C_0) + 2^{-\delta} a^+(K) C_{int, 2} C_{\mu^+} N (C_0 )   \\
%&    - 2^{\delta} \gamma C_{int, 1} C_{\mu^-} (1+C_0^\beta)  - 2^{\delta} a^-(K) C_{int, 1} C_{\mu^-} (1+C_0^\beta) \\
%& \sim a^+ (K) , \text{ as } K \rightarrow +\infty . \\
%\end{split}
%\end{equation}
%This together with \eqref{E:field-H2-unbounded-case1-eq1} implies
%\begin{equation*}
%\frac{C_0}{ a^+ (K)   } \gtrsim \frac{ d_\Omega (x)^2}{6} 
%\end{equation*}
%(since $\|(\phi^{K, 0}, A^{K, 0}_\varphi)\|_{L^\infty} \leq C_0 $).

%If $\mu^+$ is non-increasing in $e$, then 
%Note that $\inf_x \int_{\mathbb{R}^3} \mu^+ (\la v \ra +  \phi^{\lambda, 0}(x) ) dv =  I^+(c)$ independent of $\lambda$. 
%$\inf_x \int_{\mathbb{R}^3} \mu^+ (\la v \ra +  \phi^{\lambda, 0}(x) ) dv =  \min \{ \int_{\mathbb{R}^3} \mu^+ (\la v \ra ) dv, \int_{\mathbb{R}^3} \mu^+ (\la v \ra +  \|\phi^{\lambda, 0} \|_{L^\infty} ) dv \}$. 
Let us pick some $x$ with $d_\Omega (x) >0$. Letting $K \rightarrow +\infty$, we arrive at a contradiction to the assumption $\lim_{K \rightarrow \infty}  N(C_0, C_0, K) = +\infty$. Therefore there must hold
\begin{equation}
\sup_{ \{ ( A_\varphi^{K, 0} , \phi^{K, 0}, K) \} \in \mathcal{C} } K < +\infty  
\end{equation}
or 
\begin{equation}
%\frac{c(\lambda)}{(H^+)^{-1} (a^+(\lambda))} \nrightarrow 0 \text{ as } \lambda \rightarrow \infty , \ 
\sup_{ \{ ( A_\varphi^{K, 0} , \phi^{K, 0}, K) \} \in \mathcal{C} }  \| (\phi^{K, 0}, A^{K, 0}_\varphi)\|_{H^2 (\Omega)} = +\infty .
\end{equation}
In the first case, we must have
\begin{equation}
\begin{split}
& \sup_{ \{ ( A_\varphi^{K, 0} , \phi^{K, 0}, K) \} \in \mathcal{C} }  \| (\phi^{K, 0}, A^{K, 0}_\varphi)\|_{H^2 (\Omega)} = +\infty   \\ 
%\  \sup_{ \{ ( A_\varphi^{\lambda, 0} , \phi^{\lambda, 0}, \lambda, M) \} \in \mathcal{C} }  \| \phi^{\lambda, 0}\|_{L^\infty (\Omega)} = +\infty   \\
\end{split}
\end{equation}
since the set $\mathcal{C}$ is unbounded. Therefore in either case \eqref{E:field-H2-unbounded-general-eq0-1} holds. 
%Now, \eqref{E:field-H2-unbounded-general-eq0-2} follows from Lemma \ref{lm:ellipticupperbound-1} and the fact that the right hand side of the elliptic system is uniformly bounded in $L^\infty$.
\end{proof}

%
% there holds 
%\begin{equation}
%\begin{split}
%& \sup_{ \{ ( A_\varphi^{K, 0} , \phi^{K, 0}, K) \} \in \mathcal{C} }  \| (\phi^{K, 0}, A^{K, 0}_\varphi)\|_{H^2 (\Omega)} = +\infty . \\
%\end{split}
%\end{equation}
%\eqref{E:field-H2-unbounded-general-eq0-1} is verified. 

\begin{lemma} \label{L:phiAvarphiunbdd-general}
%Let all the assumptions in Proposition \ref{P:prop-sol-general} and in particular the assumption in Proposition \ref{P:prop-sol-general} 2) hold. 
Let Assumption \ref{prop-sol-assump} as well as the assumption in Proposition \ref{P:prop-sol-general} 2) hold. 
%, and assume$$ \int_{\mathbb{R}^3} \ell (\la v \ra + a, r(v_\varphi +b)) dv \geq  D (a, b) $$
%for some positive function $D (a, b)$.
We have, as $K \rightarrow \infty$,
\begin{equation}
\|(\phi^{K, 0}, A^{K, 0}_\varphi) \|_{L^\infty} \rightarrow +\infty .
\end{equation}
\end{lemma}

%{\color{blue}{
\begin{proof}
%We first prove the result for $p=\infty$. 
Denote
$$ \sup_{ \{ ( A_\varphi^{K, 0} , \phi^{K, 0}, K) \} \in \mathcal{C} }  \|(\phi^{K, 0}, A^{K, 0}_\varphi) \|_{L^\infty} := c_\infty (K) . $$
Let $B_0$ be a maximum inscribed ball of $\Omega$ and $R$ be its radius, then $d_\Omega (x) =  R$. 
%Obviously $R_0$ can be bounded from below by some constant that only depends on $\Omega$. 
Let $B_1$ be the ball which has the same center as $B_0$ and half the radius as it, that is, the radius of $B_1$ is $\tfrac{1}{2} R$. 

In the same way as in Lemma \ref{L:field-H2-unbounded-general}, we have
\begin{equation}   \label{E:phiAvarphiunbdd-general-eq4}
\begin{split}
& \quad \text{The right hand side of \eqref{system01-nonneutral-arbimass}}  \geq  N (c_\infty (K), c_\infty (K), K) . \\
\end{split}
\end{equation}
%\begin{equation}   \label{E:phiAvarphiunbdd-general-eq4}
%\begin{split}
%& \quad \text{The right hand side of \eqref{system01-nonneutral-arbimass}}  \\
%& \geq  2^{-\delta} \gamma C_{int, 2} C_{\mu^+} N(c_\infty(K)) + 2^{-\delta} a^+(K) C_{int, 2} C_{\mu^+} N(c_\infty(K))  \\
%& \quad  - 2^{\delta} \gamma C_{int, 1} C_{\mu^-} (1+|\phi^{K, 0}(x)|^\beta)  - 2^{\delta} a^-(K) C_{int, 1} C_{\mu^-} (1+|\phi^{K, 0}(x)|^\beta)  \\
%%& \geq \inf_x \int_{\mathbb{R}^3} [\mu^{K, +} (\la v \ra +  \phi^{K, 0}(x) , r(v_\varphi + A^{K, 0}_\varphi (x))) - \mu^{K, -} (\la v \ra - \phi^{K, 0}(x) , r(v_\varphi - A^{K, 0}_\varphi (x)))] dv  \\
%& \geq  2^{-\delta} \gamma C_{int, 2} C_{\mu^+} N(c_\infty(K)) + 2^{-\delta} a^+(K) C_{int, 2} C_{\mu^+} N(c_\infty(K)) \\
%& \quad  - 2^{\delta} \gamma C_{int, 1} C_{\mu^-} (1+c_\infty(K)^\beta)  - 2^{\delta} a^-(K) C_{int, 1} C_{\mu^-} (1+c_\infty(K)^\beta)   \\
%& =: J^+ (K) . \\
%\end{split}
%\end{equation}
Assume that $c_\infty (K)$ stays bounded when $K  \rightarrow +\infty$, i.e. $c_\infty (K) \leq C_\infty$ for some constant $C_\infty >0$ independent of $K$, then $ N (c_\infty (K), c_\infty (K), K) \geq N_0$ for some $N_0 >0$ when $K \rightarrow \infty$. 
%Moreover, $J^+ (K) \sim a^+ (K)$ when $|K| \rightarrow \infty$. 

For large $K$, we apply Lemma \ref{lm:ellipticlowerbound} in Appendix \ref{AppendixC} to $- \Delta \phi^{K, 0}(x) \geq \alpha (x) \zeta (u)$ with Dirichlet boundary condition and $\alpha (x) \equiv 1 $, $\zeta  \equiv  N (c_\infty (K), c_\infty (K), K)$. 
%$\zeta  \equiv J^+ (K) =  2^{-\delta} \gamma C_{int, 2} C_{\mu^+} N(c_\infty(K)) + 2^{-\delta} a^+(K) C_{int, 2} C_{\mu^+} N(c_\infty(K)) - 2^{\delta} \gamma C_{int, 1} C_{\mu^-} (1+C_\infty^\beta)  - 2^{\delta} a^-(K) C_{int, 1} C_{\mu^-} (1+C_\infty^\beta) $. 
We obtain that for all $x \in \Omega$,
\begin{equation}
F(\phi^{K, 0} (x) ) \geq \frac{d_\Omega (x)^2}{6}  .
\end{equation} 
Here
\begin{equation*}
\begin{split}
& F(\phi^{K, 0}(x))  =   \int^{\phi^{K, 0}(x)}_0  N_0^{-1} ds  . \\
\end{split}
\end{equation*}
%= \int^{\phi^{\lambda, 0}(x)}_0 \frac{ds}{a^+ (\lambda) \inf_x \int_{\mathbb{R}^3} \mu^+ (\la v \ra +  \phi^{\lambda, 0}(x) ) dv}
Hence
\begin{equation}  
\begin{split}
& \ |\phi^{K, 0}(x)| N_0^{-1}    \geq  \frac{d_\Omega (x)^2}{6} ,  \\
\end{split}
\end{equation} 
%that is, 
%\begin{equation} \label{E:phiAvarphiunbdd-general-eq1}
%\begin{split}
%& \ |\phi^{K, 0}(x)|   \geq  \frac{d_\Omega (x)^2}{6} J^+ (K)   \\
%\end{split}
%\end{equation} 
%holds for all $x \in \Omega$. 
Let us pick some $x \in  B_1$. 
%This gives, for large $|K|$, 
%\begin{equation} \label{E:phiAvarphiunbdd-general-eq5}
%\begin{split}
%&  |\phi^{K, 0}(x)|   \geq  \frac{(R/2)^2}{6} J^+ (K) \sim a^+ (K)  , \ \forall x \in  B_1 ,   \\
%\end{split}
%\end{equation} 
%and in particular, 
%\begin{equation} \label{E:phiAvarphiunbdd-general-eq2}
%C_\infty   \geq c_\infty (K)  \geq J^+(K) \frac{(R/2)^2}{6} \sim a^+ (K)  
%\end{equation} 
%when $|K| \rightarrow \infty$. 
Taking $K \rightarrow +\infty$ yields a contradiction to our assumption. Hence there must hold 
$$ c_\infty (K) \rightarrow +\infty \text{ as }  K \rightarrow +\infty . $$
\end{proof}

\begin{lemma}  \label{L:field-K-unbounded-general}
%Let all the assumptions in Proposition \ref{P:prop-sol-general} and in particular the assumption in Proposition \ref{P:prop-sol-general} 3) hold. Then we have
Let Assumption \ref{prop-sol-assump} as well as the assumption in Proposition \ref{P:prop-sol-general} 3) hold. Then we have
%Assume Assumption \ref{Ass:apm} and \eqref{mu-assump-2}. There must hold
\begin{equation}
\begin{split}
& \sup_{ \{ ( A_\varphi^{K, 0} , \phi^{K, 0}, K) \} \in \mathcal{C} } K = +\infty . \\
\end{split}
\end{equation}
\end{lemma}

\begin{proof}
We have $\mu^{K, +} \geq 0$, $\mu^{K, -} = 0$, hence the solution $\phi^{K, 0}$ to the elliptic system \eqref{system01-nonneutral-arbimass} -- \eqref{system02-nonneutral-arbimass} must be non-negative. Suppose $\sup_{ \{ ( A_\varphi^{K, 0} , \phi^{K, 0}, K) \} \in \mathcal{C} }  |K| < +\infty$. Then there must hold 
$$\sup_{ \{ ( A_\varphi^{K, 0} , \phi^{K, 0}, K) \} \in \mathcal{C} }  \|(\phi^{K, 0}, A^{K, 0}_\varphi)\|_{H^2 (\Omega)} = +\infty.  $$
We denote $\sup_{ \{ ( A_\varphi^{K, 0} , \phi^{K, 0}, K) \} \in \mathcal{C} }  K := \kappa$. 

By the standard elliptic estimates and that $\|g\|_{L^2 (\Omega)} \lesssim \|g \|_{L^\infty (\Omega)}$ (since $\Omega$ is bounded), as well as Lemma \ref{L:integralbound-0} below, we have 
\begin{equation*}
\begin{split}
& \quad \|(\phi^{K, 0}, A^{K, 0}_\varphi)\|_{H^2 (\Omega)}  \\
& \lesssim \| \int_{\mathbb{R}^3}   \mu^{K, +} (\la v \ra + \phi^{K, 0}(x), r(v_\varphi + A^{K, 0}_\varphi (x)))  dv \|_{L^2 (\Omega)} \\
& \lesssim \| \int_{\mathbb{R}^3}   \mu^{K, +} (\la v \ra + \phi^{K, 0}(x), r(v_\varphi + A^{K, 0}_\varphi (x)))  dv \|_{L^\infty (\Omega)} \\
%& \lesssim (1 + a^+( \kappa))  \big \| \int_{\mathbb{R}^3} \big( \frac{1}{1 + |\la v \ra + \phi^{K, 0}(x)|^\delta + |r(v_\varphi + A^{K, 0}_\varphi(x))|^\delta}  \big) dv \big\|_{L^\infty (\Omega)} \\
%& \leq (1 + a^+( \kappa))  \big \| \int_{\mathbb{R}^3} \big( \frac{1}{1 + \la v \ra^\delta + |r(v_\varphi + A^{K, 0}_\varphi(x))|^\delta}  \big) dv \big\|_{L^\infty (\Omega)} \\
& \lesssim  m ( \kappa)  \big \| \int_{\mathbb{R}^3}   \frac{1}{1 + |\la v \ra + \phi^{K, 0}(x)|^\delta  }   dv \big\|_{L^\infty (\Omega)} \\
& \leq  m ( \kappa) \big \| \int_{\mathbb{R}^3}  \frac{1}{1 + \la v \ra^\delta  }    dv \big\|_{L^\infty (\Omega)} \\
& \lesssim m ( \kappa) .  \\
\end{split}
\end{equation*}
%By Sobolev's inequality, we have 
%$$\|(\phi^{K, 0}, A^{K, 0}_\varphi)\|_{H^2 (\Omega)}  \lesssim (1+ a^+ ( \kappa)) (1 + \|(\phi^{K, 0}, A^{K, 0}_\varphi)\|_{H^2}^\epsilon) . $$
Taking the supremum over $\mathcal{C}$ and using the assumption $\sup_{ \{ ( A_\varphi^{K, 0} , \phi^{K, 0}, K) \} \in \mathcal{C} }  K  < +\infty$, we arrive at a contradiction to $\sup_{ \{ ( A_\varphi^{K, 0} , \phi^{K, 0}, K) \} \in \mathcal{C} }  \|(\phi^{K, 0}, A^{K, 0}_\varphi)\|_{H^2 (\Omega)} = +\infty$. Hence there must hold
$$\sup_{ \{ ( A_\varphi^{K, 0} , \phi^{K, 0}, K) \} \in \mathcal{C} }  K = +\infty . $$

\end{proof}

%Notice that we can switch the role of ions and electrons ($+$ and $-$), and the same results hold. In the case of electrons ($\mu^{K, -} \geq 0$, $\mu^{K, +} = 0$), $\phi^{K, 0}(x) \leq 0$, $e^- = \la v \ra - \phi^{K, 0}(x) \geq \la v \ra$. 

\begin{lemma} \label{L:integralbound-0}
For each $x \in \Omega$, $\delta >3$, there holds
%\begin{equation} \label{integralbound-0-eq0}
% \int_{\mathbb{R}^3}  
% \frac{1}{ 1+ |\la v \ra \pm \phi^{ 0} (x) |^\delta + |r( v_\varphi \pm A^0_\varphi (x))|^\delta} dv  \leq  C_\delta (1+ | \phi^{ 0} (x) |) \int_{\mathbb{R}^3} \frac{1}{ 1+ \la v\ra^\delta + r^\delta |v_\varphi|^\delta } dv ,
%\end{equation}
%\begin{equation} \label{integralbound-0-eq1}
% \int_{\mathbb{R}^3}  
% \frac{1}{ 1+ |\la v \ra \pm \phi^{ 0} (x) |^\delta + |r( v_\varphi \pm A^0_\varphi (x))|^\delta} dv  \leq  (2+ 2^\delta | \phi^{ 0} (x) |^\delta + 2^\delta | A^0_\varphi (x)|^\delta ) \int_{\mathbb{R}^3} \frac{1}{ 1+ \la v\ra^\delta} dv , 
%\end{equation}
\begin{equation} \label{integralbound-0-eq1}
 \int_{\mathbb{R}^3}  
 \frac{1}{ 1+ |\la v \ra \pm \phi^{ 0} (x) |^\delta + |r( v_\varphi \pm A^0_\varphi (x))|^\delta} dv  \leq  (2+ 2^\delta | \phi^{ 0} (x) |^\delta  ) \int_{\mathbb{R}^3} \frac{1}{ 1+ \la v\ra^\delta} dv , 
\end{equation}
%\begin{equation} \label{integralbound-0-eq2}
% \int_{\mathbb{R}^3}  
% \frac{1}{ 1+ |\la v \ra \pm \phi^{ 0} (x) |^\delta + |r( v_\varphi \pm A^0_\varphi (x))|^\delta} dv  \geq  \frac{1}{2^\delta + 2^\delta | \phi^{ 0} (x) |^\delta + 2^\delta r_0^\delta | A^0_\varphi (x)|^\delta} \int_{\mathbb{R}^3} \frac{1}{ 1+ \la v\ra^\delta + r^\delta |v_\varphi|^\delta } dv .
%\end{equation}
\begin{equation} \label{integralbound-0-eq2}
\begin{split}
& \int_{\mathbb{R}^3}   \frac{1}{ 1+ |\la v \ra \pm \phi^{ 0} (x) |^\delta + |r( v_\varphi \pm A^0_\varphi (x))|^\delta} dv  \geq   \\
&  \quad  \qquad  \frac{1}{2^\delta + 2^\delta | \phi^{ 0} (x) |^\delta + 2^\delta r^\delta | A^0_\varphi (x)|^\delta } \int_{\mathbb{R}^3} \frac{1}{ 1+ \la v\ra^\delta + r^\delta |v_\varphi|^\delta } dv .  \\
 \end{split}
\end{equation}
Here $C_\delta >0$ is a constant that only depends on $\delta$.
%, $r_0 := \sup_{x\in \Omega} r(x)$. 
\end{lemma}

\begin{proof}
Firstly, let us prove \eqref{integralbound-0-eq1}. For any $x \in \Omega$, the integral $ \int_{\mathbb{R}^3} \frac{1}{ 1+ \la v\ra^\delta} dv$ is convergent because $\delta > 3$.
For any $x$, it suffices to prove 
\begin{equation}
1+ \la v \ra^\delta \leq (2+ 2^\delta | \phi^{0} (x) |^\delta  )  (1+ |\la v\ra \pm \phi^{ 0} (x) |^\delta  ) ,
\end{equation}
which implies
\begin{equation}
1+ \la v \ra^\delta \leq (2+ 2^\delta | \phi^{0} (x) |^\delta  )  (1+ |\la v\ra \pm \phi^{ 0} (x) |^\delta  + |r( v_\varphi \pm A^0_\varphi (x))|^\delta ) . 
\end{equation}
Let $D_1 : = \{ v \in \mathbb{R}^3 :  |\la v\ra \pm \phi^{0} (x) | > \frac{1}{2} \la v \ra  \}$ , $D_2 : =   \{ v \in \mathbb{R}^3 :  |\la v\ra \pm \phi^{0} (x) | \leq  \frac{1}{2} \la v \ra  \} $. We will show that the claim holds true on both sets. Indeed, on $D_1$,
\begin{equation}
1+ \la v \ra^\delta \leq 2 (1+ |\la v\ra \pm \phi^{0} (x) |^\delta )   \ .
\end{equation}
On $D_2$, note that $|\la v\ra \pm \phi^{0} (x) | \leq  \frac{1}{2} \la v \ra$ implies $1 \leq \la v \ra \leq  2 | \phi^{0}  (x) |$. Therefore 
\begin{equation}
1 + \la v \ra^\delta \leq 1+ 2^\delta | \phi^{0}  (x) |^\delta \ .
\end{equation}
\eqref{integralbound-0-eq1} is verified.

Next, we prove \eqref{integralbound-0-eq2}. It suffices to show
\begin{equation}  \label{integralbound-0-eq3}
  1+ |\la v \ra \pm \phi^{ 0} (x) |^\delta + |r( v_\varphi \pm A^0_\varphi (x))|^\delta  \leq  (2^\delta + 2^\delta | \phi^{ 0} (x) |^\delta + 2^\delta r^\delta | A^0_\varphi (x)|^\delta) ( 1+ \la v\ra^\delta + r^\delta |v_\varphi|^\delta )  .
\end{equation}
Indeed, the left hand side of the inequality \eqref{integralbound-0-eq3} satisfies
\begin{equation*} 
\begin{split}
1+ |\la v \ra \pm \phi^{ 0} (x) |^\delta + |r( v_\varphi \pm A^0_\varphi (x))|^\delta  
%& \leq  1+  2^\delta \la v \ra^\delta  +  2^\delta | \phi^{ 0} (x) |^\delta + 2^\delta |rv_\varphi |^\delta  + 2^\delta r^\delta |A^0_\varphi (x)|^\delta  \\
&  \leq  1+  2^\delta \la v \ra^\delta  +  2^\delta | \phi^{ 0} (x) |^\delta + 2^\delta |rv_\varphi |^\delta  + 2^\delta r^\delta |A^0_\varphi (x)|^\delta ,  \\
 \end{split} 
\end{equation*}
and the right hand side of \eqref{integralbound-0-eq3} satisfies
\begin{equation*} 
 (2^\delta + 2^\delta | \phi^{ 0} (x) |^\delta + 2^\delta r^\delta | A^0_\varphi (x)|^\delta) ( 1+ \la v\ra^\delta + r^\delta |v_\varphi|^\delta )  \geq  1+  2^\delta \la v \ra^\delta  +  2^\delta | \phi^{ 0} (x) |^\delta + 2^\delta |rv_\varphi |^\delta  + 2^\delta r^\delta |A^0_\varphi (x))|^\delta .
\end{equation*}
The two estimates above then gives \eqref{integralbound-0-eq3}, and hence \eqref{integralbound-0-eq2} is proved. The proof of the lemma is complete.

\end{proof}

\section{Spectral Stability in a Single Species Case}
\label{S:spectral-stability}

%{\color{red}{Take $\mu^{K, \pm} (e, p) \leq K^m \frac{C_\mu }{1+ |e|^\delta }$ for large $K$ ... }}

%In this section, we consider only $K \geq 0$ since the case $K \leq 0$ is essentially the same, see Remark \ref{R:even-in-K}. 
In this section, we consider a single species plasma and study the spectral stability for the solutions $(\phi^{K, 0}, A_\varphi^{K, 0})$ in the set $\mathcal{C}$. We assume  
%under Assumption \ref{Ass:stability-general} below. 
%Let us first state the following theorem on spectral stability of equilibria in \cite{Z1} (notice that the $\mu^\pm$ in Theorem \ref{mainresult} below is different from the $\mu^\pm$ in all other places in the paper):

\begin{assumption}  \label{Ass:stability-general}
\begin{equation}   \label{mucondition-general}
\begin{split}
&\text{For all } K ,   \quad  \mu^{K, -} \equiv 0 \quad \text{and} \quad \mu^{K, +}_e < 0 .  \\
%& \mu^{K, +}_e < 0 \text{ for all } K .   \\
& \text{When } K \text{ is large enough, the following inequalities hold:}  \\ 
& p \mu^{K, +}_p (e, p) \geq C'_\mu K^{1+m -\epsilon } |p| \la p  \ra^{-\epsilon} \nu (e) ,   \quad  \mu^{K, +} (e, p) \leq K^m \frac{C_\mu }{1+ |e|^\delta +K^\delta |p|^\delta }   ,  \\
& \mu^{K, +}_e (e, p) \leq K^m \frac{C_\mu }{1+ |e|^\delta +K^\delta |p|^\delta }  ,  \quad  \mu^{K, +}_p (e, p) \leq K^{1+m} \frac{C_\mu }{1+ |e|^\delta }  .  \\
\end{split}
\end{equation}
%when $|K| \rightarrow +\infty$. 
Here $\nu (e)$ is some positive function satisfying
\begin{equation}  \label{nucondition-general}
\nu (e) \geq C_\nu \exp (- e) 
\end{equation}
and the constants $\epsilon$, $\delta$, $m$ and $C_\nu$ satisfy $\epsilon>0$, $\delta >4$, $m \in (-1, 1)$, $C_\nu >0$, $ \epsilon <1  - |m|$ (hence $ \epsilon + m <1$, $ \epsilon <1 + m$).
\end{assumption}

Let us first state a theorem on spectral stability of equilibria in \cite{Z1}. Define
%We replace $\mu^\pm$ in the operators by $\mu^{K, \pm }$ and define
\begin{equation}
D^{K, \pm} = v \cdot \nabla_x \pm (\textbf{E}^{ K, 0} + \hat{v} \times \textbf{B}^{K, 0} ) \cdot \nabla_v  \ . 
\end{equation}
Here $\textbf{E}^{K, 0} = - \nabla \phi^{K, 0}$, $\textbf{B}^{K, 0} = \nabla \times \textbf{A}^{ K, 0} $, $\textbf{A}^{K, 0} = A^{K, 0}_\varphi e_\varphi $. 
%where $( \phi^{K, 0}, A^{K, 0}_\varphi)$ is a bounded solution to the equation
%\begin{equation}
%-\Delta \phi^{K, 0} = \int_{\mathbb{R}^3} (\mu^{K, +} (e^{K, +}, p^{K, +}) - \mu^{K, -} (e^{K, -}, p^{K, -})) dv  , 
%\end{equation}
%\begin{equation}
%(-\Delta +\frac{1}{r^2  })  A_\varphi^{K, 0} = \int_{\mathbb{R}^3} \hat{v}_\varphi ( \mu^{K, + } (e^{K, +}, p^{K, +}) - \mu^{K, - } (e^{K, -}, p^{K, -}) ) dv . 
%\end{equation}  
Let $\mathcal{H}^{K, \pm}$ be the space of functions of $x \in \Omega$ and $v \in \mathbb{R}^3$ with the norm
\begin{equation}
\| g (x, v) \|_{\mathcal{H}^{K, \pm}} : = \big( \int_\Omega \int_{\mathbb{R}^3}  | ( \mu^{K, \pm})_e (e^\pm, p^\pm) | g^2 (x,v)   dx dv  \big)^{1/2}
\end{equation}
and $\mathcal{P}^{K, \pm} $ be the orthogonal projection from $\mathcal{H}^{K, \pm} $ onto $\ker D^{K, \pm} $. We define the operators
%In analogy with \eqref{A01definition}, \eqref{A02definition} and \eqref{B0definition}, we define    
\begin{equation}
\mathcal{A}^{0, K}_1 h := \Delta h + \sum_\pm \int_{\mathbb{R}^3} ( \mu^{K, \pm} )_e (1- \mathcal{P}^{K, \pm}) h dv
\end{equation}
\begin{equation}
\mathcal{A}^{0, K}_2 h := (-\Delta + \frac{1}{r^2 }) h - \sum_\pm \int_{\mathbb{R}^3} \hat{v}_\varphi \big( (\mu^{ K, \pm} )_p r   h  + ( \mu^{K, \pm})_e  \mathcal{P}^{K, \pm} (\hat{v}_\varphi h )  \big) dv
\end{equation}
\begin{equation}   \label{B0Kdefinition}
\mathcal{B}^{0, K} h := - \sum_\pm \int_{\mathbb{R}^3}  \hat{v}_\varphi (\mu^{ K, \pm})_e (1-\mathcal{P}^{K, \pm}) h dv
\end{equation}
\begin{equation}
\mathcal{L}^{0, K} := \mathcal{A}^{0, K}_2 - \mathcal{B}^{0, K} (\mathcal{A}^{0, K}_1)^{-1} (\mathcal{B}^{0, K} )^*   . 
\end{equation}
Then we have the following result on spectral stability of equilibria from \cite{Z1}:
\begin{theorem} [\cite{Z1}] \label{mainresult}
Let $\Omega$ be a $C^1$ axisymmetric bounded domain as stated in Section \ref{S:Setup}. 
%Let $\mathcal{H}^{\pm}$ be the space of functions of $x \in \Omega$ and $v \in \mathbb{R}^3$ with the norm
%\begin{equation}
%\| g (x, v) \|_{\mathcal{H}^{ \pm}} : = \big( \int_\Omega \int_{\mathbb{R}^3}  |  \mu^{\pm}_e (e^\pm, p^\pm) | g^2 (x,v)   dx dv  \big)^{1/2} . 
%\end{equation}
%Denote $D^{\pm} = v \cdot \nabla_x \pm (\textbf{E}^{ 0} + \hat{v} \times \textbf{B}^{ 0} ) \cdot \nabla_v $, and $\mathcal{P}^\pm$ be the orthogonal projection on the kernel of $D^\pm$ in the space $\mathcal{H}^\pm $. Define
%\begin{equation}  \label{A01definition}
%\mathcal{A}^0_1 h = \Delta h + \sum_\pm \int_{\mathbb{R}^3} \mu^\pm_e (1- \mathcal{P}^\pm) h dv  ,
%\end{equation}
%\begin{equation}  \label{A02definition}
%\mathcal{A}^0_2 h = (-\Delta + \frac{1}{r^2  }) h - \sum_\pm \int_{\mathbb{R}^3}  \hat{v}_\varphi \big( \mu^\pm_p r   h   + \mu^\pm_e \mathcal{P}^\pm(\hat{v}_\varphi h ) \big) dv   , 
%\end{equation}
%\begin{equation}  \label{B0definition}
%\mathcal{B}^0 h = - \sum_\pm \int_{\mathbb{R}^3} \hat{v}_\varphi \mu^\pm_e (1 - \mathcal{P}^\pm) h dv   .
%\end{equation}
Let $(\mu^{K, \pm},\textbf{E}^{K,0}, \textbf{B}^{K,0})$ be an equilibrium of the relativistic Vlasov-Maxwell system satisfying $ \mu^{K, \pm} (e^{K, \pm}, p^{K, \pm}) \geq 0$ and $\mu^{K, \pm}_e (e, p) < 0$, $   |\mu^{K, \pm}_p| +|\mu^{K, \pm}_e| \leq \frac{C_{\mu} }{1+ |e|^\delta}$ with $ \delta> 3$, $\mu^{K, \pm} \in C^1$, $\phi^{K, 0} \in C (\bar{\Omega})$, $A^{K, 0}_\varphi \in C (\bar{\Omega})$. Then the operator
\begin{equation} \label{L0definition}
\mathcal{L}^{0, K} = \mathcal{A}^{0, K}_2 - \mathcal{B}^{0, K} (\mathcal{A}^{0, K}_1)^{-1} (\mathcal{B}^{0, K})^* 
\end{equation}
on $\mathcal{X}$ is self-adjoint. Also, we have \\
1) If $\mathcal{L}^{0, K} \geq 0$, there exists no growing mode of the linearized RVM equation.  \\
2) Any growing mode, if it exists, must be purely growing, i.e. the exponent $\lambda$ of instability must be a real number.  \\
3) If $\mathcal{L}^{0, K} \ngeq 0 $, there exists a growing mode of the linearized Vlasov equation and the linearized Maxwell system with the boundary conditions. \\
4) The operator $\mathcal{A}^{0, K}_1$ is negative definite, and hence $\mathcal{B}^{0, K} (\mathcal{A}^{0, K}_1)^{-1} (\mathcal{B}^{0, K})^*$ is negative definite. Therefore spectral instability only occurs from the operator $\mathcal{A}^{0, K}_2$. To be precise, $\mathcal{L}^{0, K}$ is nonnegative definite if $\mathcal{A}^{0, K}_2$ is nonnegative definite.   \\
Moreover, the result in the theorem holds for the single species case (either $\mu^{K, +} \equiv 0$ or $\mu^{K, -} \equiv 0$). 
\end{theorem}

Let us first consider the case when $K$ is small. We give the following observation:
%{\color{blue}{
\begin{proposition}  \label{smallK-stable}
%{\color{red}{Taking $\gamma$ sufficiently small}}
Let $\Omega$ be a $C^1$ axisymmetric bounded domain as stated in Section \ref{S:Setup}. Let $\mu^{K, \pm}$ satisfy 
$$ (\mu^{0, \pm})_p \equiv 0 ,  \quad   (\mu^{0, \pm})_e < 0  .  $$ 
then for $K= K_0=0$, the equilibrium $(\mu^{0, \pm}, \phi^{0, 0}, A_\varphi^{0, 0}) = (\mu^{0, \pm}, \phi^{0, 0}, 0)$ is spectrally stable. Hence by continuity we have, when $K$ is close to $0$, the equilibrium $(\mu^{K, \pm}, \phi^{K, 0}, A_\varphi^{K, 0})$ is spectrally stable. Moreover, the result in the theorem holds for the single species case (either $\mu^+ \equiv 0$ or $\mu^- \equiv 0$). 
\end{proposition}

%{\color{blue}{Explain}}

\begin{proof}
%For $K= K_0=0$, the operators $\mathcal{A}^{0, 0}_1$, $\mathcal{A}^{0, 0}_2$, $\mathcal{B}^{0, 0}$ corresponding to the equilibrium $(\mu^{0, \pm}, \phi^{0, 0}, A_\varphi^{0, 0}) = (0, 0, 0)$ are described as
%\begin{equation}
%\mathcal{A}^{0, 0}_1 h = \Delta h + \sum_\pm \int_{\mathbb{R}^3} \gamma  \mu^{0}_e (1- \mathcal{P}^{0, \pm}) h dv , 
%\end{equation}
%\begin{equation}
%\mathcal{A}^{0, 0}_2 h = (-\Delta + \frac{1}{r^2 }) h - \sum_\pm \int_{\mathbb{R}^3} \gamma  \hat{v}_\varphi  \mu^0_e  \mathcal{P}^{0, \pm} (\hat{v}_\varphi h )   dv , 
%\end{equation}
%\begin{equation}   
%\mathcal{B}^{0, 0} h = - \sum_\pm \int_{\mathbb{R}^3}  \gamma \hat{v}_\varphi \mu^0_e (1-\mathcal{P}^{0, \pm}) h dv . 
%\end{equation}
%By Theorem \ref{mainresult}, it suffices for us to show that $\mathcal{A}^{0, 0}_2$ is nonnegative definite. 
By Theorem \ref{mainresult}, it suffices for us to show that $\mathcal{A}^{0, 0}_2$ is positive definite. Let $h \neq 0$. For $K= K_0=0$, the operator $\mathcal{A}^{0, 0}_2$ corresponding to the equilibrium $(\mu^{0, \pm}, \phi^{0, 0}, A_\varphi^{0, 0}) = (\mu^{0, \pm}, \phi^{0, 0}, 0)$ are described as
\begin{equation}
\mathcal{A}^{0, 0}_2 h = (-\Delta + \frac{1}{r^2 }) h - \sum_\pm \int_{\mathbb{R}^3}   \hat{v}_\varphi  (\mu^{0, \pm})_e  \mathcal{P}^{0, \pm} (\hat{v}_\varphi h )   dv . 
\end{equation}
Using the Dirichlet boundary condition, we have
\begin{equation*}  
\begin{split}
\la \mathcal{A}^{0, 0}_2 h, h\ra_{L^2} =
& \int_\Omega ( |\nabla h|^2  +  \frac{1}{r^2  } |h|^2 ) dx + \sum_\pm \|\mathcal{P}^{0, \pm} (\hat{v}_\varphi h) \|^2_{\mathcal{H}^{0, \pm}}  > 0  . 
\end{split}
\end{equation*}
The proof of the proposition is complete. (The proof for the single species case is similar. )
\end{proof}

%From now on in this section, we assume that
%\begin{equation}  \label{Ass:lambdaunbdd-general}
%\sup_{ \{ ( A_\varphi^{K, 0} , \phi^{K, 0}, K) \} \in \mathcal{C} }  K = +\infty , 
%\end{equation}
%and the following assumption holds:

%From now on in this section, we assume 
%\begin{assumption}  \label{Ass:stability-general}
%\begin{equation}   \label{mucondition-general}
%\begin{split}
%& \mu^{K, -} \equiv 0 \text{ for all } K ,   \\
%& \mu^{K, +}_e < 0 \text{ for all } K ,   \\
%& p \mu^{K, +}_p (e, p) \geq C'_\mu K^{1+m -\epsilon } |p| \la p  \ra^{-\epsilon} \nu (e) \text{ when } K \text{ is large enough, }  \\
%& \mu^{K, +} (e, p) \leq K^m \frac{C_\mu }{1+ |e|^\delta +K^\delta |p|^\delta }  \text{ when } K \text{ is large enough, }  \\
%& \mu^{K, +}_e (e, p) \leq K^m \frac{C_\mu }{1+ |e|^\delta +K^\delta |p|^\delta }  \text{ when } K \text{ is large enough, }  \\
%& \mu^{K, +}_p (e, p) \leq K^{1+m} \frac{C_\mu }{1+ |e|^\delta }  \text{ when } K  \text{ is large enough. }  \\
%\end{split}
%\end{equation}
%%when $|K| \rightarrow +\infty$. 
%Here $\nu (e)$ is some positive function satisfying
%\begin{equation}  \label{nucondition-general}
%\nu (e) \geq C_\nu \exp (- e) 
%\end{equation}
%and the constants $\epsilon$, $\delta$, $m$ and $C_\nu$ satisfy $\epsilon>0$, $\delta >4$, $m \in (-1, 1)$, $C_\nu >0$, $ \epsilon <1  - |m|$ (hence $ \epsilon + m <1$, $ \epsilon <1 + m$).
%\end{assumption}

%{\color{red}{Why $\delta >4$ instead of $\delta >3$? }}

\begin{theorem}  \label{unstableexample2-general}
%Take $a^+ (K) \sim K^m$ or $a^- (K) \sim K^m$ for some $m \in (-1, 1)$ when $K \rightarrow +\infty$. 
Let $\Omega$ be a $C^1$ axisymmetric bounded domain as stated in Section \ref{S:Setup}, with $\inf_{x \in \Omega} r(x) > 0$. Recall $d  := \sup_{x \in \Omega} r(x) $ and assume $d>1$. Let all the assumptions in Theorem \ref{nonneutralglobalthm-general-1}, Theorem \ref{nonneutralglobalthm-general-analytic-1}, Proposition \ref{P:prop-sol-general}, in particular Proposition \ref{P:prop-sol-general} 3), hold (so we have $\sup_{ \{ ( A_\varphi^{K, 0} , \phi^{K, 0}, K) \} \in \mathcal{C} }  K = +\infty $). Let Assumption \ref{Ass:stability-general} hold. 
% and \eqref{Ass:lambdaunbdd-general} hold. 
%Let $\epsilon$ be a positive constant such that $ \epsilon <1  - |m|$, so $ \epsilon + m <1$, $ \epsilon <1 + m$. Let $\mu^0$, $\mu^\pm$ satisfy the decay assumptions \eqref{decayassumption}, and be such that 
%\begin{equation}
%p \mu^{K, \pm}_p (e, p) \geq C'_\mu K^{1+m -\epsilon } |p| \la p  \ra^{-\epsilon} \nu (e) 
%\end{equation}
%for some positive function $\nu (e)$, with 
%\begin{equation}  \label{nucondition-general}
%\nu (e) \geq C_\nu \exp (- e) 
%\end{equation}
%for some constant $C_\nu >0$. 
Then when $K$ is large enough, then the equilibrium $( \mu^{K, + } (e, p), \mu^{K, -} (e, p) \equiv 0 ,  \phi^{K, 0}, A_\varphi^{K, 0})$ is spectrally unstable. More precisely, suppose there exists some $h \in \mathcal{X}$ (normalized in such a way that $ \int_\Omega ( |\nabla h|^2 + \frac{1}{r^2 } |h|^2 ) dx  =1 $), and $K$ is so large that 
\begin{equation}  \label{unstableinequality-general}
1 - H_1 C_1 C'_\mu  K^{1+m -\epsilon}  + 120 \cdot  2^{\delta} \pi^2  b^2  (H_1+H_2) C_\mu^2  K^{ 2 m  }  + 2^\delta H_2 C_2  C_\mu K^{m}  +   256 \pi^2 c_P C_\mu^2  H_2 K^{2m}   < 0 \ ,
\end{equation} 
holds, where 
\begin{equation*}
\begin{split}
& H_1 =   \int_{r \geq 1}  r  |h|^2 dx ,   \quad  H_2 = \|h\|^2_{L^2} ,  \\
& C_1  =    2^{-1-\epsilon/2} b^{-\epsilon}  ,   \quad  C_2 =  \frac{8}{3} \pi + \frac{4}{\delta -1} \pi^2 .  \\  
\end{split}
\end{equation*}
%$$H_1 =   \int_{r \geq 1}  r  |h|^2 dx \ ,$$
%$$H_2 = \|h\|^2_{L^2} \ , $$ 
%$$C_1  =    2^{-1-\epsilon/2} b^{-\epsilon}   \ , $$
%$$C_2 =  \frac{8}{3} \pi + \frac{4}{\delta -1} \pi^2 \ ,  $$
and $c_P$ is the square of the Poincar\'{e} constant of $\Omega$, $C_\nu$ is as defined in \eqref{nucondition-general}. 
Then the equilibrium $( \mu^{K, + } (e, p), \mu^{K, -} (e, p) \equiv 0 ,  \phi^{K, 0}, A_\varphi^{K, 0})$ is spectrally unstable.  
\end{theorem}

\begin{remark}
The condition $p \mu^{K, +}_p (e, p) \geq C'_\mu K^{1+m -\epsilon } |p| \la p  \ra^{-\epsilon} \nu (e)  $ together with that $\mu^{K, +}_p (e, p) \leq K^{1+m} \frac{C_\mu }{1+ |e|^\delta } $ implies that $|\nu (e) | $ is controlled by $ \frac{  \la p  \ra^{\epsilon}  }{1+ |e|^\delta} $ with $\delta>  4 $, which ensures the integrals in the proof are convergent. (This is where we need $\delta > 4$. Notice that $\epsilon \leq 1$.)
\end{remark}

%\begin{remark}
%The condition $p \mu^\pm_p \geq C'_\mu |p|  \la p  \ra^{-\epsilon} \nu (e)  $ together with that $|\mu^\pm_p (e, p)| \leq \frac{C_{\mu}}{1+ |e|^\delta}$ (assumed throughout all discussion) implies that $|\nu (e) | \leq \frac{  \la p  \ra^{\epsilon}  }{1+ |e|^\delta} $ with $\delta>  4 $, which ensures the integrals in the proof are convergent.
%\end{remark}
%\textit{Remark} 

\begin{remark}
We can switch the role of ions and electrons ($+$ and $-$), and the same results hold. 
%s if $a^- (K) \rightarrow +\infty \text{ as  } |K| \rightarrow \infty$ and $\frac{a^+(K)}{a^-(K)} = C_K \leq \widetilde{C}_K^0 < 1 \text{ for all } K \text{ large enough}$. 
\end{remark}

\begin{remark}
In fact, we only need $ \mu^{K, -} \equiv 0 $ to hold for $K$ large enough. Having $ \mu^{K, -} \equiv 0 $ in Assumption \ref{Ass:stability-general} is mainly for convenience in description. Moreover, we can extend the result for the case when $ \mu^{K, -} \equiv 0 $ for $K$ large enough does not necessarily hold, with the constraint on the constant $m \in (-1, 1)$ replaced by $m \in (-1, 0)$. The proof is similar.
\end{remark}

%We assume $K >0$ in the following discussion for simplicity.

Let us discuss the proof of Theorem \ref{unstableexample2-general}. In the same way as in the proofs of Lemma 8.2, 8.3 and 8.4 in \cite{Z1}, we can obtain
\begin{lemma}   \label{unstableexampletermestimates-general}
1) We have
\begin{equation}
\| \mathcal{A}^{0, K}_1 \|_{L^2 \rightarrow L^2} \geq c^{-1}_P .
\end{equation}
Hence for all $\tilde{g}  \in \mathcal{X}$, there holds
\begin{equation}
| \la ( \mathcal{A}^{0, K}_1 ) ^{-1} \tilde{g}, \tilde{g}  \ra_{L^2}| \leq c_P \|    \tilde{g} \|^2_{L^2}   .
\end{equation}
Here $c_P$ is the square of the Poincar\'{e} constant of $\Omega$. \\
2) For all $K \geq 1$, $\|(\mathcal{B}^{0, K})^*\|_{L^2 \rightarrow L^2} \leq  8 \sqrt{2} \pi C_\mu K^m  $.  \\
3) For any $h \in \mathcal{X}$, there holds
\begin{equation} 
\begin{split} 
\la (\mathcal{A}^{0, K}_1)^{-1}  (\mathcal{B}^{0, K})^* h,  (\mathcal{B}^{0, K})^* h \ra_{L^2}  
&  \leq  c_P  \|(\mathcal{B}^{0, K})^* h \|^2_{L^2} \\
& \leq 2 c_P  \cdot  128  \pi^2 C_\mu^2  K^{2 m}   \|h\|^2_{L^2} \\
& = 256 \pi^2 c_P C_\mu^2  H_2  K^{2 m}    .
\end{split}
\end{equation}
\end{lemma} 
We also show following 
\begin{lemma}  \label{potentialbound-general}
Assume that $\inf_{x \in \Omega} r(x) > 0$. A solution $(\phi^{K, 0}, A_\varphi^{K, 0})$ to the coupled elliptic system with the settings in this section satisfies
\begin{equation} \label{potentialboundeqn-general}
\|( \phi^{K, 0} , A^{K, 0}_\varphi ) \|_{L^\infty}  \lesssim K^{-1+m} ,
\end{equation} 
where the constant in $\lesssim$ is independent of $K$.
\end{lemma} 

%{\color{red}{Can only do the ion ($+$) single species case? }}

\begin{proof}
%By standard elliptic estimates, the assumptions $\mu^+=0 \text{ as well as } \gamma =0 \text{ or } \mu^0 =0 $ and the decay assumption on $\mu^+$, it suffices to estimate
By standard elliptic estimates, the assumptions $\mu^{K, -}=0$ and the decay assumption on $\mu^{K, +}$, it suffices for us to estimate
\begin{equation} \label{potentialboundeqn2-general}
\int_{\mathbb{R}^3}  
K^m \frac{1}{ 1+ |\la v \ra + \phi^{K, 0} (x) |^\delta + |K r( v_\varphi + A^{K, 0}_\varphi (x))|^\delta} dv  .
\end{equation}
(For $A^{K, 0}_\varphi$, we make use of the identity $-\Delta (ge^{i\varphi}) = (-\Delta + \frac{1}{r^2}) ge^{i\varphi}$, which holds if $g$ is cylindrically symmetric, to get rid of the term $\frac{1}{r^2}A^{K, 0}_\varphi$ in the elliptic equation.)

Using a change of variable $K (v_\varphi +A^{K, 0}_\varphi (x) ) \mapsto v_\varphi$, we compute
\begin{equation*}  
\begin{split}
& \quad  \int_{\mathbb{R}^3}  K^m \frac{1}{ 1+ |\la v \ra + \phi^{K, 0} (x) |^\delta + |K r( v_\varphi + A^{K, 0}_\varphi (x))|^\delta} dv    \\
& = \int_{\mathbb{R}^3}  K^m \frac{1}{ 1+ |\la v \ra + \phi^{K, 0} (x) |^\delta + |K r( v_\varphi + A^{K, 0}_\varphi (x))|^\delta} d\tilde{v} dv_\varphi  \\
& = \int_{\mathbb{R}^3}  K^{-1+m} \frac{1}{ 1+ | \sqrt{1+ |\tilde{v}|^2 +  | \frac{ v_\varphi  }{K} - A^{K, 0}_\varphi (x) |^2} \pm \phi^{K, 0} (x) |^\delta + | r v_\varphi  |^\delta} d\tilde{v} dv_\varphi  \\
& \leq \int_{\mathbb{R}^3}  K^{-1+m} \frac{1}{ 1+ | \sqrt{1+ |\tilde{v}|^2  } + \phi^{K, 0} (x) |^\delta + | r v_\varphi  |^\delta} d\tilde{v} dv_\varphi  .  \\
\end{split}
\end{equation*}
%By Lemma \ref{L:integralbound-0}, 
By the maximum principle, $\phi^{K, 0} (x) \geq 0$. The last line is then bounded from above as 
\begin{equation*}  
\begin{split}
& \quad  \int_{\mathbb{R}^3}  K^{-1+m} \frac{1}{ 1+ | \sqrt{1+ |\tilde{v}|^2  } + \phi^{K, 0} (x) |^\delta + | r v_\varphi  |^\delta} d\tilde{v} dv_\varphi   \\
& \leq  K^{-1+m}     \int_{\mathbb{R}^3} \frac{1}{ 1+ \la \tilde{v} \ra^\delta + | r v_\varphi  |^\delta} dv  \\
& \lesssim  K^{-1+m}     .  \\
\end{split}
\end{equation*}
Here in the last line we have used $\inf_{x \in \Omega} r(x) > 0$. \eqref{potentialboundeqn-general} now follows from standard elliptic estimates.  
\end{proof} 

Notice that we can switch the role of ions and electrons ($+$ and $-$), and the same results hold. In the case of electrons ($\mu^{K, -} \geq 0$, $\mu^{K, +} = 0$), $\phi^{K, 0}(x) \leq 0$, $e^- = \la v \ra - \phi^{K, 0}(x) \geq \la v \ra$.

%We are now ready to prove Proposition \ref{unstableexample2-general}.

We now prove Theorem \ref{unstableexample2-general}.

\begin{proof}
By the assumptions in Proposition \ref{P:prop-sol-general}, in particular the assumption in Proposition \ref{P:prop-sol-general} 3), we have $\sup_{ \{ ( A_\varphi^{K, 0} , \phi^{K, 0}, K) \} \in \mathcal{C} }  K = +\infty $. 

For $h \in \mathcal{X}$, let us first consider the term $\la \mathcal{A}^{0, K}_2 h, h\ra_{L^2} $. For simplicity, we normalize $h$ such that $ \int_\Omega ( |\nabla h|^2 + \frac{1}{r^2 } |h|^2 ) dx  =1 $. Recall that $e^{K, \pm} =  \la v \ra \pm \phi^{K, 0} (x)  , \ p^{K, \pm} = r   (v_\varphi \pm A_\varphi^{K, 0} (x))  $. 
%We observe 
%\begin{equation}
%\begin{split}
%\la \mathcal{A}^{0, K}_2 h, h \ra_{L^2} 
%& = 1 - \sum_\pm \int_\Omega \int_{\mathbb{R}^3} \la v \ra^{-1} K \cdot K^m p^{K, \pm} \mu^\pm_{p} (e^{K, \pm}, Kp^{K, \pm})  |h|^2 dv dx  \\
%&  + \sum_\pm  \int_\Omega \int_{\mathbb{R}^3} K \cdot K^m \frac{ \mu^\pm_{p} (e^{K, \pm}, K p^{K, \pm})}{\la v \ra} r A^{K, 0}_\varphi |h|^2 dv dx + \sum_\pm \| \mathcal{P}^{K, \pm} (\hat{v}_\varphi h)  \|^2_{\mathcal{H}^{K, \pm}}  \\
%& := 1+ I+II+III  .    \\
%\end{split}
%\end{equation}
We write
\begin{equation}
\begin{split}
\la \mathcal{A}^{0, K}_2 h, h \ra_{L^2} 
& = 1 -   \int_\Omega \int_{\mathbb{R}^3} \la v \ra^{-1}   p^{K, +} \mu^{K, +}_{p} (e^{K, +}, p^{K, +})  |h|^2 dv dx  \\
&  +   \int_\Omega \int_{\mathbb{R}^3}  \frac{ \mu^{K, +}_{p} (e^{K, +},  p^{K, +})}{\la v \ra} r A^{K, 0}_\varphi |h|^2 dv dx +   \| \mathcal{P}^{K, +} (\hat{v}_\varphi h)  \|^2_{\mathcal{H}^{K, +}}  \\
& := 1+ I+II+III  .    \\
\end{split}
\end{equation}
We are going to estimate the term $I$, and show that when $K $ is large enough, it remains negative and dominates all the others. We compute 
\begin{equation}
\begin{split}
I 
%& = - \sum_\pm \int_\Omega \int_{\mathbb{R}^3} \la v \ra^{-1} K \cdot K^m p^{K, \pm} \mu^\pm_{p} (e^{K, \pm}, Kp^{K, \pm}) dv \  |h|^2 dx \\
& \leq - C'_\mu   \int_\Omega \int_{\mathbb{R}^3} \la v \ra^{-1} K^{1+m-\epsilon}  |p^{K, +}| \la p^{K, +} \ra^{-\epsilon} \nu (e^{K, +})  dv \ |h|^2 dx \\
& \leq -  C'_\mu  K^{1+m - \epsilon}    \int_\Omega \int_{ | p^{K, +} (x, v) | >1}   \la v \ra^{-1}   | p^{K, +} |^{1-\epsilon}  \cdot 2^{-\epsilon/2}  \nu (e^{K, +}) dv  \  |h|^2 dx \\
& \leq - C'_\mu   K^{1+m -\epsilon }  M    , \\
\end{split}
\end{equation}
where
$$ M = 2^{-\epsilon/2}   \int_{ r(x) \geq 1}   \inf_{x \in \Omega, r(x) \geq 1}  \Big[   \int_{ \{ v \in \mathbb{R}^3,  | p^{K, +} (x, v) | >1 \}} \la v \ra^{-1} |v_\varphi  + A^{K, 0}_\varphi (r, z) |^{1-\epsilon} \nu (e^{K, +})  dv \Big]  r^{1- \epsilon}  |h|^2 dx    \ . $$
Noting that $\| (\phi^{K, 0} ,A^{K, 0}_\varphi ) \|_{L^\infty} \lesssim K^{-1+m}  \leq 1/2$ when $K$ is large, we have
\begin{equation}
\begin{split}
M 
& =  2^{-\epsilon/2}    \int_{ r(x) \geq 1}  \inf_{x \in \Omega, r(x) \geq 1} \Big[  \int_{ \{ v \in \mathbb{R}^3,  | p^{K, +} (x, v) | >1 \}} \la v \ra^{-1} |v_\varphi  + A^{K, 0}_\varphi (r, z) |^{1-\epsilon}  \nu (e^{K, +})  dv \Big]  r^{1- \epsilon}  |h|^2 dx      \\
& \geq  2^{-\epsilon/2}    \int_{ r(x) \geq 1}   \inf_{x \in \Omega, r(x) \geq 1} \Big[  \int_{ \{ v \in \mathbb{R}^3,  |v_\varphi | > 2 + \frac{1}{r(x)} \}} \la v \ra^{-1} (1 +\frac{1}{r(x)})^{1-\epsilon}  \nu (  \la v \ra + \phi^{K, 0} )  dv \Big]   r^{1- \epsilon}  |h|^2 dx     \\
& \geq  2^{-\epsilon/2}   \int_{  r(x) \geq 1}   \inf_{x \in \Omega, r(x) \geq 1} \Big[  \int_{ \{ v \in \mathbb{R}^3,  |v_\varphi | > 3 \}} \la v \ra^{-1}    \nu (  \la v \ra + \phi^{K, 0} )  dv \Big]   r^{1- \epsilon}  |h|^2 dx     \\
& \geq  2^{-\epsilon/2}  C_\nu     \int_{  r(x) \geq 1}  \inf_{x \in \Omega, r(x) \geq 1}  \Big[  \int_{ \{ v \in \mathbb{R}^3,  |v_\varphi | > 3 \}} \la v \ra^{-1}   \exp (- \la v \ra - \phi^{K, 0} )  dv \Big]   r^{1- \epsilon}  |h|^2 dx      \\
& \geq  2^{-\epsilon/2}  C_\nu  \cdot 2 \exp(- \frac{1}{2})  \int_{ r(x) \geq 1}   \int_{ \{ v \in \mathbb{R}^3,  |v_\varphi | > 3 \}} \la v \ra^{-1}  \exp (- \la v \ra  )  dv   r^{1- \epsilon}  |h|^2 dx     \\
& \geq  2^{-\epsilon/2}  C_\nu  \cdot \frac{1}{2}    \int_{  r(x) \geq 1}     r^{1- \epsilon}  |h|^2 dx  \\
& \geq   2^{-\epsilon/2}  \cdot \frac{1}{2} d^{-\epsilon}  C_\nu   \int_{  r(x) \geq 1}     r   |h|^2 dx   , \\
\end{split}
\end{equation}
where $d = \sup_{x \in \Omega} r(x) >1$. From Remark 1, we have $|\nu (e^{K, +}) | \leq \frac{  \la p^{K, +}  \ra^{\epsilon}  }{C(1+ |e^{K, +}|^\delta)} $ with $\delta >  4 $, which ensures that the integrals involved here are convergent. Hence 
$$I \leq    - H_1 C_1 C_\nu C'_\mu K^{1+m -\epsilon}   , $$
where $H_1 =    \int_{r \geq 1}     r   |h|^2 dx $, $C_1 =      2^{-1-\epsilon/2}  b^{-\epsilon}  $. 

%In the end we obtain $I = O (K^{1+m -\epsilon})$. 

%From Remark 1, we have $|\nu (e^{K, +}) | \leq \frac{  \la p^{K, +}  \ra^{\epsilon}  }{C(1+ |e^{K, +}|^\delta)} $ with $\delta >  4 $, which ensures that the integral $ \int_{\mathbb{R}^3} \la v \ra^{-1} v_\varphi^2  \nu (e^{K, +})  dv $ is convergent. 

%{\color{red}{Change the exponent $\gamma$ to $\delta$ below?}}

Let us then consider $II$. $II = 0$ if $A^{K, 0}_\varphi = 0$. If $A^{K, 0}_\varphi \neq 0$, we have
\begin{equation*}
\begin{split}
II 
& \leq 20 \pi b^2 C_\mu  (H_1 + H_2) \| A^{K, 0}_\varphi  \|_{L^\infty}     \sup_{x } ( \int_{\mathbb{R}^3} \la v \ra^{-1} |\mu^{K, +}_{p} (e^{K, +},  p^{K, +})| dv) .  \\
\end{split}
\end{equation*}
Using Lemma \ref{potentialbound-general} and Assumption \ref{Ass:stability-general}, we estimate
\begin{equation}
\begin{split}
II 
%& \leq 20 \pi b^2 C_\mu  (H_1 + H_2)   K^{1+m } K^{-1+m}   \sum_\pm    \sup_{x } ( \int_{\mathbb{R}^3} \la v \ra^{-1} |\mu^\pm_{p} (e^{K, \pm}, K p^{K, \pm})| dv) \\
%& \leq 20 \pi b^2 C_\mu  (H_1 + H_2) \| A^{K, 0}_\varphi  \|_{L^\infty}     \sup_{x } ( \int_{\mathbb{R}^3} \la v \ra^{-1} |\mu^{K, +}_{p} (e^{K, +},  p^{K, +})| dv) \\
& \leq  20 \pi b^2 C_\mu   (H_1 + H_2)    C_\mu K^{-1+m} \cdot  K^{1+m}   \sup_x     \int_{\mathbb{R}^3} \frac{1}{\la v \ra (1+ |e^{K, +}|^\delta)} dv \\
& \leq 20 \pi b^2 C_\mu   (H_1 + H_2)     C_\mu K^{2m  }        \sup_x \int_{\mathbb{R}^3} \frac{1}{\la v \ra (1+ |\la v\ra + \phi^{K, 0} (x)|^\delta)} dv (1+ K^{(-1+m)\delta}) \\
& \leq 40 \cdot 2^{\delta} \pi b^2   (H_1 + H_2)   C_\mu^2  K^{2m  }  \int_{\mathbb{R}^3} \frac{1}{ 1+ \la v\ra^\delta} dv \\
& \leq   40 \cdot  2^{\delta} \pi  b^2   (H_1 + H_2)   C_\mu^2  K^{2m  }  \cdot 3 \pi       \\
& \leq   120  \cdot  2^{\delta} \pi^2  b^2   (H_1 + H_2)  C_\mu^2  K^{2m  }     \\
\end{split}
\end{equation}
%\begin{equation}
%\begin{split}
%II 
%%& \leq 20 \pi b^2 C_\mu  (H_1 + H_2)   K^{1+m } K^{-1+m}   \sum_\pm    \sup_{x } ( \int_{\mathbb{R}^3} \la v \ra^{-1} |\mu^\pm_{p} (e^{K, \pm}, K p^{K, \pm})| dv) \\
%& \leq 20 \pi b^2 C_\mu  (H_1 + H_2) \| A^{K, 0}_\varphi  \|_{L^\infty}   \sum_\pm    \sup_{x } ( \int_{\mathbb{R}^3} \la v \ra^{-1} |\mu^{K, \pm}_{p} (e^{K, \pm},  p^{K, \pm})| dv) \\
%& \leq  20 \pi b^2 C_\mu   (H_1 + H_2)    C_\mu K^{2m-1  }    \sum_\pm  \sup_x     \int_{\mathbb{R}^3} \frac{1}{\la v \ra (1+ |e^{K, \pm}|^\gamma)} dv \\
%& \leq 20 \pi b^2 C_\mu   (H_1 + H_2)     C_\mu K^{2m-1  }    \sum_\pm     \sup_x \int_{\mathbb{R}^3} \frac{1}{\la v \ra (1+ |\la v\ra \pm \phi^{K, 0} (x)|^\gamma)} dv (1+ K^{(-1+m)\gamma}) \\
%& \leq 40 \cdot 2^{\gamma} \pi b^2   (H_1 + H_2)   C_\mu^2  K^{2m-1  }  \int_{\mathbb{R}^3} \frac{1}{ 1+ \la v\ra^\gamma} dv \\
%& \leq   40 \cdot  2^{\gamma} \pi  b^2   (H_1 + H_2)   C_\mu^2  K^{2m-1  }  \cdot 3 \pi       \\
%& \leq   120  \cdot  2^{\gamma} \pi^2  b^2   (H_1 + H_2)  C_\mu^2  K^{2m-1  }   .   \\
%\end{split}
%\end{equation}
where $H_2 =  \int_\Omega  |h|^2 dx =  \|h \|^2_{L^2} $. Therefore the term $II$ is $O(K^{2m})$.

For $III=  \|\mathcal{P}^{K, +}(\hat{v}_\varphi h)\|^2_{\mathcal{H}^{K, +}}$, we use Assumption \ref{Ass:stability-general} to estimate
\begin{equation}
\begin{split}
III = \|\mathcal{P}^{K, +}(\hat{v}_\varphi h)\|^2_{\mathcal{H}^{K, +}} 
%& \leq   K^m \sum_\pm  \int_\Omega \sup_v |\mathcal{P}^{K, \pm} (\hat{v}_\varphi h)|^2 dx  \sup_x (\int_{\mathbb{R}^3} | \mu^\pm_e (e^{K, \pm}, Kp^{K, \pm})| dv)  \\
%& \leq      K^m H_2 \sum_\pm   \sup_x (\int_{\mathbb{R}^3} | \mu^\pm_e (e^{K, \pm}, Kp^{K, \pm})| dv)  \\
& \leq     \int_\Omega \sup_v |\mathcal{P}^{K, +} (\hat{v}_\varphi h)|^2 dx  \sup_x (\int_{\mathbb{R}^3} | \mu^{K, +}_e (e^{K, +}, p^{K, +})| dv)  \\
& \leq      K^m  H_2    \sup_x (\int_{\mathbb{R}^3} \frac{C_{\mu}}{1+ | \la v \ra + \phi^{K, 0} (x) |^\delta} dv)  \\
& \leq 2^\delta   H_2 C_\mu \cdot 2  K^m  \int_{\mathbb{R}^3} \frac{1}{1+ \la v \ra^\delta} dv \\
& \leq  2^\delta  H_2 C_2 C_\mu  K^m .  \\
\end{split}
\end{equation}  
Here $C_2 =  \frac{8}{3} \pi + \frac{4}{\delta -1} \pi^2  $. Therefore $III = O (K^{m})$.

Combining together all the estimates above, as well as Lemma \ref{unstableexampletermestimates-general}, we conclude 
\begin{equation} \label{unstableinequality0-general}
\begin{split}
& \la \mathcal{L}^{0, K} h, h \ra_{L^2} \\
& \leq 1 - H_1 C_1 C_\nu C'_\mu  K^{1+m -\epsilon}  + 120 \cdot  2^{\delta} \pi^2  b^2  (H_1 +H_2) C_\mu^2  K^{2m }  + 2^\delta H_2 C_2  C_\mu K^{m} \\
& +   256 \pi^2 c_P C_\mu^2  H_2 K^{2m}  .  \\
\end{split}
\end{equation}
Hence $( \mu^{K, \pm } ,  \phi^{K, 0}, A_\varphi^{K, 0})$ is spectrally unstable if $C_\mu$, $K$ and $h$ satisfies \eqref{unstableinequality-general}.

\end{proof}

\section{Example}
\label{S:global-cont-sp}

In this section, we present an example of parametrized family of particle density distribution functions, where Assumption \ref{Ass: mu-general-1}, Assumption \ref{prop-sol-assump} and Assumption \ref{Ass:stability-general} are all satisfied, and the theory in Sections \ref{S:global-cont-I}, \ref{S:prop-sol-general}, \ref{S:spectral-stability} can be applied.  

We take some $\mu^\pm (e, p)$, $\mu^{0, \pm} (e, p)$ which are non-negative $C^1$ functions that satisfy the following: 

\begin{assumption} \label{Ass-example-2}
\begin{equation} \label{decayassumption-sp1}
\begin{split}
%\mu^\pm_e (e, p) <  0, \  
& |\mu^0  (e , p )|+ |\mu^0_p (e , p )| +|\mu^0_e (e, p)| \leq \frac{C_\mu }{1+ |e|^\delta  }, \ \delta > 4 , \\ 
& |\mu^\pm  (e , p )|+ |\mu^\pm_p (e , p )| +|\mu^\pm_e (e, p)| \leq \frac{C_\mu }{1+ |e|^\delta }, \ \delta > 4 , \\
\end{split}
\end{equation}
%\eqref{decayassumption-sp1}, 
and $\mu^{0} (e, p)$ is even in $p$. Let $\gamma \geq 0 $ be a temporarily fixed parameter. For any $K \in [0, +\infty)$, consider the equilibrium $\mu^{K, \pm}$ given by
%\begin{equation} \label{E:muK-def-II}
%\begin{split}
%& \mu^{K, +} (e^{K, +},  p^{K, +}) = \gamma \mu^{0} (e^{K, +},  K p^{K, +}) +  a^+(K) \mu^+ (e^{K, +},  K  p^{K, +})  ,   \\
%& \mu^{K, -} (e^{K, -},  p^{K, -}) =  \gamma \mu^{0} (e^{K, -},  K p^{K, -}) +  a^-(K)  \mu^- (e^{K, -},   K  p^{K, -}) .   \\
%\end{split}
%\end{equation}
\begin{equation} \label{E:muK-def-II}
\begin{split}
& \mu^{K, +} (e ,  p) = \gamma \mu^{0} (e ,  K p) +  a^+(K) \mu^+ (e ,  K  p)  ,   \\
& \mu^{K, -} (e ,  p) =  \gamma \mu^{0} (e ,  K p) +  a^-(K)  \mu^- (e ,   K  p) .   \\
\end{split}
\end{equation}
Here $a^\pm (K)$ are nonnegative functions of $K$ which satisfies $a^+ (0) =  a^- (0) =0$, and there exists no $K \neq 0$ such that $a^+ (K) =  a^- (K) =0$. Moreover, we assume that $(a^+)'(0)  \mu^+ (e ,  0) + a^+(0) p \mu^+_p (e ,  0)  = (a^-)'(0)  \mu^- (e ,  0) + a^-(0) p \mu^-_p (e ,  0) = 0$. 
\end{assumption}

\begin{theorem}  \label{nonneutralglobalthm-arbimass-2}
Let $\Omega$ be a $C^1$ axisymmetric bounded domain as stated in Section \ref{S:Setup}. Let $\mu^{K, \pm}$ satisfy Assumption \ref{Ass-example-2}. There exists a discrete subset $S_\gamma$ of $[0, +\infty)$, such that $0 \in S_\gamma$, and for all $\gamma \in [0, +\infty) \setminus S_\gamma$, the following holds: \\
There exists an unbounded continuous solution set or loop $\mathcal{C} := \{ ( A_\varphi^{K, 0} , \phi^{K, 0}, K) \} \subset H^2(\Omega) \times H^2(\Omega) \times \mathbb{R}$ to the system \eqref{system01-nonneutral-arbimass} -- \eqref{system02-nonneutral-arbimass} (which includes the parameter $\gamma$) with $e^{K, \pm} = \la v \ra \pm \phi^{K, 0}(x) , \  p^{K, \pm} =r ( v_\varphi \pm A_\varphi^{K , 0}(x) ) $, in which the solution $(0, 0, 0)$ is included. \\ 
%If $\mu^{0} =0$ or $\gamma =0$, then $\mathcal{C}$ is unbounded.  \\
%{\color{blue}{
Moreover, consider the (strict) single species case when there are only ions in the system ($\gamma \mu^0 (e, p) \equiv 0$, $a^- (K) \mu^-(e, p) \equiv 0$, $\mu^{K, +} > 0$ for all $K>0$) and let $a^+ (K)$ be an analytic function, $\mu^0$ and $\mu^{+}$ satisfy 
%\eqref{Ass:analytic-1}
\begin{equation}   \label{Ass:analytic-1}
\begin{split}
& \big| \partial_e^k \partial_p^l \mu^0 (e, p) \big|\leq  \frac{C}{1+|e|^\delta } \ k!  \ l!  , \\
& \big| \partial_e^k \partial_p^l \mu^+ (e, p)  \big|  \leq  \frac{C}{1+|e|^\delta }  \ k!  \ l!  \\
\end{split}
\end{equation}
for all $k$, $l \in \mathbb{Z}_{\geq 0}$ and some $C>0$, then \eqref{Ass:analytic} holds, and hence the solution set $\mathcal{C}$ contains a locally analytic curve $\widetilde{\mathcal{C}}$. Moreover, $\widetilde{\mathcal{C}}$ is unbounded in the sense that \eqref{E:nonneutralglobalthm-general-1-eq0} holds. 
%Moreover, either the analytic curve $\widetilde{\mathcal{C}}$ is unbounded, or it is a loop connecting the trivial solution $(0, 0)$ back to itself.  
%}}
\end{theorem}

\begin{proof}
By Theorem \ref{nonneutralglobalthm-general-1}, we only need to show that there exists a discrete subset $S_\gamma$ of $[0, +\infty)$, such that for all $\gamma \in [0, +\infty) \setminus S_\gamma$, $\mu^{K, \pm}$ satisfy Assumption \ref{Ass: mu-general-1}. It suffices to verify 4) in Assumption \ref{Ass: mu-general-1}, since all others are straightforward (in particular, $(a^+)'(0)  \mu^+ (e ,  0) + a^+(0) p \mu^+_p (e ,  0)   = 0$ implies $\frac{\partial \mu^{K, +} (e, p)}{\partial K} |_{K=0} = 0$).

Notice that the matrix operator in 4) in Assumption \ref{Ass: mu-general-1} now can be written as
\begin{equation}
% \begin{array}{ccc}
%(D_{(u, w)} G) (u_0, w_0 , K_0)= 
% \end{array}  
\left( \begin{array}{ccc}
\tilde{J}_{11} \ \tilde{J}_{12} \\
\tilde{J}_{21} \ \tilde{J}_{22} \\
\end{array} \right) 
: \mathcal{X} \times \mathcal{X}   \rightarrow H^2(\Omega) \times H^2(\Omega)  
\end{equation}
%$(D_{(u, w)} G (u_0, w_0, K_0) = D_{(u, w)} G (u_0, w_0, 0) : \mathcal{X} \times \mathcal{X} \times \mathbb{R} \rightarrow H^2(\Omega) \times H^2(\Omega) \times \mathbb{R}$, 
where
\begin{equation}
\begin{split}
& \tilde{J}_{11} (\delta u)  = \delta u , \\
& \tilde{J}_{12} (\delta w)  = -   \gamma ( -\Delta + \frac{1}{r^2})^{-1} \big\{   f_2 (r)  \delta w   \big\} ,   \\
& \tilde{J}_{21} (\delta u)  = 0 , \\
& \tilde{J}_{22} (\delta w)  = \delta w +  \gamma \Delta^{-1} \big\{  f_4 (r) \delta w   \big\} ,  \\
\end{split}  
\end{equation}
and
\begin{equation*}
\begin{split}
%& f_1 (r) := 2 \int_{\mathbb{R}^3} r \hat{v}_\varphi \mu^{0}_p (\la v \ra , r0 )  dv , 
\\
& f_2 (r) := 2 \int_{\mathbb{R}^3} \hat{v}_\varphi \mu^{0}_e (\la v \ra , r v_\varphi  )  dv , \\
%& f_3 (r) := 2 \int_{\mathbb{R}^3} r \mu^{0}_p (\la v \ra , r v_\varphi  )  dv , \\
& f_4 (r) := 2 \int_{\mathbb{R}^3} \mu^{0}_e (\la v \ra , 0 )  dv . \\
\end{split}
\end{equation*}
This matrix operator is bounded with a bounded inverse when $\gamma = 0$. We apply Lemma \ref{lm-RS1} to this matrix operator with the varying parameter $\gamma$ playing the role of $z$ in the lemma. We obtain that there exists a discrete set $S_\gamma \subset [0, +\infty)$ such that for all $\gamma \in [0, +\infty) \setminus S_\gamma$, this matrix operator is bounded with a bounded inverse. This completes the proof of the first part of the theorem. 
%The analytic part can be shown in the same way as in Theorem \ref{nonneutralglobalthm-arbimass}.
%$D_{(u, w)} G (u_0, w_0, K_0)  : \mathcal{X} \times \mathcal{X} \rightarrow H^2(\Omega) \times H^2(\Omega)$ is bounded with a bounded inverse.

%Moreover, if $d = \sup_{x \in \Omega} r(x) < +\infty$, $a^\pm (K)$ are analytic functions, $\mu^0$ and $\mu^{\pm}$ satisfy \eqref{Ass:analytic-1} for all $k$, $l \in \mathbb{Z}_{\geq 0}$ and some $C>0$, then
Moreover, in the case when there are only ions in the system, with the function $a^+ (K)$ being analytic, $\mu^0$ and $\mu^{\pm}$ satisfy \eqref{Ass:analytic-1} for all $k$, $l \in \mathbb{Z}_{\geq 0}$ and some $C>0$, then
\begin{equation} 
\begin{split}
& \big| \int_{\mathbb{R}^3} \partial_1^k \partial_2^l  \mu^{0} (\la v \ra + a, r(v_\varphi +b)) dv \big| \leq \int_{\mathbb{R}^3} \frac{C}{1+|\la v \ra + a|^\delta+|r(v_\varphi +b)|^\delta} dv  \leq C_1^{k+l} \ k! \ l! , \\
& \big| \int_{\mathbb{R}^3} \hat{v}_\varphi \partial_1^k \partial_2^l  \mu^{0} (\la v \ra + a, r(v_\varphi +b)) dv \big| \leq \int_{\mathbb{R}^3} \frac{C}{1+|\la v \ra + a|^\delta+|r(v_\varphi +b)|^\delta} dv  \leq C_1^{k+l} \  k! \ l! , \\
& \big|  \int_{\mathbb{R}^3} \partial_1^k \partial_2^l  \mu^{+} (\la v \ra + a, r(v_\varphi +b)) dv  \big|  \leq \int_{\mathbb{R}^3} \frac{C}{1+|\la v \ra + a|^\delta+|r(v_\varphi +b)|^\delta} dv \leq C_2^{k+l} \ k! \ l! , \\
& \big| \int_{\mathbb{R}^3} \hat{v}_\varphi \partial_1^k \partial_2^l  \mu^{+} (\la v \ra + a, r(v_\varphi +b)) dv \big|  \leq \int_{\mathbb{R}^3} \frac{C}{1+|\la v \ra + a|^\delta+|r(v_\varphi +b)|^\delta} dv  \leq C_2^{k+l} \  k! \ l! \\
\end{split}
\end{equation}
%\begin{equation} 
%\begin{split}
%& \big| \int_{\mathbb{R}^3} \partial_1^k \partial_2^l  \mu^{0} (\la v \ra + a, r(v_\varphi +b)) dv \big| \leq \int_{\mathbb{R}^3} \frac{C}{1+|\la v \ra + a|^\delta+|r(v_\varphi +b)|^\delta} dv  \leq C_1^{k+l} \ k! \ l! , \\
%& \big| \int_{\mathbb{R}^3} \hat{v}_\varphi \partial_1^k \partial_2^l  \mu^{0} (\la v \ra + a, r(v_\varphi +b)) dv \big| \leq \int_{\mathbb{R}^3} \frac{C}{1+|\la v \ra + a|^\delta+|r(v_\varphi +b)|^\delta} dv  \leq C_1^{k+l} \  k! \ l! , \\
%& \big|  \int_{\mathbb{R}^3} \partial_1^k \partial_2^l  \mu^{\pm} (\la v \ra + a, r(v_\varphi +b)) dv  \big|  \leq \int_{\mathbb{R}^3} \frac{C}{1+|\la v \ra + a|^\delta+|r(v_\varphi +b)|^\delta} dv \leq C_2^{k+l} \ k! \ l! , \\
%& \big| \int_{\mathbb{R}^3} \hat{v}_\varphi \partial_1^k \partial_2^l  \mu^{\pm} (\la v \ra + a, r(v_\varphi +b)) dv \big|  \leq \int_{\mathbb{R}^3} \frac{C}{1+|\la v \ra + a|^\delta+|r(v_\varphi +b)|^\delta} dv  \leq C_2^{k+l} \  k! \ l! \\
%\end{split}
%\end{equation}
hold for all $k$, $l \in \mathbb{Z}_{\geq 0}$, $a$, $b \in \mathbb{R}$, $C_1 = C_1(a, b) >0$, $C_2 = C_2(a, b) >0$. Taking $C(a, b, K) = \gamma C_1 (a, b) + a_+(K)   C_2(a, b)$, then \eqref{Ass:analytic} holds. Therefore the solution set $\mathcal{C}$ contains a locally analytic curve by Theorem \ref{nonneutralglobalthm-general-analytic-1}. The rest of the proof is the same as the argument in the proof of Theorem \ref{nonneutralglobalthm-general-analytic-1}.
\end{proof}

%Let us consider the single species case and assume furthermore
%%\begin{assumption} \label{Ass:apm}
%\begin{equation}  \label{apm-assump-2}
%%\label{a-diff}
%\begin{split}
%& a^+ (K) \rightarrow +\infty \text{ as  } K \rightarrow \infty ,  \  a^- (K) \equiv 0 ,  \\
%%& C_K = \frac{a^+(K)}{a^-(K)} \geq C_K^0 > 1  \text{ for } K \text{ large enough} ,  \\
%\end{split} 
%\end{equation}
%%with $C_K^0 >1$ being a constant independent of $K$, 
%%\end{assumption}
%and
%\begin{equation}  \label{mu-assump-2}
%\begin{split}
%& \mu^+ (e, p) \geq C_{\mu^+} (1 + |e|^\alpha + |p|^\alpha  )^{-1}  ,   \  \mu^0 (e, p) \equiv 0   \\
%%, \ \mu^- (e, p) \leq C_{\mu^-} (1 + |e|^\beta  )^{-1} , \\
%%& C_{\mu^+} (1 + |e|^\alpha + |p|^\alpha  )^{-1} \leq \mu^0 (e, p) \leq C_{\mu^-} (1 + |e|^\beta  )^{-1}  \\
%\end{split}
%\end{equation}
%for all $(e, p) \in \mathbb{R}^2$ and some constants $\alpha>3$, $C_{\mu^+} >0$. 

Next, let us assume furthermore
%consider the case when there are only ions in the system and
\begin{assumption} \label{Ass-example-2-1}
Assume that there are only ions in the system:
$$\gamma \mu^0 (e, p) \equiv 0 \quad \text{and} \quad  a^- (K) \mu^-(e, p) \equiv 0 \quad (\text{so } \mu^{K,-}(e, p) \equiv 0) $$
and furthermore:  \\
1) 
\begin{equation}  \label{apm-assump-2}
%\label{a-diff}
\begin{split}
& a^+ (K) \sim K^m \quad \text{for some $m \in (0, 1)$ when $K \rightarrow +\infty$,}
%& a^+ (K) \rightarrow +\infty \text{ as  } K \rightarrow \infty ,   \\
%\  a^- (K) \equiv 0 ,  \\
\end{split} 
\end{equation}
%with $C_K^0 >1$ being a constant independent of $K$, 
%\end{assumption}
%and
2) 
\begin{equation}
\mu^+_e < 0 , 
\end{equation}
\begin{equation}  \label{mu-assump-2}
\begin{split}
& \mu^+ (e, p) \geq C_{\mu^+} (1 + |e|^\alpha + |p|^\alpha  )^{-1}    \\
% \  \mu^0 (e, p) \equiv 0   \\
\end{split}
\end{equation}
for all $(e, p) \in \mathbb{R}^2$ and some constants $\alpha>6$, $C_{\mu^+} >0$. Moreover, 
\begin{equation}
p \mu^+_p (e, p) \geq C'_\mu |p| \la p  \ra^{-\epsilon} \nu (e) 
%p \mu^\pm_p (e, p) \geq C'_\mu |p| \la p  \ra^{-\epsilon} \nu (e) 
\end{equation}
for some positive function $\nu (e)$ and positive constant $\epsilon < 1 - |m|$, with 
\begin{equation}  \label{nucondition}
\nu (e) \geq C_\nu \exp (- e) 
\end{equation}
for some constant $C_\nu >0$. 
\end{assumption}

Using Proposition \ref{P:prop-sol-general}, we obtain: 
 
\begin{proposition} \label{P:prop-sol-case2}
Let $\Omega$ be a $C^1$ axisymmetric bounded domain as stated in Section \ref{S:Setup}. Then with Assumption \ref{Ass-example-2} and Assumption \ref{Ass-example-2-1}, the set $\mathcal{C}$ obtained in Theorem \ref{nonneutralglobalthm-arbimass-2} is unbounded. Moreover, there holds:
\begin{equation} \label{E:lambdaunbdd-case2}
\sup_{ \{ ( A_\varphi^{K, 0} , \phi^{K, 0}, K) \} \in \mathcal{C} }  K  = +\infty , 
\end{equation}
and
\begin{equation}
\text{As $K  \rightarrow \infty$, } \|(\phi^{K, 0}, A^{K, 0}_\varphi) \|_{L^\infty} \rightarrow +\infty .
\end{equation}
\end{proposition}

\begin{proof}
Applying Theorem \ref{nonneutralglobalthm-arbimass-2}, we conclude that the set $\mathcal{C}$ obtained in Theorem \ref{nonneutralglobalthm-arbimass-2} is unbounded. It then suffices to verify Assumption 4.1, and the only non-trivial one is \eqref{mu-assump-general}. Indeed, by Lemma \ref{L:integralbound-0}, we have
\begin{equation}  
\begin{split}
& \quad  \int_{\mathbb{R}^3} [ \mu^{K, +} (\la v \ra +a, r(v_\varphi +b)) -  \mu^{K, -} (\la v \ra -a, r(v_\varphi -b)) ] dv  \\
& =  \int_{\mathbb{R}^3}   (a^+(K) \mu^+) (\la v \ra +a, r(v_\varphi +b))   \\
& \geq   2^{-\delta} a^+(K) C_{int, 2} C_{\mu^+}  (1+|a|^\alpha + d^\alpha |b|^\alpha)^{-1}  \\
%& \geq   2^{-\delta} \gamma C_{int, 2} C_{\mu^+} (1+|a|^\alpha + d^\alpha |b|^\alpha)^{-1} + 2^{-\delta} a^+(K) C_{int, 2} C_{\mu^+}  (1+|a|^\alpha + d^\alpha |b|^\alpha)^{-1}  \\
%& \quad  - 2^{\delta} \gamma C_{int, 1} C_{\mu^-} (1+|a|^\beta)  - 2^{\delta} a^-(K) C_{int, 1} C_{\mu^-} (1+|a|^\beta)   \\
& =: N(a, b, K) .  \\
\end{split}
\end{equation}
Then when $K$ is large enough, $N(a, b, K) $ is positive and monotonically decreasing in $a$ and $b$. Moreover, for each $(e, p) \in \mathbb{R}^2$,
%where $ \ell (e, p, K)$ satisfies
\begin{equation}
\begin{split}
& N (e, p, K) \rightarrow +\infty \text{ when } K \rightarrow +\infty .  \\
\end{split}
\end{equation}
\eqref{mu-assump-general} is verified. Therefore we can apply Proposition \ref{P:prop-sol-general} and obtain the desired results.
\end{proof}

%\begin{proof}
%The proof is very similar to the one for Proposition \ref{P:prop-sol-case1} so we omit it.
%\end{proof}
 
We also consider the spectral stability of the solutions $(\phi^{K, 0}, A_\varphi^{K, 0})$. We first have the following observation for small $K$:

%\begin{proposition}
%Let $\Omega$ be a $C^1$ axisymmetric bounded domain as stated in Section \ref{S:Setup}. Let $ \mu^{0}_p \equiv 0 ,  \quad   \mu^{0}_e < 0 $, then for $K= K_0=0$, the equilibrium $(\mu^{0, \pm}, \phi^{0, 0}, A_\varphi^{0, 0}) = (\mu^{0, \pm}, \phi^{0, 0}, 0)$ is spectrally stable. Hence by continuity we have, when $K$ is close to $0$ (depends on $\gamma$, $a^\pm$, $\mu^0$, $\mu^\pm$), the equilibrium $(\mu^{K, \pm}, \phi^{K, 0}, A_\varphi^{K, 0})$ is spectrally stable. 
%\end{proposition}

\begin{proposition}
Let $\Omega$ be a $C^1$ axisymmetric bounded domain as stated in Section \ref{S:Setup}. When $K$ is close to $0$ (depends on $a^+$, $\mu^+$), the equilibrium $( \mu^{K, + } (e, p) ,   \mu^{K, - } (e, p)  \equiv 0 ,   \phi^{K, 0}, A_\varphi^{K, 0})$ is spectrally stable. 
\end{proposition}

\begin{proof}
The conclusion directly follows from Proposition \ref{smallK-stable}.
\end{proof}

%Now, let us assume furthermore the following conditions and discuss the spectral stability of the solutions $(\phi^{K, 0}, A_\varphi^{K, 0})$ for large $K$: 
%
%\begin{assumption}  \label{Ass:stability}
%%$$  \mu^+_e < 0, \   \mu^-=0 \text{ as well as } \gamma \mu^0 =0 $$
%$$  \mu^+_e < 0  $$
%and
%$$a^+ (K) \sim K^m  $$ 
%%$$a^+ (K) \sim K^m \text{ or } a^- (K) \sim K^m $$ 
%for some $m \in (-1, 1)$ when $K \rightarrow +\infty$.
%\end{assumption}
%%$$\mu^+=0 \text{ as well as } \gamma =0 \text{ or } \mu^0 =0 $$ 
%Hence we have only ions in the system. 

For large $K$, we have

\begin{theorem} \label{unstableexample2}
%Take $a^+ (K) \sim K^m$ or $a^- (K) \sim K^m$ for some $m \in (-1, 1)$ when $K \rightarrow +\infty$. 
%Assume that \eqref{E:prop-sol-case2-K-eq1} holds. 
Let $\Omega$ be a $C^1$ axisymmetric bounded domain as stated in Section \ref{S:Setup}, with $\inf_{x \in \Omega} r(x) > 0$. Let Assumption \ref{Ass-example-2} and Assumption \ref{Ass-example-2-1} hold. 
%$\mu^{K, +}$ satisfy \eqref{E:muK-def-II} and the decay assumption \eqref{decayassumption-sp1}, with $\delta > 3$ replaced by $\delta >4$. Assume that \eqref{apm-assump-2}, \eqref{mu-assump-2} and Assumption \ref{Ass:stability} hold. 
Recall $d : = \sup_{x \in \Omega} r(x) $ and assume $d>1$. Let $\epsilon$ be a positive constant such that $ \epsilon <1  - |m|$, so $ \epsilon + m <1$, $ \epsilon <1 + m$. 
%Let $\mu^+$ be such that 
%\begin{equation}
%p \mu^+_p (e, p) \geq C'_\mu |p| \la p  \ra^{-\epsilon} \nu (e) 
%%p \mu^\pm_p (e, p) \geq C'_\mu |p| \la p  \ra^{-\epsilon} \nu (e) 
%\end{equation}
%for some positive function $\nu (e)$, with 
%\begin{equation}  \label{nucondition}
%\nu (e) \geq C_\nu \exp (- e) 
%\end{equation}
%for some constant $C_\nu >0$. 
Then if $K$ is large enough, then the equilibrium $( \mu^{K, + } (e, p) ,   \mu^{K, - } (e, p)  \equiv 0 ,   \phi^{K, 0}, A_\varphi^{K, 0})$ is spectrally unstable. 
\end{theorem}

%\begin{remark}
%The condition $p \mu^+_p \geq C'_\mu |p|  \la p  \ra^{-\epsilon} \nu (e)  $ together with that $|\mu^+_p (e, p)| \leq \frac{C_{\mu}}{1+ |e|^\gamma}$ (assumed throughout all discussion) implies that $|\nu (e) | \leq \frac{  \la p  \ra^{\epsilon}  }{1+ |e|^\gamma} $ with $\gamma >  4 $, which ensures the integrals in the proof are convergent. (This is where we need $\delta > 4$. Notice that $\epsilon \leq 1$.)
%\end{remark}

\begin{proof}
It is easy to verify that the assumptions in Theorem \ref{unstableexample2-general}, that is, Assumption \ref{Ass:stability-general}, hold. In particular, $p \mu^+_p (e, p) \geq C'_\mu |p| \la p  \ra^{-\epsilon} \nu (e) $ together with that $ \mu^{K, +} (e^{K, +},  p^{K, +}) = \gamma \mu^{0} (e^{K, +},  K p^{K, +}) +  a^+(K) \mu^+ (e^{K, +},  K  p^{K, +}) \sim K^m \mu^+ (e^{K, +},  K  p^{K, +})$ for large $K$ imply $p \mu^{K, +}_p (e, p) \geq C'_\mu K^{1+m -\epsilon } |p| \la p  \ra^{-\epsilon} \nu (e)$ for large $K$. Therefore we can directly apply Theorem \ref{unstableexample2-general} and obtain Theorem \ref{unstableexample2}.
\end{proof}

\begin{remark}
There does exists selections of $a^+(K)$ and $\mu^+ (e, p)$ such that both Assumption \ref{Ass-example-2} and Assumption \ref{Ass-example-2-1} hold. For instance, we can take $a^+(K)$ to be a $C^\infty$ function such that $a^+(0)= (a^+)'(0) = 0$ with $a^+(K) \sim K^{1/4} $ for large $K$, and $\mu^+ (e, p) = \frac{(|p|+1)^{3/2}}{1+|e|^6}$. 
\end{remark}

\begin{remark}
We can switch the role of ions and electrons ($+$ and $-$), and the same results in this section hold. 
\end{remark}

\appendix

\section{Derivative Operators in the Cylindrical Coordinates}
\label{AppendixA}

For the readers' convenience, we provide the information of the derivatives under the cylindrical coordinates. Using $(x_1, x_2, x_3)$ to denote the Cartesian coordinates and $(r, \varphi, z)$ to be the cylindrical coordinates, we have
$$  
x_1 = r \cos \varphi, \qquad x_2 = r \sin \varphi, \qquad  x_3 = z 
$$
and therefore
$$
e_r = (\cos \varphi, \sin \varphi , 0), \qquad e_\varphi = (- \sin \varphi , \cos \varphi, 0), \qquad e_z =(0, 0, 1) .
$$
Hence for any function $f(r, \varphi, z)$ and any vector field $ \textbf{A} (r, \varphi, z)$, we have
$$
\nabla_x f = \frac{\partial f}{\partial r} e_r + \frac{1}{r} \frac{\partial f}{\partial \varphi} e_\varphi + \frac{\partial f }{\partial z } e_z  ,
$$
$$
\Delta_x f = \frac{1}{r} \frac{\partial}{\partial r} ( r \frac{\partial f}{\partial r}) + \frac{1}{r^2} \frac{\partial^2 f}{\partial \varphi^2} + \frac{\partial^2 f}{\partial z^2}    ,
$$
$$
\nabla_x \cdot \textbf{A} = \frac{1}{r} \frac{\partial (r A_r)}{\partial r}  + \frac{1}{r} \frac{\partial A_\varphi}{\partial \varphi}  + \frac{\partial A_z }{\partial z }   ,
$$
$$
\nabla_x \times \textbf{A} = (\frac{1}{r} \frac{\partial A_z}{\partial \varphi} - \frac{\partial A_\varphi}{\partial z} ) e_r +( \frac{\partial A_r}{\partial z} - \frac{\partial A_z}{\partial r} ) e_\varphi + \frac{1}{r} ( \frac{\partial (r A_\varphi) }{\partial r } - \frac{\partial A_r}{\partial \varphi} ) e_z  .
$$

\section{Key Invariants}
\label{AppendixB}

In this section, we prove the invariance of $e^\pm (x,v) = \la v \ra \pm \phi^0 (r, z)  $ and $ p^\pm (x,v) = r  ( v_\varphi  \pm A^0_\varphi (r, z))  $ (or $ p^\pm (x,v) = r  ( v_\varphi  \pm A^0_\varphi (r, z) \pm A_{\varphi, ext} (r, z) )  $ when an external magnetic potential $\textbf{A}_{ext} = A_{\varphi, ext} (r, z) e_\varphi$ is present) along the particle trajectories given in Section \ref{S:Setup}:
\begin{equation}
\begin{split}
& \dot{X}^\pm = \hat{V}^\pm,  \  \dot{V}^\pm =  \pm \textbf{E}^0 (X^\pm) \pm \hat{V}^\pm \times \textbf{B}^0 (X^\pm) .  \\
%& \text{Or }  \dot{X}^\pm = \hat{V}^\pm,  \  \dot{V}^\pm =  \pm \textbf{E}^0 (X^\pm) \pm \hat{V}^\pm \times \textbf{B}^0 (X^\pm)  \pm \hat{V}^\pm \times \textbf{B}_{ext} (X^\pm) \text{ when the external magnetic potential is present.}
\end{split}
\end{equation}
(Or
\begin{equation}
\dot{X}^\pm = \hat{V}^\pm,  \  \dot{V}^\pm =  \pm \textbf{E}^0 (X^\pm) \pm \hat{V}^\pm \times \textbf{B}^0 (X^\pm)  \pm \hat{V}^\pm \times \textbf{B}_{ext} (X^\pm) 
\end{equation}
when the external magnetic potential is present.)

We only compute the $+$ case for simplicity. The $-$ case is similar. In fact, along the particle trajectories, we have 
\begin{equation}
\begin{split}
\dot{e}
&  = \hat{V} \cdot \dot{V} + \dot{X} \cdot \nabla \phi^0 \\
& =   \hat{V} \cdot (\textbf{E}^0 + \hat{V} \times \textbf{B}^0 )  + \hat{V} \cdot \nabla \phi^0  \\
& =     \hat{V} \cdot (-\nabla \phi^0 + \hat{V}  \times \textbf{B}^0 )  + \hat{V}\cdot \nabla \phi^0 \\
& =0   ,
\end{split}
\end{equation}
\begin{equation}
\begin{split}
\dot{p} 
& = \dot{X} \cdot \nabla_X p + \dot{V} \cdot \nabla_V p \\
& = \dot{X} \cdot \nabla_X ( r  ( v_\varphi + A^0_\varphi (r, z)) ) + \dot{V} \cdot \nabla_V ( r  ( v_\varphi  + A^0_\varphi (r, z)) )      \\
& =  \dot{X} \cdot \nabla_X (r v_\varphi) +\dot{ X} \cdot \nabla_X (r A^0_\varphi) + \dot{V} \cdot r  \nabla_V  v_\varphi       \\
& = \dot{X} \cdot \nabla_X (r v_\varphi) + \dot{X} \cdot A^0_\varphi e_r + \dot{ X} \cdot r \frac{\partial A^0_\varphi}{\partial r} e_r + \dot{X} \cdot r \frac{\partial A^0_\varphi}{\partial z} e_z + r ( \textbf{E}^0 +  \hat{V} \times \textbf{B}^0 ) \cdot e_\varphi    \\
& = \dot{X} \cdot \nabla_X (r v_\varphi) + \dot{X} \cdot A^0_\varphi e_r + \dot{ X} \cdot r \frac{\partial A^0_\varphi}{\partial r} e_r + \dot{X} \cdot r \frac{\partial A^0_\varphi}{\partial z} e_z \\
& + r ( -\nabla\phi^0 + \dot{X} \times \frac{\partial A^0_\varphi}{\partial z} e_r + \dot{X} \times \frac{1}{r} A^0_\varphi e_z +\dot{X} \times \frac{\partial A^0_\varphi}{\partial r} e_z    ) \cdot e_\varphi      \\
& = \dot{X} \cdot \nabla_X (r v_\varphi) + \dot{X} \cdot A^0_\varphi e_r + \dot{ X} \cdot r \frac{\partial A^0_\varphi}{\partial r} e_r + \dot{X} \cdot r \frac{\partial A^0_\varphi}{\partial z} e_z - r  \frac{\partial A^0_\varphi}{\partial z} \dot{X} \cdot  e_z - A^0_\varphi   \dot{X} \cdot e_r - r \frac{\partial A^0_\varphi}{\partial r}  \dot{X} \cdot e_r      \\
& =\dot{X} \cdot \nabla_X (r v_\varphi)   . \\
\end{split}
\end{equation}
(Here $r$, $\varphi$, $z$, $v_r$, $v_\varphi$, $v_z$ are components of $X$ and $V$.) In the computation above, we used that $\nabla \phi^0$ does not have $e_\varphi$-component. We evaluate $\dot{X} \cdot \nabla_X (r v_\varphi) $ using the Cartesian coordinates:
\begin{equation}
\begin{split}
\dot{X} \cdot \nabla_X (r v_\varphi) 
& = \dot{X} \cdot \nabla_X ( r V \cdot ( -\frac{X_2}{r} , \frac{X_1}{r} , 0 )) \\
& = \dot{X} \cdot \nabla_X (- X_2 V_1 + X_1 V_2) \\
& = \frac{1}{\la V \ra} (V_1, V_2, V_3 ) \cdot (V_2, -V_1, 0)  . \\
& = 0
\end{split}
\end{equation}
Hence $\dot{p} =0$. Moreover, one can show in a similar way that $ p (x,v) = r  ( v_\varphi  + A^0_\varphi (r, z) +A_{\varphi, ext} (r, z) )  $ is invariant along the particle trajectories satisfying the ODE $\dot{X}  = \hat{V}$,  $ \dot{V}  =  \textbf{E}^0 (X ) + \hat{V} \times \textbf{B}^0 (X )  + \hat{V} \times \textbf{B}_{ext} (X) $. The invariance along the particle trajectories is now verified.

\section{Key Instrumental Theorems and Lemmas}
\label{AppendixC}

%In this section, we state some important tool theorem ans lemmas, which are key ingredients in this paper. Firstly we present the following local and global implicit function theorems. Readers can refer to \cite{Kie}, \cite{CWW1}, \cite{CWW2}, etc. for detailed discussions of these results.

In this appendix, we state some important instrumental theorems and lemmas, which are applied in this paper as key tools. Firstly we present the following local and global implicit function theorems. Readers can refer to \cite{Kie}, \cite{CWW1}, \cite{CWW2}, etc. for detailed discussions of these results.

%{\textcolor{red} {Add reference}}

\begin{theorem} \label{localimplicitfunctiontheorem}
\textbf{(Local Implicit Function Theorem)}
Let $X$, $Y$, $Z$ be (real) Banach spaces. Consider a mapping $F:U \times V \rightarrow Z$ with open sets $U \subset X$, $V \subset Y$, and the equation
\begin{equation} \label{implicitfunctiontheoremequation}
F(x, y) =0 \ . 
\end{equation}
Suppose the equation \eqref{implicitfunctiontheoremequation} admits a solution $(x_0, y_0) \in U \times V$ such that the Fr\'{e}chet derivative of $F$ with respect to $x$ at $(x_0, y_0)$ is bijective, i.e.
\begin{equation}
F(x_0, y_0) =0 \ ,
\end{equation}
\begin{equation}
D_x F(x_0, y_0): X \rightarrow Z \ is \ bounded \ (continuous) \ with \ a \ bounded \ inverse.
\end{equation}
Assume also that $F$ and $D_x F$ are continuous:
\begin{equation}
F \in C (U \times V, Z) \ , 
\end{equation}
\begin{equation}
D_x F \in C (U \times V, L(X,Z)) \ . 
\end{equation}
Here $L(X, Z)$ denotes the Banach space of bounded linear operators from $X$ to $Z$ endowed with the operator norm. Then there exists a neighbourhood $U_1 \times V_1$ in $U \times V$ of $(x_0, y_0)$ and a mapping $f: V_1 \rightarrow U_1 \subset X$, such that 
\begin{equation}
f(y_0) = x_0 ,
\end{equation}
\begin{equation}
F(f(y), y) =0 \ for \ all \ y \in V_1 .
\end{equation}
Furthermore, $f$ is continuous on $V_1$:
\begin{equation}
f \in C(V_1, X)  .
\end{equation}
Finally, every solution of \eqref{implicitfunctiontheoremequation} in $U_1 \times V_1$ is of the form $(f(y), y)$.
\end{theorem}

\begin{theorem} \label{globalimplicitfunctiontheorem}
\textbf{(Global Implicit Function Theorem)}
Let $X$ and $Z$ be (real) Banach spaces and $U$ be an open set in $X \times \mathbb{R}$. Consider a mapping $F:U   \rightarrow Z$ that is $C^1$ in the Fr\'{e}chet sense, and the equation
\begin{equation}  \label{globalimplicitfunctiontheoremequation}
F(x, y) =0 
\end{equation}
with $x \in X$, $y \in \mathbb{R}$. Suppose the equation \eqref{globalimplicitfunctiontheoremequation} admits a solution $(x_0, y_0) \in U $ such that the Fr\'{e}chet derivative of $F$ with respect to $x$ at $(x_0, y_0)$ is bijective, i.e.
\begin{equation}
F(x_0, y_0) =0  ,
\end{equation}
\begin{equation}
D_x F(x_0, y_0): X \rightarrow Z \ is \ bounded \ (continuous) \ with \ a \ bounded \ inverse.
\end{equation}
Assume also that the mapping $ (x, y) \rightarrow F(x, y) - x $ is compact from $U$ to $X$. Let $\mathcal{S}$ be the closure in $X \times \mathbb{R}$ of the solutions set $\{(x, y) : F(x, y) =0 \}$. Let $\mathcal{C}$ be the connected component of $\mathcal{S}$ to which $(x_0, y_0)$ belongs. Then one of the following three alternatives holds: \\
i) $\mathcal{C} $ is unbounded in $X \times \mathbb{R}$; \\
ii) $\mathcal{C}\setminus \{(x_0, y_0)\}$ is connected; \\
iii) $\mathcal{C} \cap \partial U $ is not empty. 
\end{theorem}

\begin{theorem} \label{analytic-globalimplicitfunctiontheorem}
\textbf{(Analytic Global Implicit Function Theorem)}
Let $X$ and $Z$ be (real) Banach spaces, $I$ be an open interval (possibly unbounded) with $0 \in \overline{I}$, and $U$ be an open set in $X$ with $0 \in \partial U$. Consider a mapping $F: U \times I \rightarrow Z$ that is real-analytic in the Fr\'{e}chet sense, and the equation
\begin{equation}  \label{analytic-globalimplicitfunctiontheoremequation}
F(x, y) =0 
\end{equation}
with $x \in X$, $y \in I$. Assume that for any $(x, y) \in U \times I$ with $F(x, y) =0$, the Fr\'{e}chet derivative $D_x F (x, y): X \rightarrow Z$ is Fredholm with index $0$. \\
Suppose that there exists a continuous curve $\mathcal{C}_{loc}$ of solution to  \eqref{analytic-globalimplicitfunctiontheoremequation}, parametrized as
\begin{equation*}
\mathcal{C}_{loc} := \{ ( \tilde{x}(y), y) : 0 < y < y_* \} \subset F^{-1} (0) 
\end{equation*}
for some $y_* >0$ and continuous $\tilde{x} : (0, y_*) \rightarrow U$. If
\begin{equation*}
\lim_{y \rightarrow 0^+} \tilde{x} (y) = 0 \in \partial U , \quad  D_x F (\tilde{x} (y) , y) : X \rightarrow Z \text{ is  invertible for all } y , 
\end{equation*}
then $\mathcal{C}_{loc} $ is contained in a curve of solutions $\mathcal{C}$, parametrized as
\begin{equation*}
\mathcal{C} := \{ ( x(s), y(s)) : 0 < s < +\infty\} \subset F^{-1} (0) 
\end{equation*}
for some continuous function $ s \mapsto ( x(s), y(s)) \in U \times I$, $s \in (0, +\infty)$, with the following properties: \\
1) One of the following alternatives holds: \\
i) As $s \rightarrow +\infty$, 
\begin{equation*}
N(s) := \|x(s)\|_X + y(s) + \frac{1}{dist (x(s), \partial U)} + \frac{1}{dist (y(s), \partial I)} \rightarrow +\infty . 
\end{equation*}
ii) There exists a sequence $s_n \rightarrow +\infty$ such that $\sup_n N(s_n) < +\infty$ but $\{ x(s_n) \}$ has no subsequences converging in $X$; \\
2) Near each point $(x(s_0), y(s_0) ) \in \mathcal{C}$, we can reparametrize $\mathcal{C}$ so that $s \mapsto (x(s), y(s))$ is real-analytic;  \\
3) $(x(s), y(s)) \notin \mathcal{C}_{loc}$ for $s$ sufficiently large. 
\end{theorem}

We also introduce the following result from \cite{RS1} Volume 1, which we need it for its real value version:
\begin{lemma} \label{lm-RS1}
\textbf{(\cite{RS1} Volume 1, Theorem VI.14, analytic Fredholm theorem)}
Let $D$ be an open connected subset of $\mathbb{C}$. Let $f: D \rightarrow L(X, Z)$ be an analytic operator-valued function such that $f(z)$ is compact for each $z \in D$. Then one of the following two alternatives holds: \\
i) $(I - f(z))^{-1}$ exists for no $z \in D$; \\
ii) $(I - f(z))^{-1}$ exists for all $z \in D\setminus S $, where $S$ is a discrete subset of $D$. In this case, $(I - f(z))^{-1}$ is meromorphic in $D$, analytic in $D\setminus S $, the residues at he pole are finite rank operators, and for $z \in S $ the equation $f(z) g =g$ has a non-zero solution in $X$.   
\end{lemma}

Moreover, we introduce the following maximum-principle-type result on elliptic equations.

%Before proving the proposition, let us introduce the following tool lemmas:

%We first state the following lemma from \cite{ATG1}:
\begin{lemma} \label{lm:ellipticlowerbound}
%1) 
Let $\Omega$ be a $C^1$ bounded domain in $\mathbb{R}^N$ ($N \leq 4$). Assume $u \in H^2 (\Omega)$ satisfies 
%{\color{red}{Assume $u \in C^2 (\Omega)$}} satisfies
\begin{equation}
\begin{split}
& -\Delta u \geq \alpha (x) \zeta (u) , \ \text{for} \ x \in \Omega, \\ 
& u \geq 0 , \ \text{for} \ x \in \partial \Omega, \\ 
\end{split}
\end{equation}
where $\alpha$ is a non-negative function and $\zeta: [0, a_\zeta) \rightarrow [0, +\infty)$ $(0 < a_\zeta \leq +\infty)$ is a non-decreasing $C^1$ function such that $\zeta (0) > 0$. Let $F(t):= \int^t_0 \frac{ds}{\zeta (s)}$, $d_\Omega (x) := dist (x, \partial \Omega)$. Then  
\begin{equation}
F(u(x)) \geq  \alpha_x (d_\Omega (x)) \frac{d_\Omega (x)^2}{2N}  \ \text{for all} \ x \in \Omega  \text{ (a.e.) } ,
\end{equation}
and equivalently (since $F(t)$ is a non-decreasing function)
\begin{equation}
u(x) \geq F^{-1} (\alpha_x (d_\Omega (x)) \frac{d_\Omega (x)^2}{2N} ) \ \text{for all} \ x \in \Omega \text{ (a.e.) }  .
\end{equation}
Here $\alpha_x (r) := \inf_{y \in B_r (x)} \alpha (y) $, $B_r (x) = \{ y : |y-x| < r \}$.  
%\\
%2) Let $\Omega$ be a bounded domain in $\mathbb{R}^N$. Assume $u \in (C^2 (\Omega))^d$ is a vector-valued function, which satisfies
%\begin{equation}
%\begin{split}
%& -\Delta u \geq \alpha  \zeta   , \ \text{for} \ x \in \Omega, \\ 
%& u \geq 0 , \ \text{for} \ x \in \partial \Omega, \\ 
%\end{split}
%\end{equation}
%where $\alpha \geq 0$ and $\zeta > 0$ are a non-negative real numbers. Let $F(t):= \int^t_0 \frac{ds}{\zeta } = \frac{t}{\zeta}$, $d_\Omega (x) := dist (x, \partial \Omega)$. Then  
%\begin{equation}
%F(u_j (x)) \geq  \alpha \frac{d_\Omega (x)^2}{2N}  \ \text{for all} \ x \in \Omega , \ j = 1, 2, \cdots , d , 
%\end{equation}
%and equivalently, 
%\begin{equation}
%u_j (x) \geq F^{-1} (\alpha  \frac{d_\Omega (x)^2}{2N} ) \ \text{for all} \ x \in \Omega ,  \ j = 1, 2, \cdots , d . 
%\end{equation}
\end{lemma}

%{\color{red}{Do we need item 2) ? }}

\begin{proof}
The proof can be adapted from Theorem 2.2 in \cite{AT1}. 
% by approximating $H^2$ functions using $C^2$ functions. 
%The proof of 1) is given in Theorem 2.1 in \cite{ATG1}. 2) can be shown in a similar way.
Let $\psi_\alpha$ be the unique solution of the equation
\begin{equation*}
\begin{split}
& -\Delta u = \alpha (x) , \ \text{for} \ x \in \Omega, \\ 
& u = 0 , \ \text{for} \ x \in \partial \Omega, \\ 
\end{split}
\end{equation*}
where $\alpha$ is a non-negative function. Denote
\begin{equation*}
\alpha_x (r) := \inf_{y \in B_r (x)} \alpha (y) \quad ( 0 < r \leq d_\Omega (x) = dist (x, \partial \Omega) ) 
\end{equation*}
Using the product rule for the divergence of product of a scalar valued function and a vector field, $F'> 0$ and $F'' \leq 0$ we compute
\begin{equation*}
\begin{split}
\Delta F(u) 
& = div (\nabla F(u)) = div (F'(u) \nabla u)  \\
& = \nabla (F'(u)) \cdot \nabla u + F'(u) div (\nabla u)  \\
& = F''(u) \nabla u \cdot \nabla u + F'(u) \Delta u  \\
& = F''(u) | \nabla u|^2 + F'(u) \Delta u \leq F'(u) \Delta u .   \\
\end{split}
\end{equation*}
Hence $-\Delta F(u) \geq F'(u) (-\Delta u)$. Moreover, 
\begin{equation*}
\begin{split}
D_i D_j F(u) 
& = D_i (F'(u) D_j u) =  F''(u) D_i u D_j u + F'(u) D_i D_j u . \\
\end{split}
\end{equation*}
Notice that since $\zeta (s)$ is in $C^1$, $F(t)= \int^t_0 \frac{ds}{\zeta (s)}$ is in $C^2$. Hence $F'(u) D_i D_j u  \in L^2 (\Omega)$. Also, $D_i u$, $D_j u \in L^4 (\Omega)$ by Sobolev embedding, so $F''(u) D_i u D_j u \in L^2 (\Omega)$ by H\"older's inequality. Thus, $F(u) \in H^2 (\Omega)$. 

Now, from $-\Delta u \geq \alpha (x) \zeta (u)$, we have
\begin{equation*}
\begin{split}
-\Delta F(u) \geq F'(u) (-\Delta u) = \frac{1}{\zeta(u)} (-\Delta u) \geq \alpha (x) = -\Delta \psi_\alpha . 
\end{split}
\end{equation*}
Since we have $F(u) = \psi_\alpha = 0$ on $\partial \Omega$, by the maximum principle (see, for example, \cite{GT1}, Theorem 8.1), we obtain $F(u(x)) \geq \psi_\alpha (x)$ in $\Omega$. 

To estimate $\psi_\alpha$, we notice that for $y \in B_{d_\Omega (x)} (x)$, 
\begin{equation*}
-\Delta \psi_\alpha (y) = \alpha (y) \geq \alpha_x (d_\Omega (x)) . 
\end{equation*}
Consider the function
\begin{equation*}
w(y) = \frac{d_\Omega (x)^2 - |x-y|^2}{2N}.  
\end{equation*}
This function satisfies $-\Delta w = 1$ in $B_{d_\Omega (x)} (x)$ and $w =0$ on $\partial B_{d_\Omega (x)} (x)$. Hence
\begin{equation*}
-\Delta \psi_\alpha (y) \geq  -\Delta (\alpha_x (d_\Omega (x)) w (y)) . 
\end{equation*}
By maximum principle, 
\begin{equation*}
\psi_\alpha (y) \geq   \alpha_x (d_\Omega (x)) w (y) . 
\end{equation*}
Combining the results above, we obtain
\begin{equation*}
F(u(x)) \geq  \alpha_x (d_\Omega (x)) \frac{d_\Omega (x)^2-|x-y|^2}{2N}  \ \text{if} \ |y-x| < d_\Omega (x) . 
\end{equation*}
Taking $x=y$, we have
\begin{equation*}
F(u(x)) \geq  \alpha_x (d_\Omega (x)) \frac{d_\Omega (x)^2}{2N}  \ \text{for all} \ x \in \Omega \text{ (a.e.) } . 
\end{equation*}
\end{proof}

\end{document}